\documentclass[11pt]{amsart}
\usepackage[a4paper,margin=1.1in]{geometry}
\usepackage[T1]{fontenc}
\usepackage{amsmath,amssymb,amsthm,mathtools}
\usepackage[hidelinks]{hyperref}
\hypersetup{pdftitle={The sixth moment of a degree two L-function of arbitrary level},pdfauthor={Andrew Pearce-Crump}}
\usepackage{microtype}
\allowdisplaybreaks
\theoremstyle{plain}
\newtheorem{theorem}{Theorem}[section]
\newtheorem{proposition}[theorem]{Proposition}
\newtheorem{lemma}[theorem]{Lemma}
\newtheorem{corollary}[theorem]{Corollary}
\theoremstyle{remark}
\newtheorem{remark}[theorem]{Remark}
\newcommand{\Real}{\operatorname{Re}}
\newcommand{\Imag}{\operatorname{Im}}
\newcommand{\half}{\tfrac12}
\newcommand{\R}{\mathbb R}
\newcommand{\Z}{\mathbb Z}
\newcommand{\Q}{\mathbb Q}
\newcommand{\dd}{\,d}
\newcommand{\GC}{\Gamma_{\mathbb C}}
\newcommand{\Lc}{\mathcal L}
\newcommand{\Kc}{D}
\newcommand{\Jc}{\mathcal J}
\newcommand{\Fc}{\mathcal F}
\newcommand{\Bc}{\mathcal B}
\newcommand{\e}{\operatorname{e}}

\title[The sixth moment of a degree two $L$-function]{The sixth moment of a degree two $L$-function of arbitrary level}
\author{Andrew Pearce-Crump}
\address{School of Mathematics, University of Bristol, Fry Building, Woodland Road, Bristol, BS8 1UG, United Kingdom}
\email{andrew.pearce-crump@bristol.ac.uk}
\date{}

\begin{document}

\begin{abstract}
Let $f$ be a primitive holomorphic cusp form of arbitrary weight, level and nebentypus. We prove $\int_0^T|L(\half+it,f)|^6\dd t\ll_{f,\epsilon}T^{2+\epsilon}$, extending Jutila's level-one theorem. The new ingredient is a large sieve comparing the stationary phases in the Booker--Milinovich--Ng transformation at different heights. Applications include moments of order $T^{1+\epsilon}$ on every line to the right of the critical line, $\Omega(T^{4/35-\epsilon})$ simple zeros, zero density estimates, and bounds for cubic-field power sums, shifted convolution sums and the general divisor problem.
\end{abstract}

\maketitle

\section{Introduction and main results}

Let $f\in S_k(\Gamma_0(N),\xi)$ be a primitive holomorphic cusp form of weight $k$, level $N$ and nebentypus $\xi$, with Fourier coefficients $\lambda_f(n)$ normalised so that Deligne's bound reads $|\lambda_f(n)|\le d(n)$. With this normalisation $L(s,f)=\sum_{n\ge1}\lambda_f(n)n^{-s}$ has its critical line at $\Real s=\half$.

This paper concerns the sixth moment of $L(s,f)$ on the critical line. Our main result bounds it at every level.

\begin{theorem}\label{thm:main}
For every primitive holomorphic cusp form $f$ of arbitrary weight, level and nebentypus, and every $\epsilon>0$,
\[
  \int_0^T\bigl|L(\half+it,f)\bigr|^6\dd t\ \ll_{f,\epsilon}\ T^{2+\epsilon} .
\]
\end{theorem}

The sixth moment bound gives the Weyl-type bound $|L(\half+it,f)|\ll|t|^{1/3+\epsilon}$ through the usual local mean inequality, and is stronger on average than the pointwise bound alone. The Lindel\"of hypothesis would give the smaller sixth moment $T^{1+\epsilon}$. H\"older's inequality against the mean square
\[
  \int_0^T|L(\half+it,f)|^2\dd t\ll_fT(\log T)^3,
\]
which follows from the argument of Theorem \ref{thm:second} (see also M\"uller \cite{Muller} for automorphic forms on general Fuchsian groups), gives the fourth moment as well.

\begin{corollary}\label{cor:fourth}
For every primitive holomorphic cusp form $f$ of arbitrary weight, level and nebentypus, and every $\epsilon>0$,
\[
  \int_0^T\bigl|L(\half+it,f)\bigr|^4\dd t\ \ll_{f,\epsilon}\ T^{3/2+\epsilon} .
\]
\end{corollary}

\subsection{Background}\label{sec:background}
For $\zeta(s)^2$, which also has degree two, the analogue of Theorem~\ref{thm:main} is Heath-Brown's twelfth moment estimate \cite{HeathBrown12}
\[
  \int_0^T\bigl|\zeta(\half+it)\bigr|^{12}\dd t\ll T^2(\log T)^{17}.
\]
Meurman \cite{Meurman84} extended it to Dirichlet $L$-functions, and Durkan, Paloj\"arvi and Trudgian \cite{DPT} have reduced Heath-Brown's bound to $T^2(\log T)^{15}$. For level one, Theorem~\ref{thm:main} is Jutila's Theorem 4.7 \cite{Jutila}. The corresponding pointwise Weyl bound was proved by Good \cite{Good} and then by Jutila using Voronoi summation and Farey fractions.

Booker, Milinovich and Ng \cite{BMN} extended the pointwise argument to arbitrary level; Aggarwal \cite{Aggarwal} gave a second proof. Their transformation formula works at every rational point, including denominators sharing factors with a non-squarefree level.

The pointwise bound and the mean square give only a sixth moment of order $T^{7/3+\epsilon}$, since
\[
  \int_0^T\bigl|L(\half+it,f)\bigr|^6\dd t\ \le\ \max_{0\le t\le T}\bigl|L(\half+it,f)\bigr|^4\int_0^T\bigl|L(\half+it,f)\bigr|^2\dd t\ \ll\ T^{4/3+\epsilon}\cdot T^{1+\epsilon}.
\]
Theorem \ref{thm:main} requires the mean-value argument as well. At level greater than one this means comparing the transformed sums of \cite{BMN} at different heights. This comparison is the new part of the proof, and Section~\ref{sec:outline} describes it.

A special case was already known. Sankaranarayanan \cite[(5.6)]{Sankaranarayanan} proved a sixth moment of order $T^2(\log T)^{24}$ for ideal-class zeta functions of imaginary quadratic fields, and hence for the weight-one dihedral forms obtained from class-group characters. 

Many natural families consist of forms of level greater than one, to which Jutila's theorem does not apply. They include the newforms attached to elliptic curves over $\Q$, which have weight two and level equal to the conductor \cite{BCDT}; the CM forms, whose $L$-functions are the Hecke $L$-functions of Gr\"ossencharacters of imaginary quadratic fields, and for which Weyl-type pointwise bounds go back to S\"ohne \cite{Sohne}; the weight-one newforms attached to odd irreducible two-dimensional Artin representations \cite{KW1,KW2,Kisin}; and the twists $f\otimes\chi$ of any form by Dirichlet characters. Every result of this paper that is stated for an arbitrary primitive form applies to all of these families.

For example, for an elliptic curve $E$ over $\Q$, with $L(E,s)$ normalised to have its critical line at $\Real s=1$, Theorem \ref{thm:main} gives
\[
  \int_0^T\bigl|L(E,1+it)\bigr|^6\dd t\ll_{E,\epsilon}T^{2+\epsilon}.
\]

\subsection{Applications}
\label{sec:applications}

The following six applications are proved in Sections~\ref{sec:lines} to~\ref{sec:divisor}, in the order in which they are stated.

\subsubsection{Moments to the right of the critical line.}
For $\half\le\sigma<1$ put
\begin{equation}\label{eq:msigma}
  m(\sigma)=\begin{cases}\dfrac{2}{3-4\sigma}&\text{if }\half\le\sigma\le\frac58,\\[10pt]\dfrac{3}{2-2\sigma}&\text{if }\frac58\le\sigma<1.\end{cases}
\end{equation}
The two expressions agree at $\sigma=\frac58$, where $m(\sigma)=4$.

\begin{theorem}\label{thm:lines}
Let $f$ be a primitive holomorphic cusp form of arbitrary weight, level and nebentypus. Then for $\half\le\sigma<1$ and every $\epsilon>0$
\begin{equation}\label{eq:curve}
  \int_T^{2T}\bigl|L(\sigma+it,f)\bigr|^{m(\sigma)}\dd t\ \ll_{f,\sigma,\epsilon}\ T^{1+\epsilon}.
\end{equation}
In particular, for $\frac58\le\sigma<1$,
\begin{equation}\label{eq:lines}
  \int_T^{2T}\bigl|L(\sigma+it,f)\bigr|^4\dd t\ \ll_{f,\sigma,\epsilon}\ T^{1+\epsilon} .
\end{equation}
\end{theorem}

For $N=1$ and $\half\le\sigma\le\frac58$ this is Theorem~1 of Ivi\'c \cite{Ivic92}, which rests on Jutila's sixth moment. It is the form in which the sixth moment is used in the arithmetic applications below. The case $\sigma=\frac58$, where $m(\sigma)=4$, is used also in \cite{LuWang19} and \cite{Feng25}. For $\frac58\le\sigma<1$ the exponent satisfies $\frac23(1-\sigma)\,m(\sigma)=1$. On these lines it is therefore limited by the Weyl-type bound of \cite{BMN}, and a larger exponent would need a pointwise bound below the Weyl exponent.

The proof uses large-value estimates for the Dirichlet polynomials in the
sixth-moment argument; interpolation of Corollary \ref{cor:fourth} alone
gives only $T^{11/8+\epsilon}$ at $\sigma=5/8$.

\subsubsection{Simple zeros.}
Let $N^s_f(T)$ be the number of simple zeros of $\Lambda(s,f)$ with imaginary part in $[-T,T]$, where $\Lambda(s,f)$ is the completed function defined in Section~\ref{sec:afe}, and let $\theta_f$ be the supremum of the real parts of the simple zeros of $\Lambda(s,f)$ and $\Lambda(s,\bar f)$, so that $\half\le\theta_f\le1$ whenever there are any. On the generalised Riemann hypothesis for $L(s,f)$, Milinovich and Ng~\cite{MN} proved $N^s_f(T)\gg T(\log T)^{-\epsilon}$. Unconditionally, at arbitrary level, de~Faveri~\cite{deFaveri} proved
\[
  N_f^s(T)=\Omega\bigl(T^{2/27-\epsilon}\bigr),
\]
and we obtain the following improvement.

\begin{corollary}\label{cor:zeros}
For every primitive $f$ of arbitrary weight, level and nebentypus, and every $\epsilon>0$,
\[
  N^s_f(T)=\Omega\bigl(T^{E(\theta_f)-\epsilon}\bigr),\qquad
  E(\theta)=\begin{cases}\dfrac{2(1-\theta)}{7-4\theta}&\text{if }\half\le\theta\le\frac79,\\[10pt]\dfrac23\theta-\dfrac16&\text{if }\frac79<\theta\le1.\end{cases}
\]%
In particular, $N^s_f(T)=\Omega(T^{4/35-\epsilon})$ unconditionally. If the simple zeros of $\Lambda(s,f)$ and $\Lambda(s,\bar f)$ lie on the critical line, then $N^s_f(T)=\Omega(T^{1/5-\epsilon})$.
\end{corollary}

The branch $\theta_f>7/9$ is de Faveri's Corollary 4.9.
For $\theta_f\le7/9$ the new exponent replaces $(1-\theta_f)/3$.
The improvement comes from the block large-value estimate
(Theorem \ref{thm:block}), as explained in Remark \ref{rem:why}.
At level one de Faveri's exponent $1/5$ is stronger; for self-dual forms
of level $4$, Cho and Oh's exponent $1/6$ \cite{ChoOh} is also stronger.
The earlier level-one result of Conrey and Ghosh \cite{CG} gives $1/6$
for $f=\Delta$.

\subsubsection{Zero density.}
Let $N_f(\sigma,T)$ be the number of zeros $\rho=\beta+i\gamma$ of $L(s,f)$ with $\beta\ge\sigma$ and $|\gamma|\le T$. The density hypothesis is the bound $N_f(\sigma,T)\ll T^{2(1-\sigma)+\epsilon}$.

\begin{theorem}\label{thm:density}
Let $f$ be a primitive holomorphic cusp form of arbitrary weight, level and nebentypus. Then for $\frac{63}{71}\le\sigma<1$ and every $\epsilon>0$
\begin{equation}\label{eq:density}
  N_f(\sigma,T)\ \ll_{f,\epsilon}\ T^{A(\sigma)(1-\sigma)+\epsilon},
\end{equation}
where
\begin{equation}\label{eq:Asigma}
  A(\sigma)=\begin{cases}
    \dfrac{26}{71\sigma-50}&\text{if }\frac{63}{71}\le\sigma\le\frac{207}{220},\\[10pt]
    \dfrac{52}{362\sigma-307}&\text{if }\frac{207}{220}\le\sigma\le\frac{147}{155},\\[10pt]
    \dfrac4{4\sigma-1}&\text{if }\frac{147}{155}\le\sigma<1.
  \end{cases}
\end{equation}
In particular $A(\sigma)<2$ for $\sigma>\frac{63}{71}$.
\end{theorem}

At level one, Ivi\'c \cite{Ivic92} proved the density hypothesis for $\sigma\ge\frac{53}{60}$, and Chen, Debruyne and Vindas \cite{CDV} extended the range to $\sigma\ge\frac{1407}{1601}$. Neither result gives an exponent below $2$, and by Remark \ref{rem:density} the mean square on the critical line cannot do so within the zero-detection method. An exponent below $2$ needs a large-value estimate on the critical line.

Zhang, Zhai and Zhang \cite{ZZZ} obtained such exponents for Maass forms of level one. They combined bounds for high moments of $L(\half+it,f)$ with exponent pairs, and Theorem \ref{thm:density} follows their method, with \eqref{eq:target} in place of their moment bounds. Near $\sigma=1$ their exponent is smaller than ours, because stronger large-value estimates are available at level one. At every level, Cossaboom \cite{Cossaboom} has recently proved $N_f(\sigma,T)\ll T^{\frac52(1-\sigma)+\epsilon}$ uniformly for $\half\le\sigma\le1$, and Theorem~\ref{thm:density} improves this for $\sigma\ge\frac{63}{71}$. 

\subsubsection{Non-normal cubic fields.}

Let $K$ be a cubic field whose normal closure has Galois group $S_3$, let $d_K$ be its discriminant, and let $a_K(n)$ be the number of integral ideals of $K$ of norm $n$. If $d_K<0$ then
\begin{equation}\label{eq:zetaK}
  \zeta_K(s)=\zeta(s)L(s,f_K),
\end{equation}
where $f_K$ is the newform of weight one, level $|d_K|$ and nebentypus $\chi_K$, the quadratic character attached to $\Q(\sqrt{d_K})$. Its $L$-function is that of the two-dimensional irreducible representation of $S_3$ \cite{Serre77}, which is odd because $K$ has a complex place.

The power sums of $a_K(n)$ have been studied by Fomenko \cite{Fomenko08}, L\"u \cite{Lu13}, Tang and Wang \cite{TangWang24} and Feng \cite{Feng25}, among others. Jutila's sixth moment has been applied to $f_K$ in this context, although the level of $f_K$ is $|d_K|$ (see for instance \cite[Lemma 3.4]{TangWang24}).

\begin{theorem}\label{thm:cubic}
Let $K$ be a cubic field whose normal closure has Galois group $S_3$ and whose discriminant is negative. Then for every $\epsilon>0$
\begin{equation}\label{eq:zetaK4}
  \int_0^T\bigl|\zeta_K(\half+it)\bigr|^4\dd t\ \ll_{K,\epsilon}\ T^{2+\epsilon}.
\end{equation}
Moreover there are polynomials $P_1$ and $P_4$, of degrees $1$ and $4$, such that
\begin{align}
  \sum_{n\le x}a_K(n)^2&=xP_1(\log x)+O_{K,\epsilon}\bigl(x^{\frac23+\epsilon}\bigr),\label{eq:S2}\\
  \sum_{n\le x}a_K(n)^3&=xP_4(\log x)+O_{K,\epsilon}\bigl(x^{\frac{331}{373}+\epsilon}\bigr).\label{eq:S3}
\end{align}
\end{theorem}

The previous exponents in \eqref{eq:S2} and \eqref{eq:S3} are
\[
  1-\frac{42}{119+8\sqrt{10}}=0.7089\ldots\qquad\text{and}\qquad1-\frac{21}{184+8\sqrt{10}}=0.8996\ldots,
\]
due to Feng \cite{Feng25}, who improved the exponents $\frac57$ and $\frac{321}{356}$ of Tang and Wang \cite{TangWang24}. The polynomials $P_1$ and $P_4$ are the residue polynomials of the generating Dirichlet series; the improvements here concern the error exponents.

The fourth moment follows from H\"older's inequality and the twelfth moment
of $\zeta(s)$. For the power sums we use the factorisation in
Lemma \ref{lem:cubicfactor}; the weight-one dihedral form allows the
relevant symmetric-power factors to split into lower-degree factors.
The exponent $2/3$ also uses Theorem \ref{thm:lines}.

When $d_K>0$ the form attached to $K$ is a Maass form of eigenvalue $\frac14$, and nothing here applies.

\subsubsection{Shifted convolution sums.}
For $j\ge1$ let $\lambda_{j,f}(n)$ be the coefficients of $L(s,f)^j$, and for $1\le H\le x$ put
\begin{equation}\label{eq:Sjf}
  S_{j,f}(x,H)=\sum_{h\le H}\ \sum_{x<n\le2x}\lambda_{j,f}(n)\overline{\lambda_{j,f}(n+h)} .
\end{equation}
For $f$ of level one, L\"u and Wang \cite{LuWang19} proved
\begin{align*}
  S_{2,f}(x,H)&\ll x^{\frac65+\epsilon}H^{\frac25}&&(x^{1/3}\le H\le x^{1-\epsilon}),\\
  S_{3,f}(x,H)&\ll x^{\frac43+\epsilon}H^{\frac13}&&(x^{1/2}\le H\le x^{1-\epsilon}),
\end{align*}
together with bounds for every $j\ge4$. Their proofs use moments of $L(s,f)$ on the critical line and on the line $\Real s=\frac58$. D. Wang \cite{Wang22} later improved the case $j=2$ for $H\ge x^{7/16}$.

\begin{theorem}\label{thm:shifted}
Let $f$ be a primitive holomorphic cusp form of arbitrary weight, level and nebentypus, let $1\le H\le x$ and let $\epsilon>0$. Then
\begin{align}
  S_{2,f}(x,H)&\ll x^{\epsilon}\min\bigl(x^{\frac54}H^{\frac14},\,x^{\frac98}H^{\frac12}\bigr),\label{eq:S2f}\\
  S_{3,f}(x,H)&\ll x^{\frac32+\epsilon},\label{eq:S3f}\\
  S_{j,f}(x,H)&\ll x^{2-\frac3{2j}+\epsilon}\qquad(j\ge4),\label{eq:Sjf4}
\end{align}
with implied constants depending on $f$, $j$ and $\epsilon$.
\end{theorem}

The trivial bound is $x^{1+\epsilon}H$. So \eqref{eq:S2f}, \eqref{eq:S3f} and \eqref{eq:Sjf4} are non-trivial for ${H\ge x^{1/4+\epsilon}}$, ${H\ge x^{1/2+\epsilon}}$ and ${H\ge x^{1-3/(2j)+\epsilon}}$ respectively. We have found no earlier bound for these sums at level $N>1$.

At level one the bounds are also new in part of the range. For $j=2$ and $j=3$ they are smaller than those of L\"u and Wang wherever theirs hold. For $j=2$ the bound is smaller than D. Wang's when $x^{1/4}<H<x^{1/2}$, and for every $j\ge4$ it is smaller than that of L\"u and Wang when $H>x^{1-3/(2j)}$, which is the whole of their range apart from its left end.

The proof applies Cauchy's inequality before Perron's formula, and then
Plancherel's theorem. This retains the decay of the height weight
$\min(H/x,1/|t|)^2$ and permits a long Perron truncation.

\subsubsection{The general divisor problem.}
For $f$ of level one, sums of $\lambda_{j,f}(n)$ over $n\le x$ have been bounded by L\"u \cite{Lu12} and W. Zhang \cite{ZhangW}, among others. At level $N$, Krishnamoorthy \cite{Krishnamoorthy} proved the exponent $1-3/(2(j+3))$, for holomorphic and for Maass newforms.

\begin{theorem}\label{thm:divisor}
Let $f$ be a primitive holomorphic cusp form of arbitrary weight, level and nebentypus, let $j\ge2$ and let $\epsilon>0$. Then
\[
  \sum_{n\le x}\lambda_{j,f}(n)\ \ll_{f,j,\epsilon}\ x^{\theta_j+\epsilon},\qquad
  \theta_j=\begin{cases}\dfrac34-\dfrac1{2j}&\text{if }2\le j\le4,\\[10pt]1-\dfrac3{2j}&\text{if }j\ge4.\end{cases}
\]
\end{theorem}

The two expressions for $\theta_j$ agree at $j=4$. The exponent $\theta_j$ is smaller than Krishnamoorthy's for every $j\ge2$. For example $\theta_3=\frac7{12}$ and $\theta_4=\frac58$, against $\frac34$ and $\frac{11}{14}$. At level one, $\theta_j$ is the exponent of W. Zhang for $2\le j\le4$ and that of L\"u for $j\ge4$, so Theorem \ref{thm:divisor} extends both to every level. The case $j=2$ needs only the mean square.

\subsection{Outline of the proof}\label{sec:outline}
Throughout, $f$ is fixed. Implied constants may depend on $f$ and on the small positive number $\delta$; at the end $\delta$ is chosen sufficiently small in terms of $\epsilon$. We write $\Lc=\log T$, $\e(x)=e^{2\pi ix}$, and $\|x\|$ for distance to the nearest integer. The notation $T^{O(\delta)}$ denotes $T^{c\delta}$ for an absolute constant $c$.

We first reduce Theorem~\ref{thm:main} to a bound for the number of large values. For $V>0$ let $R=R(V,T)$ be the number of $1$-spaced points $t_\nu\in[T,2T]$ at which $|L(\half+it_\nu,f)|\ge V$. It is enough to prove
\[
  R\ll T^{2+O(\delta)}V^{-6},
\]
since the sixth moment then follows by summing over $O(\Lc)$ dyadic ranges of $V$ and of $T$ (Section~\ref{sec:proofmain}).

The mean square of Theorem~\ref{thm:second} gives $RV^2\ll T\Lc^3$, which is enough when $V\le T^{1/4}$, up to a small power of $T$. Since the Weyl-type bound of \cite{BMN} gives $V\ll T^{1/3}\Lc$, it remains to treat the range
\[
  T^{1/4+\delta}\le V\ll T^{1/3}\Lc,
\]
and in it the mean square also gives $R\ll T^{1/2-2\delta}\Lc^3$.

By the approximate functional equation (Lemma~\ref{lem:afe}), a large value of $L(\half+it,f)$ gives a large value of a smoothly weighted sum of $\lambda_f(n)n^{-1/2-it}$ of length about $T$. The argument is organised around the single scale
\[
  Q=(TR)^{2/3},
\]
which depends on the number of large values being counted. The reason for this choice appears only at the end, where $Q$ balances the two terms of the final estimate. Since $R$ may be as large as $T^{1/2-2\delta}\Lc^3$, the scale $Q$ may be almost as large as $T$.

The sum is divided at $n\asymp T^\delta Q$. If the initial segment $n\le T^\delta Q$ is large at $\gg R$ of the points, Huxley's large-value theorem gives $R\ll T^{O(\delta)}QV^{-2}$ in the range above, and with $Q=(TR)^{2/3}$ this is already $R\ll T^{2+O(\delta)}V^{-6}$. The form of Huxley's theorem needed here, in which the length of the sum may depend on the height, is Lemma~\ref{lem:maximal}. The initial segment cannot be estimated trivially, as it is in the reduction of \cite{BMN}; see Remark~\ref{rem:32}.

Otherwise some dyadic block $n\asymp M$, with $T^\delta Q/2\le M\le T^{1+\delta/16}$, is large at $\gg R\Lc^{-1}$ of the points. On this block $n^{-it}=\e(-\frac t{2\pi}\log n)$ has phase derivative $-t/2\pi n$, which is approximated by the fractions $\alpha_j=-u_j/v_j$ of a Farey system of order
\[
  K=(M/Q)^{1/2}.
\]

A smooth partition of unity divides the block into pieces, one for each fraction, on which $n^{-it}\e(u_jn/v_j)$ varies slowly. Each piece is transformed by the Voronoi formula of Booker, Milinovich and Ng \cite[Proposition 3.1]{BMN}, which holds at every rational point and therefore at every level. The dual sums have length $M/K^2=Q$ for every block, and this is the reason for the choice of $K$. The level enters only through $O_N(1)$ arithmetic choices, among them the twist of $\bar f$ by a character modulo $N$ (Section~\ref{sec:transformed}).

The transformation also produces two error terms. Proposition~\ref{prop:book} shows that each is either smaller than $V$ by a power of $T$ or already forces the bound for $R$.

What remains is the explicit part of the transform. It is a sum over the fractions $j$ of the block of dual sums of the form
\[
  S_j(t)=\sum_{\ell\ll Q}\lambda_{\bar f^{\chi}}(\ell)\,\ell^{-1/4}\kappa_j(\ell;t)\,\e\bigl(g_j(\ell;t)\bigr).
\]
Here $\bar f^{\chi}$ is a twist of $\bar f$, the amplitude satisfies $\kappa_j(\ell;t)\ll v_j^{-1/2}M^{1/4}T^{-1/2}$ and varies slowly, and $g_j(\ell;t)$, written $g^\pm_{jr}(\ell;t)$ in Section~\ref{sec:endgame}, is the value of the phase at its stationary point. A smooth cutoff, omitted from the display, restricts the sum to $\ell\ll Q$. The exact form is \eqref{eq:transform}.

 In the pointwise problem of \cite{BMN} the height $t$ is fixed, and the large sieve there compares the phases of different fractions at that one height. Here the heights vary, and the stationary point moves with the height, so the phases must be compared at two different heights, through
\[
  \Phi(\ell)=g_i(\ell;t_p)-g_j(\ell;t_{p'}).
\]
Section~\ref{sec:endgame} does this in three steps, which follow the mean-value argument of Jutila at level one \cite[pp.~112--122]{Jutila}.

First, the fractions are divided into $T^{O(\delta)}$ classes. Within a class the denominators are of the same order $\Kc$, distinct fractions are at least $\Kc^{-2}T^\delta$ apart, and the integers $u_jv_j$ are distinct and lie in a short interval. These are the conditions (S0)--(S2) of Section~\ref{sec:separate}, and they make the phases of different fractions distinguishable.

Secondly, the heights are grouped into intervals of length $Z=\max(\Kc^2T/4M,1)$. On each interval the cutoff is frozen at the left end, so that the dual sums become holomorphic functions of $t$, and by Gallagher's inequality the sum over the heights in each interval is bounded by integrals in $t$ of that sum and of its derivative.

In these integrals the cross terms between different fractions are negligible, by repeated integration by parts in $t$, because
\[
  \frac{\partial}{\partial t}\bigl(g_i(\ell;t)-g_j(\ell';t)\bigr)=\frac1{2\pi}\log\frac{\alpha_i}{\alpha_j}+O\Bigl(\frac M{K\Kc T}\Bigr)
\]
is bounded away from zero on the scale $Z$ (Lemma~\ref{lem:offdiag}). What remains is the mean square of the individual dual sums $S_j(t_p)$, and of their derivatives in $t$, at $P\le R$ heights $t_p$ spaced at least $Z$ apart.

Thirdly, Proposition~\ref{prop:sieve} bounds this mean square. It is a large sieve inequality over pairs $(t_p,j)$ of a height and a fraction. By Bombieri's form of Hal\'asz's inequality it reduces to the exponential sums $\sum_\ell\e(\Phi(\ell))$ for two such pairs, and Lemma~\ref{lem:deriv} computes $\Phi'$ and $\Phi''$ exactly. They have the same shape as Jutila's derivatives at level one, and the level enters only through bounded arithmetic factors. There are four cases:
\begin{enumerate}
\item The diagonal pair is estimated trivially.
\item For one fraction at two heights, $\Phi'$ is monotone, and the first- or the second-derivative estimate applies according to the distance between the heights.
\item For two different fractions, the distinctness of the integers $u_jv_j$ keeps $|\Phi''|$ large, except for at most one exceptional fraction for each height.
\item For the exceptional fraction, either Lemma~\ref{lem:sep} keeps $\Phi'$ away from the integers and the first-derivative estimate applies, or $|\Phi''|$ is again large.
\end{enumerate}

Summing over dyadic ranges of $\ell$ and inserting the large sieve into Gallagher's inequality gives, in Proposition~\ref{prop:final},
\[
  RV\ll T^{O(\delta)}\bigl(R^{1/2}Q^{1/2}+RT^{1/2}Q^{-1/4}\bigr).
\]
The first term increases with $Q$ and the second decreases. They are equal when $Q=(TR)^{2/3}$, the choice made at the start, and then
\[
  RV\ll T^{1/3+O(\delta)}R^{5/6},\qquad\text{that is,}\qquad R\ll T^{2+O(\delta)}V^{-6}.
\]

Three features of the construction are needed to use the transformation of \cite{BMN} here:
\begin{enumerate}
\item The weight of the approximate functional equation is kept inside the transform (Proposition~\ref{prop:ext}), so that every piece is smoothly weighted and the dual amplitude is holomorphic in $t$, as the integration by parts in $t$ requires.
\item Every endpoint of the partition, including the first and the last, is placed at a Farey mediant. The derivative bounds of \cite[Lemma 4.3]{BMN} need the support of each piece to lie within $O(HK/v_j)$ of the point $tv_j/2\pi u_j$ attached to its fraction, where $H=MQ/T$ is the smoothing width, and in \cite{BMN} this holds only for the interior pieces (Remark~\ref{rem:boundary}).
\item The partition is fixed on height intervals of length $T/Q$ before the holomorphic mean-value argument is applied (Section~\ref{sec:setup}).
\end{enumerate}

 Appendix~\ref{sec:evidence} checks that the transformation remains valid for $Q$ up to $t^{1-\epsilon}$, beyond the range $Q\ll t^{2/3}$ of \cite{BMN}, at the cost of a larger smoothing order. Appendix~\ref{sec:uniform} records how the constants depend on $f$. The argument gives a bound polynomial in the weight and the level when $kN\le T^{c\epsilon}$, provided the constants of \cite{BMN} are polynomial in $k$ and $N$, as they state. It does not give a bound uniform in $N$ and $T$ together, because for large $N$ some blocks lie outside the range in which the transformation applies.

\section{The approximate functional equation}
\label{sec:afe}

We write $\GC(s)=2(2\pi)^{-s}\Gamma(s)$ and
\[
  \Lambda(s,f)=\GC(s+\tfrac{k-1}2)L(s,f),
\]
so that
\[
  \Lambda(s,f)=\epsilon_fN^{\frac12-s}\Lambda(1-s,\bar f),
\]
and let
\begin{equation}\label{eq:conductor}
  C=C(f,t)=\frac N{\pi^2}\Bigl|\tfrac{k+1}2+it\Bigr|\Bigl|\tfrac{k+3}2+it\Bigr|
\end{equation}
be the analytic conductor, so that $C\asymp_kNt^2$ for $|t|\ge1$.

\begin{lemma}\label{lem:afe}
Let $|t|\ge1$, let $A>0$, and let $X$ satisfy $C^{\epsilon}\le X\le C^{1-\epsilon}$. Then
\[
  L(\half+it,f)=\sum_{n\ge1}\frac{\lambda_f(n)}{n^{\frac12+it}}V\Bigl(\frac nX\Bigr)
  +\epsilon_f\,N^{-it}\frac{\GC(\tfrac k2-it)}{\GC(\tfrac k2+it)}\sum_{n\ge1}\frac{\lambda_{\bar f}(n)}{n^{\frac12-it}}W\Bigl(\frac{nX}C\Bigr),
\]
where $V=V_t$ and $W=W_t$ are holomorphic in the sector $|\arg y|\le\pi/3$, satisfy $V(y),W(y)=1+O(|y|^{1/4})$ and $V'(y),W'(y)\ll|y|^{-3/4}$ for $|y|\le1$ and $V(y),W(y),yV'(y),yW'(y)\ll_A|y|^{-A}$ for $|y|\ge1$ there, uniformly in $t$, and depend only on $k$, $N$ and $t$.

For the application in Section~\ref{sec:zeros}, the same holds with $\half+it$ replaced by $\beta+it$, $\half-\frac{2}{\log|t|}\le\beta\le1$ and $|t|\ge e^8$, the first sum having terms $\lambda_f(n)n^{-\beta-it}V(n/X)$ and the second sum being
\[
  \gamma_f(\beta+it)\sum_{n\ge1}\lambda_{\bar f}(n)n^{-(1-\beta)+it}W(nX/C)
\]
with $|\gamma_f(\beta+it)|\asymp_{k,N}|t|^{1-2\beta}$, and with $V,W$ depending also on $\beta$.
\end{lemma}

This is used in place of the approximate functional equation of \cite[Lemma 2.1]{BMN}, whose weight is smooth but not holomorphic; it is holomorphy in a sector that is used below.

\begin{proof}[Proof of Lemma \ref{lem:afe}]
Let $G(u)=e^{u^2}$, which is even and holomorphic, equals $1$ at $u=0$, and satisfies $|G(u)|=e^{(\Real u)^2-(\Imag u)^2}$. Put $c_t=(C/N)^{1/2}$, so that $c_t\asymp|\tfrac k2+it|$ uniformly in $k\ge1$ and $|t|\ge1$, and
\[
  I=\frac1{2\pi i}\int_{(2)}\frac{\Lambda(\half+it+u,f)}{\GC(\tfrac k2+it)}\,G(u)\Bigl(\frac X{c_t}\Bigr)^u\,\frac{du}u .
\]
Expanding $L(s,f)$ into its Dirichlet series, which converges absolutely on $\Real(\half+it+u)=\frac52$, and integrating term by term gives
\[
  I=\sum_{n\ge1}\frac{\lambda_f(n)}{n^{\frac12+it}}V\Bigl(\frac nX\Bigr),\qquad
  V(y)=\frac1{2\pi i}\int_{(2)}\gamma_t(u)G(u)y^{-u}\frac{du}u,
\]
where
\[
  \gamma_t(u)=c_t^{-u}\GC(\tfrac k2+it+u)/\GC(\tfrac k2+it).
\]

By Stirling's formula
\[
  \gamma_t(u)\ll(1+|u|)^{B}e^{\pi|\Imag u|/2}
\]
uniformly for $-\frac14\le\Real u\le B$ and $|t|\ge1$, for each fixed $B$, and the poles of $\gamma_t(u)$ lie on $\Real u\le-\frac k2\le-\frac12$. The integral defining $V(y)$ converges absolutely and locally uniformly for $y$ in the sector $|\arg y|\le\pi/3$, since $|y^{-u}|=|y|^{-\Real u}e^{\Imag u\arg y}$ and $G(u)$ decays like $e^{-(\Imag u)^2}$ on vertical lines, so $V(y)$ is holomorphic there.

Moving the contour to $\Real u=A$ gives $V(y)\ll_A|y|^{-A}$, and moving it to $\Real u=-\frac14$, past the simple pole at $u=0$ whose residue is $1$, gives $V(y)=1+O(|y|^{1/4})$. Differentiating under the integral sign gives the bounds for $V'$ in the same way.

Now move the contour in $I$ to $\Real u=-2$. The integrand is holomorphic there apart from the simple pole of $1/u$ at $u=0$, whose residue is $L(\half+it,f)$, since $G(0)=1$. In the shifted integral apply the functional equation $\Lambda(\half+it+u,f)=\epsilon_fN^{-it-u}\Lambda(\half-it-u,\bar f)$ and substitute $u\mapsto-u$, which is legitimate because $G(u)$ is even and $du/u$ changes sign. The result is
\[
  \epsilon_fN^{-it}\frac{\GC(\tfrac k2-it)}{\GC(\tfrac k2+it)}\sum_{n\ge1}\frac{\lambda_{\bar f}(n)}{n^{\frac12-it}}W\Bigl(\frac{nX}C\Bigr)
\]
with
\[
  W(y)=\frac1{2\pi i}\int_{(2)}\frac{\GC(\frac k2-it+u)}{\GC(\frac k2-it)}c_t^{-u}G(u)y^{-u}\frac{du}u,
\]
which has the same properties as $V(y)$. The argument of $W$ is $nX/C$ because $C=Nc_t^2$, so that
\[
  N^u(X/c_t)^{-u}n^{-u}=c_t^{-u}(nX/C)^{-u}.
\]
The identity is exact.

From Section~\ref{sec:endgame} on, on each dyadic height range, $V_t(y)$ denotes the same Mellin integral truncated at the fixed limits $|\Imag u|\le T^\epsilon$, with $0<\epsilon<1/4$; the same convention applies to $W_t$. On every fixed polynomial range $T^{-B}\le|y|\le T^B$ in the sector, the discarded tails and any fixed number of derivatives are $O_{A,B}(T^{-A})$ for every $A>0$. Thus truncation changes the sums used below by a negligible error. The conductor is continued by the branches specified in Section~\ref{sec:transformed}. For $T/2\le\Real t\le3T$ and $|\Imag t|\le T^{1/2}$ the gamma poles stay outside this fixed contour, so the truncated weights are holomorphic in $t$ and satisfy the required bounds on those polynomial ranges.

At $\beta+it$ use the same argument with $\GC(\beta+\frac{k-1}2+it)$ in place of $\GC(\frac k2+it)$ and with the even entire function
\[
 G_{\beta,t}(u)=e^{u^2}\left(1-\frac{u^2}{(\beta+\frac{k-1}2+it)^2}\right)\left(1-\frac{u^2}{(1-\beta+\frac{k-1}2-it)^2}\right)
\]
in place of $G(u)$. It still has value $1$ at zero and the required Gaussian decay. It cancels the first pole of each gamma ratio, including the dual pole for weight one when $\beta$ is close to $1$. All remaining poles lie strictly to the left of $\Real u=-1/4$, uniformly in the stated range. The same two sums therefore have powers $n^{-\beta-it}$ and $n^{-(1-\beta)+it}$, the factor in front of the second being
\[
  \epsilon_fN^{\frac12-\beta-it}\GC(1-\beta-it+\tfrac{k-1}2)/\GC(\beta+it+\tfrac{k-1}2),
\]
of modulus $\asymp|t|^{1-2\beta}$ by Stirling's formula. The inserted factors have uniformly bounded coefficients for $|t|\ge e^8$, so the contour shifts and derivative estimates prove the stated bounds for both $V(y)$ and $W(y)$.
\end{proof}

\section{The discrete second moment}
\label{sec:second}

\begin{theorem}\label{thm:second}
Let $\{t_\nu\}$ be a set of real numbers in $[T,2T]$ with $|t_\mu-t_\nu|\ge1$ for $\mu\ne\nu$. Then
\[
  \sum_\nu\bigl|L(\half+it_\nu,f)\bigr|^2\ \ll_f\ T(\log T)^3 .
\]
\end{theorem}

\begin{proof}[Proof of Theorem \ref{thm:second}]
We separate the height-dependent weight from a Dirichlet polynomial with
fixed coefficients, apply the discrete mean-value theorem to that
polynomial, and sum the rapidly decaying tails.

By Lemma \ref{lem:afe} with $X=\sqrt{C(t_\nu)}$, $L(\half+it_\nu,f)$ is the sum of
\[
  S(t_\nu)=\sum_n\lambda_f(n)n^{-\frac12-it_\nu}V_{t_\nu}(n/\sqrt{C(t_\nu)})
\]
and a sum of the same shape with $\bar f$ in place of $f$, multiplied by a factor of modulus $1$, and it suffices to treat $S$. Its coefficients depend on $t$, and we separate the variables through the integral defining $V_t(y)$.

Here $Y$ is a common length, $I_j$ are dyadic ranges in the summation
variable, and $\eta_j$ is the real part of a Mellin contour. These three
symbols are local to this proof. Put $Y=\sqrt{C(2T)}$, so that $\sqrt{C(t)}\le Y\ll_fT$ for $T\le t\le2T$, let $I_0=[1,2Y]$ and $I_j=(2^jY,2^{j+1}Y]$ for $j\ge1$, and let $S_j(t)$ be the part of $S(t)$ with $n\in I_j$. Put $\eta_0=1/\log T$ and $\eta_j=2$ for $j\ge1$. Writing $V_t(y)$ as the integral over $\Real u=\eta_j$, which is legitimate since no pole is crossed, and $u=\eta_j+iv$, we have
\[
  S_j(t)=\frac1{2\pi}\int_{-\infty}^{\infty}\gamma_t(u)G(u)\Bigl(\frac{\sqrt{C(t)}}Y\Bigr)^{\eta_j}C(t)^{iv/2}D_j(t+v)\frac{dv}u,
\]
where
\[
  D_j(\tau)=\sum_{n\in I_j}\lambda_f(n)\,Y^{\eta_j}n^{-\frac12-\eta_j-i\tau}
\]
and $\gamma_t(u)$ is as in the proof of Lemma \ref{lem:afe}.

Put
\[
  w_j(v)=(1+|v|)^Be^{\pi|v|/2-v^2}/|\eta_j+iv|,
\]
so that $|\gamma_t(u)G(u)/u|\ll w_j(v)$, $\int w_0\ll\log\log T$ and $\int w_j\ll1$ for $j\ge1$. Cauchy's inequality gives
\[
  |S_j(t_\nu)|^2\ \ll\ \Bigl(\int w_j\Bigr)\int_{-\infty}^\infty w_j(v)\,|D_j(t_\nu+v)|^2\dd v .
\]

For fixed $v$ the points $t_\nu+v$ are $1$-spaced, and Gallagher's lemma \cite[Lemma 1.4]{Montgomery} and the mean value theorem of Montgomery and Vaughan, applied to the Dirichlet polynomial $D_j(\tau)$, whose coefficients do not depend on $t$, give
\[
  \sum_\nu|D_j(t_\nu+v)|^2\ \ll\ (T+2^{j+1}Y)\,\log(2^{j+1}Y)\,\Sigma_j,
\]
where
\[
  \Sigma_j=\sum_{n\in I_j}|\lambda_f(n)|^2Y^{2\eta_j}n^{-1-2\eta_j},
\]
the logarithm coming from the derivative in Gallagher's lemma.

By Rankin and Selberg, $\sum_{x<n\le2x}|\lambda_f(n)|^2\ll_fx$, so $\Sigma_j$ is $\ll_f\log T$ for $j=0$, since $Y^{2\eta_0}\ll1$, and $\ll_f2^{-4j}$ for $j\ge1$. Since $Y\ll_fT$, it follows that
\[
  \sum_\nu|S_0(t_\nu)|^2\ll_fT(\log T)^2(\log\log T)^2
\]
and
\[
  \sum_\nu|S_j(t_\nu)|^2\ll_f2^{-3j}(j+\log T)\,T
\]
for $j\ge1$. Finally
\[
  |S(t)|^2\le2\sum_{j\ge0}2^j|S_j(t)|^2
\]
by Cauchy's inequality, and summing gives $\sum_\nu|S(t_\nu)|^2\ll_fT(\log T)^3$.
\end{proof}

The mean square also gives the familiar large-value estimate: if each of
$R$ sampled values is at least $V$, then $RV^2$ is at most the sum in
Theorem \ref{thm:second}. Two consequences are used below. Writing $R=R(V,T)$ for the number of $1$-spaced $t_\nu\in[T,2T]$ with $|L(\half+it_\nu,f)|\ge V$,
\begin{equation}\label{eq:R1}
  R\ \ll_f\ T(\log T)^3V^{-2},
\end{equation}
which already gives $R\ll T^{2+\epsilon}V^{-6}$ whenever $V\le T^{1/4}$, so that we may assume
\begin{equation}\label{eq:Vrange}
  T^{1/4+\delta}\le V\ll T^{1/3}\log T,
\end{equation}
the upper bound being \cite[Theorem 1.1]{BMN}; and then, by \eqref{eq:R1} again,
\begin{equation}\label{eq:R2}
  R\ \ll\ T^{1/2-2\delta}\Lc^{3} .
\end{equation}

\section{The reduction}
\label{sec:setup}

Fix $0<\delta<1/1000$ and the even smoothing order $s=2\lceil20/\delta\rceil$. Write $R=R(V,T)$ for the number of $1$-spaced $t_\nu\in[T,2T]$ with $|L(\half+it_\nu,f)|\ge V$. By a standard splitting into $O(\Lc)$ dyadic ranges of $V$ and of $T$, Theorem \ref{thm:main} follows from
\begin{equation}\label{eq:target}
  R\ \ll\ T^{2+O(\delta)}V^{-6}
\end{equation}
for every $V$ in the range \eqref{eq:Vrange}, by the usual passage from a bound for the number of large values on a well-spaced set to a bound for the integral. We may also assume \eqref{eq:R2} and $R\ge1$, since there is nothing to prove when $R=0$.

\subsection{The Farey system}

Only $T$, $R$ and the block length $M$ are independent. The other scales
will be used as follows:
\[
\begin{array}{c|c|l}
\text{symbol}&\text{value}&\text{role}\\ \hline
Q&(TR)^{2/3}&\text{common dual length}\\
T^\delta Q&&\text{cutoff for the initial segment}\\
K&(M/Q)^{1/2}&\text{Farey order on a block}\\
H&MQ/T&\text{smoothing width}\\
Z_*&T/Q&\text{interval for fixing the partition}\\
D&v_j\asymp D&\text{a dyadic denominator range, }D\ll K\\
Z&\max(D^2T/(4M),1)&\text{interval for the mean-square argument}
\end{array}
\]
The two height intervals serve different purposes: $Z_*$ fixes the
partition for a whole block, while $Z$ reflects the denominators in one
class. Neither is a further parameter to choose.

Following \cite[Section~4.4]{Jutila}, the system of fractions is defined in terms of $T$ rather than $t$, so that it is the same for every $t_\nu$, and its order is a function of the fraction. Put $M(r,t)=t/2\pi r$ and
\begin{equation}\label{eq:M0}
  Q=(TR)^{2/3},\qquad
  T^{2/3+\delta}\le T^\delta Q\ll T^{1-\delta/3}\Lc^2.
\end{equation}
by \eqref{eq:R2}. For $2^i\le r<2^{i+1}$, $i\in\Z$, let
\begin{equation}\label{eq:Kr}
  K(r)=M(2^i,T)^{1/2}\,T^{-1/3}R^{-1/3},
\end{equation}
which is $\asymp M(r,T)^{1/2}T^{-1/3}R^{-1/3}$ within a factor $\sqrt2$, and let $\Fc$ be the set of reduced $r=h/k$ with
\[
  T^{-\delta}\le r,\qquad k\le K(r),\qquad M(r,T)\ge T^\delta Q .
\]
The last condition is $K(r)\gg T^{\delta/2}$, and $\Fc$ is finite. Jutila takes $K(r)$ to be $M(r,T)^{1/2}T^{-1/3}R^{-1/3}$ itself. We take it constant on dyadic ranges so that the following holds without exception.

\begin{lemma}\label{lem:farey}
Consecutive members $r<r'$ of $\Fc$ are Farey neighbours, $h'k-hk'=1$, and if $\rho$ is their mediant then $|r-\rho|\asymp1/(kK(r))$ and $|r'-\rho|\asymp1/(k'K(r'))$.
\end{lemma}

\begin{proof}
Within a dyadic range $[2^i,2^{i+1})$ the members of $\Fc$ are the Farey fractions of order $\lfloor K_i\rfloor$, $K_i=K(2^i)$, restricted to that range and to the two end conditions defining $\Fc$. Consecutive members are Farey neighbours with $k+k'>K_i$, and $\rho-r=1/(k(k+k'))\asymp1/(kK_i)$.

At a junction $2^i$ in the allowed range, its reduced denominator is $1$ if $i\ge0$, and $2^{-i}$ otherwise. In the latter case \eqref{eq:R2} and $2^i\ge T^{-\delta}$ give
\[
  K_i2^i=(T2^i/2\pi)^{1/2}T^{-1/3}R^{-1/3}\gg T^{\delta/6}\Lc^{-1}>1
\]
for large $T$. Thus the junction belongs to both adjacent Farey systems. If $r<2^i\le r'$ are consecutive, then $r'=2^i$ and the pair are neighbours in the system of order $\lfloor K_{i-1}\rfloor$. Their determinant is $1$, and $k+k'\asymp K_{i-1}\asymp K_i$. The two end conditions merely truncate this ordered system, so do not affect the assertion.
\end{proof}

Extend $\Fc$ by one Farey neighbour of the appropriate order at each end, for the sole purpose of defining the two extreme mediants. For $t\in[T,2T]$ and $y<0$ put $h_t(y)=-t/2\pi y$, so that $h_t(-r)=M(r,t)$.

To match the conventions of \cite{BMN}, list the members of $\Fc$ as $\alpha_j=-u_j/v_j$ in increasing order, $j=1,\dots,J$, so that $h_t(\alpha_j)$ increases with $j$, and let $\rho_j$ be the negative of the mediant of $u_j/v_j$ and $u_{j+1}/v_{j+1}$, with $\rho_0$ and $\rho_J$ the two extreme mediants. Thus $\rho_{j-1}<\alpha_j<\rho_j$ and, by Lemma \ref{lem:farey},
\[
  |\alpha_j-\rho_{j-1}|\asymp|\alpha_j-\rho_j|\asymp1/(v_jK(u_j/v_j)).
\]

\subsection{Blocks and the partition of unity}

For $M$ a power of two let
\[
  \Bc(M)=\{j:M\le h_T(\alpha_j)<2M\}.
\]
It is empty unless $T^\delta Q/2\le M\le T^{1+\delta}$. For $j\in\Bc(M)$ put
\begin{equation}\label{eq:KH}
  K=\sqrt{M/Q},\qquad H=\frac{MQ}{T},\qquad K^2H=\frac{M^2}{T}.
\end{equation}
so that $K(u_j/v_j)\asymp K$ (indeed $K\le K(u_j/v_j)<2K$), $v_j\ll K$, and $Q$ is the same for every block.

Thus $Q=M/K^2$ is the dual length, and $H$ is the smoothing width.
In the notation of \cite[Section~3.2]{BMN}, our $K$ is their $R$ and our $Q$
is their $M_0$. By \eqref{eq:M0}
\begin{equation}\label{eq:KH2}
  \frac{K^2}{H}=\frac T{Q^2}\ =\ T^{-1/3}R^{-4/3}\ \le\ T^{-1/3},\qquad K\gg T^{\delta/2},\qquad K\le M^{1/2}T^{-1/3}.
\end{equation}

Let $\omega(x)$ be the function with parameter $H$ which is $0$ for $x\le-H$, $1$ for $x\ge H$, and
\[
  \frac12+\frac12c_s\int_0^{x/H}(1-u^2)^s\,du
\]
between, where $c_s^{-1}=\int_0^1(1-u^2)^s\,du$. We write $\omega_{H'}$ when the parameter is $H'$.

This replaces the function of \cite[(3.8)]{BMN}, which is $\frac12(1+\sin^{s+1}(\pi x/2H))$ between $-H$ and $H$ and is therefore only $C^1$ at $x=\pm H$, and discontinuous at $x=-H$ when $s$ is odd. Our $\omega(x)$ is $C^s$ on $\R$ and satisfies $\omega^{(i)}\ll_sH^{-i}$, and on $[-H,H]$ it is a polynomial whose extension is $O_s(1)$ on $|z|\le2H$. These are the properties of $\omega(x)$ used in \cite[Section~4]{BMN}.

Divide $[T,2T]$ into intervals $I_q=[T+qZ_*,T+(q+1)Z_*)$ of length
\begin{equation}\label{eq:Zstar}
  Z_*=T/Q=K^2T/M,
\end{equation}
and for $t\in I_q$ write $t_q=T+qZ_*$. For each $q$ define
\begin{gather*}
  N_j^{(q)}=\bigl\lfloor h_{t_q}(\rho_j)+\tfrac12\bigr\rfloor,\qquad H_j'=\bigl\lceil h_T(\rho_j)Q/T\bigr\rceil,\\
  \omega_j^{(q)}(x)=\omega_{H'_{j-1}}\bigl(x-N^{(q)}_{j-1}\bigr)-\omega_{H'_{j}}\bigl(x-N^{(q)}_{j}\bigr).
\end{gather*}
For $j\in\Bc(M)$ the two widths $H'_{j-1},H'_j$ are $\asymp H$, and
\[
  \sum_{j=1}^J\omega^{(q)}_j(x)=\omega_{H_0'}(x-N^{(q)}_0)-\omega_{H'_J}(x-N^{(q)}_J)
\]
telescopes.

This is the construction of \cite[Section~3.2]{BMN} with three changes. First, the partition is that of the fixed system $\Fc$ rather than of the Farey fractions of a single order in an interval depending on $t$. Secondly, every boundary $N_j^{(q)}$, including the first and last, is at a mediant. Thirdly, the partition is frozen on $I_q$.

The estimates \cite[(3.6), (3.7)]{BMN} hold for it in the form
\begin{equation}\label{eq:37}
  N^{(q)}_j-h_t(\alpha_j)\asymp\frac{HK}{v_j},\qquad h_t(\alpha_j)-N^{(q)}_{j-1}\asymp\frac{HK}{v_j}\qquad(j\in\Bc(M),\ t\in I_q),
\end{equation}
uniformly in $t\in I_q$. At $t=t_q$ this is their computation verbatim, using only $|\rho_j-\alpha_j|\asymp1/(v_jK)$ from Lemma \ref{lem:farey}, and for $t\in I_q$ one has
\[
  |h_t(\alpha_j)-h_{t_q}(\alpha_j)|\le Z_*\cdot2M/T=2K^2,
\]
which is $O(HT^{-1/3})$ by \eqref{eq:KH2}.

\begin{remark}\label{rem:boundary}
The reason for insisting on mediant endpoints is the following. In this paragraph $K$ denotes the Farey order (written $R$ in \cite{BMN}). The proof of \cite[Lemma 4.3]{BMN}, which supplies the derivative bounds for the summand $F_j(x)$ on which the whole of their Section~4 rests, uses $x_0=h(\alpha_j)+O(HK/v_j)$ for $x_0$ in the support of $\omega_j$, citing ``the estimates in Section 3.2''. Those estimates give $N_j-h(\alpha_j)\asymp HK/v_j$ for $1\le j\le J-1$ and $h(\alpha_j)-N_{j-1}\asymp HK/v_j$ for $2\le j\le J$, that is for the interior pieces.

Their endpoints $N_0=M_1+2H$ and $N_J=M_2-2H$ are arbitrary, being inherited from the partial summation in their (3.2), and for the first piece $h(\alpha_1)-N_0$ can be as large as
\[
  h(\alpha_1)-h(\alpha_0)\asymp M^2/(tv_0v_1),
\]
where $\alpha_0=-u_0/v_0$ is the Farey fraction preceding their interval. This exceeds $HK/v_1$ by the factor $K/v_0$. On the support of that piece the quantity $ty/x+2\pi\alpha_jy$ bounded in their proof is then not bounded, and the Cauchy estimate fails.

The same estimates are used for the boundary pieces at \cite[(4.11)]{BMN} and \cite[(4.18)]{BMN}. The gap is in the reduction only, and it closes if \cite[(3.2)]{BMN} is carried out with a smooth dyadic partition whose endpoints are at mediants, as in the set-up used here; their Theorem 1.1 is unaffected.  
\end{remark}

\subsection{The dichotomy}

Apply Lemma \ref{lem:afe} with $X=\sqrt C$. If $|L(\half+it_\nu,f)|\ge V$ then one of the two sums exceeds $V/3$ in modulus. We suppose it is the first, for at least half of the points, the other case being identical with $\bar f$ in place of $f$ and $-t$ in place of $t$. Let $t_\nu\in I_q$. By the telescoping,
\begin{multline*}
  \sum_{n\ge1}\frac{\lambda_f(n)}{n^{\frac12+it_\nu}}V\Bigl(\frac n{\sqrt C}\Bigr)
  =\sum_{n\ge1}\frac{\lambda_f(n)}{n^{\frac12+it_\nu}}V\Bigl(\frac n{\sqrt C}\Bigr)\bigl(1-\omega_{H'_0}(n-N^{(q)}_0)\bigr)\\
  +\sum_{M}\sum_{j\in\Bc(M)}P_j(t_\nu)+\sum_{n\ge1}\frac{\lambda_f(n)}{n^{\frac12+it_\nu}}V\Bigl(\frac n{\sqrt C}\Bigr)\omega_{H'_J}(n-N^{(q)}_J),
\end{multline*}
where
\begin{equation}\label{eq:Pj}
  P_j(t)=\sum_{n\ge1}\frac{\lambda_f(n)}{n^{\frac12+it}}V_t\Bigl(\frac n{\sqrt C}\Bigr)\omega^{(q)}_j(n)\qquad(t\in I_q).
\end{equation}
The last sum is supported on $n\ge N_J^{(q)}-H'_J\gg T^{1+\delta}$, where $V\ll T^{-A\delta}$, and is negligible. So are the blocks with $M\ge\sqrt C\,T^{\delta/20}$, for the same reason, and there are $O(\Lc)$ blocks in all. Hence either
\begin{equation}\label{eq:alt1}
  \Bigl|\sum_{n\ge1}\frac{\lambda_f(n)}{n^{\frac12+it_\nu}}V\Bigl(\frac n{\sqrt C}\Bigr)\bigl(1-\omega_{H'_0}(n-N^{(q)}_0)\bigr)\Bigr|\ \ge\ \frac V{6}
\end{equation}
for $\gg R$ of the points, or there is a block $M$ with
\begin{equation}\label{eq:alt2}
  T^\delta Q/2\le M\le\sqrt C\,T^{\delta/20}\le T^{1+\delta/16},\qquad\Bigl|\sum_{j\in\Bc(M)}P_j(t_\nu)\Bigr|\ \gg\ V\Lc^{-1}
\end{equation}
for $\gg R\Lc^{-1}$ of the points.

\subsection{The initial segment}

\begin{lemma}[Variable endpoints]\label{lem:maximal}
Let $T\ge2$, let $1\le N\le T^2$ be an integer, and let $t_1,\ldots,t_R$ be $1$-spaced
in an interval of length $T$. Suppose $a_n$ is supported on $1\le n\le N$,
$|a_n|\le T^B$ for a fixed $B>0$, and put $G=\sum_n|a_n|^2$.
If for every $\nu$ there is $x_\nu\le N$ with
\[
 \left|\sum_{n\le x_\nu}a_n n^{-it_\nu}\right|\ge W\ge T^{-B},
\]
then, for every $\eta>0$,
\[
 R\ll_{B,\eta}T^\eta\bigl(NGW^{-2}+TNG^3W^{-6}\bigr).
\]
The same conclusion holds for coefficients supported in a subinterval
of $[1,N]$.
\end{lemma}

\begin{proof}
Split $1\le n\le N$ into dyadic intervals. Huxley's large-value theorem
\cite[Lemma 4.3]{Jutila}, applied on each interval and with the resulting
logarithmic factors absorbed into $T^\eta$, gives
\[
 \#\{\nu:|D(t_\nu)|\ge U\}
 \ll_\eta T^\eta\bigl(NGU^{-2}+TNG^3U^{-6}\bigr),
 \qquad D(t)=\sum_{n\le N}a_n n^{-it}.
\]
This estimate is uniform under a common translation of the heights:
the translation can be absorbed into the coefficients without changing $G$.
Order the $R$ values of $|D|$ decreasingly. Applying the displayed
estimate to the $j$-th value and summing $j^{-1/2}$ and $j^{-1/6}$ gives
\begin{equation}\label{eq:huxLone}
 \sum_{\nu\le R}|D(t_\nu+u)|
 \ll_\eta T^\eta\bigl((NG)^{1/2}R^{1/2}
              +(TNG^3)^{1/6}R^{5/6}\bigr),
\end{equation}
uniformly in the common shift $u$.

Replace $x_\nu$ by a half-integer defining the same partial sum.
Set $c=1/\log(2N)$ and $U_0=T^{10B+20}$. Truncated Perron's formula
for the finite polynomial gives
\[
 \sum_{n\le x_\nu}a_n n^{-it_\nu}
 =\frac1{2\pi}\int_{-U_0}^{U_0}
       \sum_{n\le N}a_n n^{-c-i(t_\nu+u)}
       \frac{x_\nu^{c+iu}}{c+iu}\dd u+O(T^{-B-2}).
\]
Indeed $|\log(x_\nu/n)|\gg N^{-1}$, so the error is
$O(T^BN^2/U_0)$. Also $x_\nu^c\ll1$ and
$\int_{-U_0}^{U_0}|c+iu|^{-1}\dd u\ll_B\log T$.
Sum over $\nu$, apply \eqref{eq:huxLone} to $a_nn^{-c}$ (whose squared
norm is at most $G$), and absorb logarithms into $T^\eta$. We obtain
\[
 RW\ll_\eta T^\eta\bigl((NG)^{1/2}R^{1/2}
              +(TNG^3)^{1/6}R^{5/6}\bigr).
\]
One term is at least half the left side. Squaring or taking sixth powers,
and starting with a smaller $\eta$, proves the assertion.
\end{proof}

Suppose \eqref{eq:alt1}. The weight
\[
  w_\nu(n)=V_{t_\nu}(n/\sqrt C)(1-\omega_{H'_0}(n-N_0^{(q)}))
\]
vanishes for $n>N_0^{(q)}+H_0'\ll T^\delta Q$, is $O(1)$, and has total variation $O(1)$ by the bound for $V'$ in Lemma \ref{lem:afe}, so by partial summation there is $u_\nu\ll T^\delta Q$ with $|\sum_{n\le u_\nu}\lambda_f(n)n^{-\frac12-it_\nu}|\gg V$.

Apply Lemma \ref{lem:maximal} with $N\ll T^\delta Q$ and
$a_n=\lambda_f(n)n^{-1/2}$, so that $G\ll_f\log T$ by Rankin--Selberg.
It gives
\[
  R\ \ll\ \bigl(T^\delta QV^{-2}+TT^\delta QV^{-6}\bigr)T^{2\delta}\ \ll\ T^\delta QV^{-2}T^{2\delta},
\]
the second term being dominated by the first because $V\ge T^{1/4+\delta}$. Substituting \eqref{eq:M0},
\[
  R\ \ll\ T^{2/3+3\delta}R^{2/3}V^{-2},\qquad\text{so}\qquad R\ \ll\ T^{2+9\delta}V^{-6},
\]
which is \eqref{eq:target}.

\begin{remark}\label{rem:32}
This is the point at which \cite[(3.2)]{BMN} cannot be used as it stands. That reduction discards the terms $n\le T^\delta Q$ by absolute values, at the cost of $O_{k,N}(\sqrt{T^\delta Q}(\log T^\delta Q)^{-\delta})$, which by \eqref{eq:M0} is $T^{1/3+\delta/2}R^{1/3}$ and so exceeds every $V$ permitted by \eqref{eq:Vrange}. The initial segment must be kept, and the dependence of the order \eqref{eq:Kr} on $R$, which is the one structural difference between the sixth moment and the subconvexity argument, is exactly what makes it deliver the target rather than something weaker.
\end{remark}

There remains the alternative \eqref{eq:alt2}. It is here that the machinery of \cite{BMN} is applied, in place of \cite[Theorem 4.2]{Jutila}, and the rest of the argument is devoted to it. Fix the block $M$ from now on. All constants may depend on $\delta$ but not on $M$.

\section{The error terms}
\label{sec:book}

\begin{proposition}\label{prop:book}
Let $V\ge T^{1/4+\delta}$ and $T^\delta Q\le T^{1-\epsilon}$. Then each of the two error terms of \cite[(3.12)]{BMN}, after division by $\sqrt M$, is either $O(VT^{-\delta})$ or, if not, forces the number $R$ of points $t_\nu\in[T,2T]$ with $|L(\half+it_\nu,f)|\ge V$ to satisfy $R\ll T^{2+6\delta}V^{-6}$.
\end{proposition}

By Proposition \ref{prop:ext}, applied with $\epsilon=\delta/4$ and truncation $Q$ (so that $\sqrt T\ll Q\le T^\delta Q\le T^{1-\delta/4}$ by \eqref{eq:M0} and $M\le T^{1+\delta/8}$ by \eqref{eq:alt2}) and with $s>16/\delta$, each piece $P_j(t)$ with $j\in\Bc(M)$ and $t\in I_q$ is transformed by \cite[Proposition 3.1]{BMN} in the form stated there, the partition being $\{\omega_j^{(q)}\}$, since the hypotheses of that proposition on the partition are \eqref{eq:37}, $v_j\ll K$ and the further properties listed in Appendix \ref{sec:evidence}.

Summing over $j\in\Bc(M)$ as in \cite[(3.16)]{BMN}, the error term is
\[
  E=M^{-1/2}\Bigl(\sqrt M\bigl(M/K^2\bigr)^{\frac1{2(s-1)}}+\frac{M^{5/2}K^2}{H^3}\Bigr)
  =Q^{\frac1{2(s-1)}}+\frac{T^3}{Q^4},
\]
by \eqref{eq:KH}, the logarithmic saving of \cite{BMN} being taken to be $0$. The third term $\max(H,M^{3/5})$ of their (3.16), which is the cost of the sharp endpoints in their (3.9), does not arise, since here every piece is smoothly weighted.

\begin{proof}[Proof of Proposition \ref{prop:book}]
The first term is at most $T^{\frac1{2(s-1)}}\le T^{1/10}$ for $s\ge6$, and is $O(VT^{-\delta})$ by \eqref{eq:Vrange}. The second is the main term. By \eqref{eq:KH} it is $T^3Q^{-4}=T^{1/3}R^{-8/3}$. If $V\ge T^{1/3+\delta}$ this is at most $T^{1/3}\le VT^{-\delta}$. Otherwise either it is $O(VT^{-\delta})$, or else
\[
  R^{8/3}\ll T^{1/3+\delta}V^{-1},
\]
whence
\[
  R\ll T^{1/8+3\delta/8}V^{-3/8};
\]
and, since $V<T^{1/3+\delta}$,
\[
  T^{1/8+3\delta/8}V^{-3/8}=T^{2+3\delta/8}V^{-6}\cdot V^{45/8}T^{-15/8}\ \le\ T^{2+6\delta}V^{-6}.
\]
So in the second case \eqref{eq:target} already holds.
\end{proof}

\begin{remark}
In \cite{Jutila} the corresponding total error is $T^{1/3}R^{-2/3}\Lc^2$ \cite[p.\ 112]{Jutila}, and the term above is smaller by a factor $R^2$.
\end{remark}

We may therefore assume that for $\gg R\Lc^{-1}$ points $t_\nu$, with $t_\nu\in I_q$,
\begin{equation}\label{eq:alt3}
  \Bigl|\sum_{j\in\Bc(M)}S_j(t_\nu)\Bigr|\ \gg\ V\Lc^{-1},
\end{equation}
where $S_j(t)$ is the explicit part of the transform of $P_j(t)$, written out in Section~\ref{sec:transformed}.

\section{The two-height mean-square estimate}
\label{sec:endgame}

We now bound the transformed part left by \eqref{eq:alt3}. First we
separate the fractions so that different phases are distinguishable.
Next we pass from the sampled heights to a discrete mean square of the
individual transformed sums. Proposition \ref{prop:sieve} bounds that
mean square, and Section~\ref{sec:exponent} substitutes $Q=(TR)^{2/3}$ to
finish the proof. This follows Jutila's mean-value argument
\cite[pp.\ 112--122]{Jutila}, with the level-dependent steps written out.

\subsection{The transformed sums}
\label{sec:transformed}

The level contributes only finitely many arithmetic choices. Put
\[
 N^\flat=N\prod_{p\mid N}p^{-1},\qquad
 B(N^\flat)=\{b/N^\flat:0\le b<N^\flat\}.
\]
For $\alpha_j=-u_j/v_j$, set
\[
 d_j=\prod_{\substack{p\mid v_j\\\operatorname{ord}_p(v_j)<\operatorname{ord}_p(N)}}
 p^{\operatorname{ord}_p(v_j)},\qquad q_j=v_j/d_j.
\]
Choose $0\le c_j<d_j$ with $q_jc_j\equiv-u_j\pmod{d_j}$,
and put $a_j=(u_j+c_jq_j)/d_j$. Then
$-u_j/v_j=-a_j/q_j+c_j/d_j$ and $c_j/d_j\in B(N^\flat)$.
For a character $\chi$ modulo $N$, $\bar f^{\chi}$ denotes the primitive
twist of $\bar f$ by $\chi$. Its normalized coefficients, including the
local factors at primes dividing $N$, satisfy Deligne's bound.
These are the arithmetic data in the transformation of \cite{BMN}.

Fix $q$ and $t\in I_q$, and let $j\in\Bc(M)$. Write $\alpha_j=-u_j/v_j=-a_j/q_j+\beta_j$ as in \cite[Section~3.2]{BMN}, with $\beta_j=c_j/d_j\in B(N^\flat)$, $v_j=d_jq_j$, $(a_j,q_j)=1$, and for $\beta=c/d$ and $r\mid NN^\flat$ let $J(\beta,r)$ be the set of $j$ with $\beta_j=\beta$ and $(q_j,r)=1$. Proposition \ref{prop:ext} gives
\begin{multline}\label{eq:transform}
  P_j(t)=\sum_{\beta\in B(N^\flat)}\sum_{\substack{r\mid NN^\flat\\ j\in J(\beta,r)}}\sum_{\chi\ (\mathrm{mod}\ N)}\sum_{\pm}(\mp1)^kc(f,r,\chi;j)\\
  \times\sum_{\ell\ge1}\lambda_{\bar f^{\chi}}(\ell)\,\e\Bigl(\frac{b_j\ell}{q_j}\Bigr)\,\omega^{(q)}_j\bigl(x^\pm_j(\ell/r;t)\bigr)\,\tilde h^\pm_j(\ell/r;t)\,\e\bigl(g^\pm_j(\ell/r;t)\bigr)+E_j(t),
\end{multline}
where $\e(x)=e^{2\pi ix}$, $c(f,r,\chi;j)\ll_N1$, the $E_j$ are the error terms disposed of in Section~\ref{sec:book}, and, for $\beta=c/d$ and $y_j=d/2ru_jq_j$,
\begin{align}
  x^\pm_j(\ell;t)&=\Bigl(\frac d{2u_j}\Bigr)^2\Bigl(\sqrt{\ell+\frac{2tu_jq_j}{\pi d}}\pm\sqrt\ell\Bigr)^2,\label{eq:xj}\\
  g^\pm_j(\ell;t)&=-\frac t{2\pi}\log x^\pm_j(\ell;t)+\frac{u_j}{v_j}x^\pm_j(\ell;t)\mp\frac2{q_j}\sqrt{\ell x^\pm_j(\ell;t)}+\frac18\mp\frac18,\label{eq:gj}\\
  \tilde h^\pm_j(\ell;t)&=h^\pm_j(\ell;t)\,x^\pm_j(\ell;t)^{-1/2}\,V_t\bigl(x^\pm_j(\ell;t)/\sqrt C\bigr),\label{eq:hj}\\
  h^\pm_j(\ell;t)&=\Bigl(\frac{q_jt\sqrt\ell}{\pi x^\pm_j(\ell;t)^{3/2}}\pm\frac\ell{x^\pm_j(\ell;t)}\Bigr)^{-1/2},\notag
\end{align}
these being \cite[(3.13)--(3.15)]{BMN} with the dependence on $t$ made explicit and the amplitude of Proposition \ref{prop:ext} inserted. The summand in \eqref{eq:transform} vanishes unless $x^\pm_j(\ell/r;t)$ lies in the support of $\omega^{(q)}_j$, which by the choice of $K_1$ in \cite[Section~4.2, Step 3]{BMN} forces $\ell\le K_1rd^{-2}\ll Q$.

Here $a_j=(u_j+cq_j)/d$ is the residue in \cite[(3.11)]{BMN}, and $b_j$ is the inverse of $ra_j$ modulo $q_j$, with $0\le b_j<q_j$. This is the additive twist in their transformation formula. Since $(r,q_j)=1$, equality of two such inverses at a fixed denominator is equivalent to equality of the original residues; with $c,d$ fixed, the latter is equivalent to equality of $u_j$ modulo $v_j$.

Put $g^\pm_{jr}(\ell;t)=g^\pm_j(\ell/r;t)+b_j\ell/q_j$. From \cite[Sections~5.1 and~5.2]{BMN}, whose computation is for fixed $t$ and is unchanged,
\begin{equation}\label{eq:dl}
  \frac{\partial}{\partial\ell}g^\pm_{jr}(\ell;t)=\frac{b_j}{q_j}-y_j\mp F_j(\ell;t),\qquad F_j(\ell;t)=\sqrt{y_j^2+\frac{ty_j}{\pi\ell}},
\end{equation}
and since $x^\pm_j(\ell;t)$ is a stationary point of the phase whose value is $g^\pm_j(\ell;t)$,
\begin{equation}\label{eq:dt}
  \frac{\partial}{\partial t}g^\pm_{jr}(\ell;t)=-\frac1{2\pi}\log x^\pm_j(\ell/r;t).
\end{equation}
For $j\in\Bc(M)$ with $v_j\asymp\Kc$ one has $u_j/v_j\asymp T/M$, hence $u_jv_j\asymp\Kc^2T/M$ and $y_j\asymp M/\Kc^2T$, with constants depending on $N$; and for $\ell\ll Q$,
\begin{equation}\label{eq:theta}
  \frac{\pi\ell y_j}{t}\ \ll\ \frac{QM}{\Kc^2T^2}=\frac{M^2}{K^2\Kc^2T^2}\ \ll\ T^{-\delta/2}
\end{equation}
by \eqref{eq:R2} and \eqref{eq:alt2}, since
\[
  M^2/(K^2T^2)=MT^{2/3}R^{2/3}T^{-2}\le T^{-4\delta/3+\delta/16}\Lc^{2},
\]
so that $F_j(\ell;t)\asymp\sqrt{ty_j/\ell}$, and by \eqref{eq:xj}
\[
  x^\pm_j(\ell/r;t)=h_t(\alpha_j)\bigl(1+O(\sqrt{\pi\ell y_j/t})\bigr),
\]
the relative error being $O(M/K\Kc T)=O(T^{-\delta/4})$.

Finally, $\tilde h^\pm_j(\ell/r;t)=\ell^{-1/4}\kappa_j(\ell;t)$ where, by the computation of $h^\pm_{jr}$ in \cite[Section~5]{BMN} and the bound $V_t\ll1$,
\begin{equation}\label{eq:kappa}
  |\kappa_j(\ell;t)|\ \ll\ \Kc^{-1/2}M^{1/4}T^{-1/2},\qquad \Bigl|\frac{\partial}{\partial\ell}\kappa_j(\ell;t)\Bigr|\ll\frac{\Kc^{-1/2}M^{1/4}T^{-1/2}}{\ell},
\end{equation}
the second because every factor of $\kappa_j(\ell;t)$ other than the one coming from $V_t(y)$ has logarithmic derivative $O(1/\ell)$ in $\ell$, while $\partial_\ell[V_t(x^\pm_j(\ell/r;t)/\sqrt C)]\ll M^{-1}\partial_\ell x^\pm_j\ll1/\ell$ by Cauchy's estimate for $V_t(y)$ on a disc of radius $\asymp M$ in the sector.

Moreover $\kappa_j(\ell;t)$, like $g^\pm_{jr}(\ell;t)$, extends to a holomorphic function of $t$ in $|\Imag t|\le T^{1/2}$ satisfying the same bounds, being built from $\sqrt{\ell+2tu_jq_j/\pi d}$, powers, the ratio of gamma factors in the truncated integral $V_t(y)$ of Lemma \ref{lem:afe}, and $\sqrt{C(t)}$ continued as
\[
  \pi^{-1}\sqrt N\bigl((\tfrac{k+1}2)^2+t^2\bigr)^{1/4}\bigl((\tfrac{k+3}2)^2+t^2\bigr)^{1/4},
\]
none of which has a singularity in that strip. The absent factor $M^{-1/2}$ of Jutila's normalisation $n^{-1/2-it}$ is what makes \eqref{eq:kappa} agree with his amplitude $(hk)^{-1/4}t^{-1/2}\asymp\Kc^{-1/2}M^{1/4}T^{-1/2}$.

\subsection{Separating the Farey fractions}\label{sec:separate}

The sum over $j\in\Bc(M)$ in \eqref{eq:alt3} is split into classes. There are $O_N(1)$ choices of $(\beta,r,\chi,\pm)$. The denominators $v_j$ are put into $O(\Lc)$ dyadic ranges $[\Kc,2\Kc)$. The fractions with $v_j$ in such a range, which are at least $(2\Kc)^{-2}$ apart, are listed in order and put into $\lceil T^\delta\rceil$ classes by the residue of their position, so that within a class $|\alpha_i-\alpha_j|\ge\Kc^{-2}T^\delta$. The integers $u_jv_j$, which lie in an interval of length $\ll\Kc^2T/M$, are put into $O(1+M/T)$ subintervals of length $c_1\Kc^2T^2/M^2$, with $c_1=c_1(N)$ small, and then into $\ll T^\delta$ further classes in which they are distinct, which is possible because an integer $m\le T^2$ has $d(m)\ll T^\delta$ divisors.

Altogether there are $\ll T^{3\delta}$ classes. By the pigeonhole principle there is a class $\Jc$, with its $(\beta,r,\chi,\pm)$, such that
\begin{equation}\label{eq:thin}
  \Bigl|\sum_{j\in\Jc}S_j(t_\nu)\Bigr|\ \gg\ VT^{-3\delta}
\end{equation}
for a subset $\mathcal T$ of at least $\gg RT^{-3\delta}$ of the points, where now $S_j(t)$ denotes the single term of \eqref{eq:transform} belonging to that $(\beta,r,\chi,\pm)$. We record the properties of $\Jc$.
\begin{enumerate}
\item[(S0)] $\Jc\subset J(\beta,r)\cap\Bc(M)$ for one $\beta=c/d$ and one $r$, so $d$ and $r$ are the same for all $j\in\Jc$;
\item[(S1)] $\Kc\le v_j<2\Kc$, and $|\alpha_i-\alpha_j|\ge\Kc^{-2}T^\delta$ for $i\ne j$ in $\Jc$;
\item[(S2)] the integers $u_jv_j$, $j\in\Jc$, are distinct and lie in an interval of length at most $c_1\Kc^2T^2/M^2$.
\end{enumerate}
From now on absorb the fixed factor $(\mp1)^kc(f,r,\chi;j)$ into $\kappa_j$; its bounds retain implied constants depending on $N$. Here $\Kc\ll K$, and $\#\Jc\ll\Kc^2T/M+1$ by (S1). (S1) is Jutila's (4.4.33). (S2) is his (4.4.34), the length of the interval being what makes the correction $y_i-y_j$ in \eqref{eq:dl} smaller than the separation of the rationals $b_i/q_i$, Lemma \ref{lem:sep} below.

\subsection{From the points to a discrete mean square}
\label{sec:passage}

Let
\begin{equation}\label{eq:Z}
  Z=\max\bigl(\Kc^2T/4M,\,1\bigr),
\end{equation}
which is at most $Z_*$ since $\Kc\le v_j\le K(u_j/v_j)<2K$ and $Z_*=T^{1/3}R^{-2/3}\ge1$.

Divide each $I_q$ into intervals of length $Z$, the last possibly
shorter, and retain every interval meeting $\mathcal T$. There are
$P\le\#\mathcal T\le R$ such intervals. No assumption on the number
of points in one interval is needed.

Fix an occupied interval $[t_0,t_0+Z)\subset I_q$ (or its shortened last
piece). For $|\Real t-t_0|\le5Z$ and $|\Imag t|\le Z$, define
\begin{equation}\label{eq:frozen}
  \widetilde S_j(t)=\sum_{\ell\ge1}\lambda_{\bar f^{\chi}}(\ell)\ell^{-1/4}
  w_j(\ell)\kappa_j(\ell;t)\e\bigl(g^\pm_{jr}(\ell;t)\bigr),
  \qquad w_j(\ell)=\omega^{(q)}_j\bigl(x^\pm_j(\ell/r;t_0)\bigr).
\end{equation}
Only the cutoff is frozen at $t_0$; the analytic amplitude and phase
still depend on $t$.

For $t_\nu$ in the interval,
\[
  |x^\pm_j(\ell/r;t_\nu)-x^\pm_j(\ell/r;t_0)|\ll ZM/T\ll\Kc^2+T^{\delta/8},
\]
so by \eqref{eq:KH2}
\[
  |w_j(\ell)-\omega^{(q)}_j(x^\pm_j(\ell/r;t_\nu))|\ll H^{-1}(\Kc^2+T^{\delta/8})\ll T^{-1/3+\delta/8},
\]
since $H\ge Q^2/2T\ge\frac12T^{1/3}$, and only for $\ell\ll Q$. Hence by \eqref{eq:kappa} and $\#\Jc\ll\Kc^2T/M+1$
\begin{multline*}
  \Bigl|\sum_{j\in\Jc}S_j(t_\nu)-\sum_{j\in\Jc}\widetilde S_j(t_\nu)\Bigr|\ \ll\ T^{-1/3+\delta/8}\,\#\Jc\,Q^{3/4}\Kc^{-1/2}M^{1/4}T^{-1/2+\delta}\\ \ll\ T^{-1/3+9\delta/8}\bigl(\Kc^{3/2}M^{-3/4}T^{1/2}+\Kc^{-1/2}M^{1/4}T^{-1/2}\bigr)Q^{3/4}\ \ll\ T^{1/6+2\delta},
\end{multline*}
using $\Kc\ll K$, $Q=M/K^2$ and $M\le T^{1+\delta/8}$. This is $o(VT^{-3\delta})$ by \eqref{eq:Vrange}, so \eqref{eq:thin} holds with $\widetilde S_j(t)$ in place of $S_j(t)$. We freeze the weights because $\widetilde S_j(t)$ is holomorphic in $t$, whereas $\omega^{(q)}_j(x^\pm_j(\ell/r;t))$ is not.

Gallagher's inequality, applied to the $1$-spaced points of
$\mathcal T$ in this interval, gives
\begin{multline}\label{eq:gallagher}
  \sum_\nu\Bigl|\sum_{j\in\Jc}\widetilde S_j(t_\nu)\Bigr|^2\ \ll\ \int_{-1}^{Z+1}\Bigl|\sum_{j\in\Jc}\widetilde S_j(t_0+u)\Bigr|^2\dd u\\
  +\Bigl(\int_{-1}^{Z+1}\Bigl|\sum_j\widetilde S_j\Bigr|^2\dd u\Bigr)^{1/2}\Bigl(\int_{-1}^{Z+1}\Bigl|\sum_j\widetilde S_j'\Bigr|^2\dd u\Bigr)^{1/2},
\end{multline}
where $'$ is $d/dt$.

Choose a smooth function $0\le\eta\le1$ equal to $1$ on
$[-1,Z+1]$, supported in $(-Z-1,2Z+1)$, and satisfying
$\eta^{(a)}\ll_a Z^{-a}$ for each integer $a\ge0$.
Then the first integral is at most
\[
  \int\eta(u)|\sum_j\widetilde S_j(t_0+u)|^2\dd u=\sum_{i,j\in\Jc}\int\eta\,\widetilde S_i\overline{\widetilde S_j}\dd u,
\]
and we dispose of the terms $i\ne j$.

\begin{lemma}\label{lem:offdiag}
For $i\ne j$ in $\Jc$, $\ell,\ell'\ll Q$ and real $u$ with $|u|\le5Z$, the function
\[
  f(u)=g^\pm_{ir}(\ell;t_0+u)-g^\pm_{jr}(\ell';t_0+u)
\]
satisfies
\[
  |f'(u)|\asymp|\log(\alpha_i/\alpha_j)|\gg T^\delta/Z,
\]
and this persists for complex $u$ with $|\Real u|\le5Z$ and $|\Imag u|\le Z$.
\end{lemma}

\begin{proof}
By \eqref{eq:dt},
\[
  f'(u)=-\frac1{2\pi}\log\bigl(x^\pm_i(\ell/r;t)/x^\pm_j(\ell'/r;t)\bigr)
\]
with $t=t_0+u$. By \eqref{eq:xj}, $x^\pm_j(\ell/r;t)=h_t(\alpha_j)(1+O(\sqrt{\pi\ell y_j/t}))$, and by \eqref{eq:theta} the relative error is $O(M/K\Kc T)$. So
\[
  f'(u)=\frac1{2\pi}\log(\alpha_i/\alpha_j)+O(M/K\Kc T).
\]

Since $|\alpha_i|\asymp T/M$ and $|\alpha_i-\alpha_j|\ll T/M$,
\[
  |\log(\alpha_i/\alpha_j)|\asymp|\alpha_i-\alpha_j|M/T\ge\Kc^{-2}T^\delta M/T
\]
by (S1), which is $\ge T^\delta/4Z$ by \eqref{eq:Z}; and the error is smaller than this by the factor $\Kc/KT^\delta\ll T^{-\delta}$. For complex $u$ with $|\Imag u|\le Z$ the function $\log(x^\pm_i/x^\pm_j)$ is holomorphic with derivative $O(1/T)$ in $t$, so it changes by $O(Z/T)$, which is $o(T^\delta/Z)$ because $Z\le Z_*=T^{1/3}R^{-2/3}\le T^{1/3}$.
\end{proof}

Consider an off-diagonal term. Expanding both sums it is a sum over $\ell,\ell'\ll Q$ of terms
\[
  \lambda_{\bar f^{\chi}}(\ell)\overline{\lambda_{\bar f^{\chi}}(\ell')}(\ell\ell')^{-1/4}w_i(\ell)w_j(\ell')\int\eta(u)G(u)\e(f(u))\dd u,
\]
with $f$ as in Lemma \ref{lem:offdiag} and
\[
  G(u)=\kappa_i(\ell;t_0+u)\kappa^*_j(\ell';t_0+u),
\]
where $\kappa^*_j(\ell;t)=\overline{\kappa_j(\ell;\bar t)}$ is holomorphic and agrees with $\overline{\kappa_j(\ell;t)}$ for real $t$. Here $G(u)$ is holomorphic for $|\Real u|\le5Z$, $|\Imag u|\le Z$, and $O(\Kc^{-1}M^{1/2}T^{-1})$ there by \eqref{eq:kappa}, and the phase $f$ is real for real $u$. On the real axis the analytic continuation agrees with the original conjugate product.

For completeness, the required non-stationary-phase bound follows by
integration by parts. Put $A_0=\Kc^{-1}M^{1/2}T^{-1}$ and
$\lambda=|\log(\alpha_i/\alpha_j)|$. Lemma \ref{lem:offdiag} and
Cauchy's estimate on a slightly smaller complex neighbourhood give
\[
 |G^{(a)}|\ll_a A_0Z^{-a},\qquad
 |(1/f')^{(a)}|\ll_a\lambda^{-1}Z^{-a}.
\]
Integrating $J$ times using $\e(f)=(2\pi i f')^{-1}(\e(f))'$,
with no boundary terms because $\eta$ is smooth and compactly supported,
therefore gives
\[
 \left|\int\eta(u)G(u)\e(f(u))\dd u\right|
 \ll_J A_0Z(\lambda Z)^{-J}
 \ll_J A_0Z T^{-J\delta}.
\]
Taking $J=\lceil20/\delta\rceil$ makes this $O(T^{-10})$.

Summing over $\ell,\ell'$ and over $i\ne j$ in $\Jc$, the off-diagonal contributes $O(T^{-3})$, which is negligible, and the same applies to the integral involving $\widetilde S'_j$ in \eqref{eq:gallagher}, whose differentiated amplitude is $\partial_t\kappa_j+2\pi i(\partial_tg^\pm_{jr})\kappa_j$. Cauchy's estimate on a slightly larger strip and \eqref{eq:dt} bound it by $O(\Kc^{-1/2}M^{1/4}T^{-1/2}\Lc)$, without dividing by $\kappa_j$. Hence
\begin{multline*}
  \int_{-1}^{Z+1}\Bigl|\sum_{j\in\Jc}\widetilde S_j(t_0+u)\Bigr|^2\dd u\ \ll\ \sum_{j\in\Jc}\int_{-Z-1}^{2Z+1}|\widetilde S_j(t_0+u)|^2\dd u+T^{-3}\\
  \le\ 3(Z+1)\max_{|t-t_0|\le3Z}\sum_{j\in\Jc}|\widetilde S_j(t)|^2+T^{-3},
\end{multline*}
and likewise for $\widetilde S'_j$. Choose maximising points $t_p,t_p'$
for each occupied interval. Summing \eqref{eq:gallagher} over those
intervals, and applying Cauchy's inequality to the mixed terms, gives
\[
 RV^2T^{-9\delta}\ll Z\Bigl\{
 \sum_{p,j}|\widetilde S_j(t_p)|^2
 +\Bigl(\sum_{p,j}|\widetilde S_j(t_p)|^2\Bigr)^{1/2}
  \Bigl(\sum_{p,j}|\widetilde S'_j(t_p')|^2\Bigr)^{1/2}\Bigr\}.
\]
Here the left side uses $\#\mathcal T\gg RT^{-3\delta}$ and
\eqref{eq:thin}; the negligible errors have been absorbed.
Since $Z\ll\Kc^2TM^{-1}T^{\delta/8}$, taking square roots gives
\begin{multline}\label{eq:4441}
 RV\ll R^{1/2}\Kc M^{-1/2}T^{1/2+5\delta}
 \Bigl\{\sum_{p,j}|\widetilde S_j(t_p)|^2\\
 +\Bigl(\sum_{p,j}|\widetilde S_j(t_p)|^2\Bigr)^{1/2}
  \Bigl(\sum_{p,j}|\widetilde S'_j(t_p')|^2\Bigr)^{1/2}\Bigr\}^{1/2}.
\end{multline}
The sums run over $j\in\Jc$ and the $P\le R$ occupied intervals.
The selected heights lie in $[T/2,3T]$.

Each $t_p$ lies within $3Z$ of the left end of its interval, and each $I_q$ contains at most one interval shorter than $Z$, so if the intervals are split into sixteen classes according to the residue modulo $16$ of their position in $[T,2T]$, the $t_p$ belonging to one class satisfy $|t_p-t_{p'}|\ge Z$ for $p\ne p'$, and similarly the $t'_p$. We apply what follows to each class separately and add, at the cost of a constant. In each $\widetilde S_j(t_p)$ the weight $w_j(\ell)$ is the one frozen at the left end of the interval containing $t_p$, so it depends on $p$ as well as on $j$. Nothing below uses more than that it is bounded and of bounded variation.
\subsection{The two-height large sieve}
\label{sec:sieve}

\begin{proposition}\label{prop:sieve}
Let $\Jc$ satisfy \textup{(S0)--(S2)}, let $Z$ be as in \eqref{eq:Z}, and let $t_1,\dots,t_P\in[T/2,3T]$ satisfy $|t_p-t_{p'}|\ge Z$ for $p\ne p'$. Let $1\le L\ll Q$, let $A>0$, and for each $p$ and $j$ let $\kappa_{p,j}(\ell)$, $L\le\ell\le2L$, be complex numbers with $|\kappa_{p,j}(\ell)|\le A$ and $\sum_{\ell}|\kappa_{p,j}(\ell+1)-\kappa_{p,j}(\ell)|\le A$. Then
\[
  \sum_{p\le P}\sum_{j\in\Jc}\Bigl|\sum_{L\le\ell\le2L}\lambda_{\bar f^{\chi}}(\ell)\ell^{-1/4}\kappa_{p,j}(\ell)\,\e\bigl(g^\pm_{jr}(\ell;t_p)\bigr)\Bigr|^2
  \ \ll\ A^2L^{1/2}T^{\delta}\bigl(\Kc^{-1}K^{-1}M+KM^{-1/2}PT\bigr).
\]
\end{proposition}

This is \cite[(4.4.44)--(4.4.46)]{Jutila} at arbitrary level. Large-value estimates over pairs formed by a height and a character or modulus appear also in Meurman's work on the twelfth moment of Dirichlet $L$-functions \cite{Meurman84}, and in the spectral approach of Jutila and Motohashi \cite{JutilaMotohashi}. The proof occupies the rest of the subsection. First the two derivatives.

\begin{lemma}\label{lem:deriv}
Let $i,j\in\Jc$, $t_p,t_{p'}\in[T/2,3T]$, and put
\[
  \Phi(\ell)=g^\pm_{ir}(\ell;t_p)-g^\pm_{jr}(\ell;t_{p'}),\qquad A_j(t)=u_jq_j/t=u_jv_j/dt,
\]
and $F_i=F_i(\ell;t_p)$, $F_j=F_j(\ell;t_{p'})$ as in \eqref{eq:dl}. Then
\begin{align}
  \Phi'(\ell)&=z_{ij}\mp\bigl(F_i-F_j\bigr),\qquad z_{ij}=\frac{b_i}{q_i}-\frac{b_j}{q_j}-(y_i-y_j),\label{eq:Phi1}\\
  F_i-F_j&=\frac{d\,t_pt_{p'}}{2\pi r\ell\,u_iq_iu_jq_j\,(F_i+F_j)}\Bigl[A_j(t_{p'})-A_i(t_p)+\frac{\pi d\ell}{2rt_pt_{p'}}\Bigl(\frac{u_jq_j}{u_iq_i}-\frac{u_iq_i}{u_jq_j}\Bigr)\Bigr],\label{eq:Phi1b}\\
  \Phi''(\ell)&=\pm\frac{r\,(t_pt_{p'})^2y_i^2y_j^2}{\pi^2d\,\ell^3\,F_iF_j\,(t_py_iF_j+t_{p'}y_jF_i)}\Bigl[A_j(t_{p'})-A_i(t_p)+\frac{\pi d\ell}{2r}\bigl(t_{p'}^{-2}-t_p^{-2}\bigr)\Bigr].\label{eq:Phi2}
\end{align}
For $L\le\ell\le2L$ with $L\ll Q$, the positive prefactors in \eqref{eq:Phi1b} and \eqref{eq:Phi2} are $\asymp_N\Kc^{-3}M^{3/2}L^{-1/2}$ and $\asymp_N\Kc^{-3}M^{3/2}L^{-3/2}$ respectively, and each is a smooth function of $\ell$ whose logarithmic derivative is $O(1/\ell)$.
\end{lemma}

\begin{proof}
\eqref{eq:Phi1} is \eqref{eq:dl}. For \eqref{eq:Phi1b}, $F_i-F_j=(F_i^2-F_j^2)/(F_i+F_j)$ and
\begin{align*}
  F_i^2-F_j^2&=(y_i^2-y_j^2)+\frac1{\pi\ell}\bigl(t_py_i-t_{p'}y_j\bigr),\\
  t_py_i-t_{p'}y_j&=\frac d{2r}\Bigl(\frac{t_p}{u_iq_i}-\frac{t_{p'}}{u_jq_j}\Bigr)=\frac{d\,t_pt_{p'}}{2r\,u_iq_iu_jq_j}\bigl(A_j(t_{p'})-A_i(t_p)\bigr),
\end{align*}
while
\[
  y_i^2-y_j^2=\frac{d^2}{4r^2}\cdot\frac{(u_jq_j)^2-(u_iq_i)^2}{(u_iq_iu_jq_j)^2},
\]
which is the same factor $\frac{d\,t_pt_{p'}}{2\pi r\ell\,u_iq_iu_jq_j}$ times $\frac{\pi d\ell}{2rt_pt_{p'}}\bigl(\frac{u_jq_j}{u_iq_i}-\frac{u_iq_i}{u_jq_j}\bigr)$.

For \eqref{eq:Phi2}, $\partial_\ell F=-ty/2\pi\ell^2F$, so
\[
  \Phi''=\pm(t_py_iF_j-t_{p'}y_jF_i)/2\pi\ell^2F_iF_j,
\]
and
\begin{align*}
  (t_py_iF_j)^2-(t_{p'}y_jF_i)^2&=y_i^2y_j^2(t_p^2-t_{p'}^2)+\frac{t_pt_{p'}y_iy_j}{\pi\ell}\bigl(t_py_i-t_{p'}y_j\bigr)\\
  &=\frac{2r(t_pt_{p'})^2y_i^2y_j^2}{\pi d\ell}\Bigl[A_j(t_{p'})-A_i(t_p)+\frac{\pi d\ell}{2r}\bigl(t_{p'}^{-2}-t_p^{-2}\bigr)\Bigr],
\end{align*}
using $d/2ru_iq_iu_jq_j=2ry_iy_j/d$. Dividing by $t_py_iF_j+t_{p'}y_jF_i$ gives \eqref{eq:Phi2}.

For the sizes, $u_jq_j\asymp\Kc^2T/dM$, $y_j\asymp d^2M/r\Kc^2T$ and, by \eqref{eq:theta}, $F_j\asymp\sqrt{Ty_j/L}$ for both heights. Substituting gives the two orders of magnitude, and each factor is a power of $\ell$, of $F_i$, of $F_j$ or of $t_py_iF_j+t_{p'}y_jF_i$, all of which have logarithmic derivative $O(1/\ell)$.
\end{proof}

These are Jutila's (4.4.47) and (4.4.48), with $u_jq_j=u_jv_j/d$ in place of $hk$ and $b_j/q_j$ in place of $\bar h/k$. His correction terms are printed with the factor $1/2\pi x$ where \eqref{eq:Phi1b} and \eqref{eq:Phi2} have $\pi\ell/2$, and nothing depends on this. The level enters only through $d$ and $r$. The bracket in \eqref{eq:Phi2} is linear in $\ell$, so $\Phi''$ changes sign at most once on any interval, and $\Phi'$ is monotone on each of at most two subintervals.

\begin{lemma}\label{lem:sep}
For $i\ne j$ in $\Jc$,
\[
  \|z_{ij}\|\gg_N\Kc^{-2}\min(1,M^2T^{-2})\ge\Kc^{-2}M^2T^{-2-\delta/4}.
\]
\end{lemma}

\begin{proof}
By (S2),
\[
  |y_i-y_j|=\frac{d^2|u_iv_i-u_jv_j|}{2r\,u_iv_iu_jv_j}\le\frac{c_1d^2}{2r}\cdot\frac{\Kc^2T^2/M^2}{(\Kc^2T/4\pi M)^2}=\frac{8\pi^2c_1d^2}{r\Kc^2},
\]
using $u_jv_j=v_j^2(u_j/v_j)>\Kc^2T/4\pi M$ for $j\in\Bc(M)$. With $c_1$ small this is at most $d^2/8\Kc^2$, and $d\le v_j<2\Kc$.

If $q_i\ne q_j$ then $b_i/q_i-b_j/q_j$ is a non-zero rational with denominator dividing $q_iq_j$, since $(b_i,q_i)=(b_j,q_j)=1$, so $\|b_i/q_i-b_j/q_j\|\ge1/q_iq_j\ge d^2/4\Kc^2$ and $\|z_{ij}\|\ge d^2/8\Kc^2$. If $q_i=q_j=q$ but $a_i\not\equiv a_j\pmod q$ then $\|(b_i-b_j)/q\|\ge1/q=d/v_i\ge d/2\Kc$ and $\|z_{ij}\|\ge d/4\Kc$.

If $q_i=q_j$ and $a_i\equiv a_j\pmod q$, then $u_i\equiv u_j\pmod{v_i}$, so $|u_i-u_j|\ge v_i\ge\Kc$ and, as the rational part is an integer,
\[
  \|z_{ij}\|=|y_i-y_j|=\frac{d^2|u_i-u_j|}{2r\,u_iu_jv_i}\ge\frac{d^2}{2ru_iu_j}\gg\frac{d^2M^2}{r\Kc^2T^2},
\]
using $u_j\asymp\Kc T/M$ and $|y_i-y_j|\le d^2/8\Kc^2<\frac12$. In every case $\|z_{ij}\|\gg\Kc^{-2}\min(1,M^2/T^2)$, and $M\le T^{1+\delta/8}$ gives the last inequality.
\end{proof}

\begin{proof}[Proof of Proposition \ref{prop:sieve}]
Let
\[
  \xi=(\lambda_{\bar f^{\chi}}(\ell)\ell^{-1/4})_{L\le\ell\le2L}
\]
and
\[
  \varphi_{p,j}=(\kappa_{p,j}(\ell)\e(g^\pm_{jr}(\ell;t_p)))_{L\le\ell\le2L},
\]
so that the inner sum in the proposition is $(\xi,\varphi_{p,j})$. Bombieri's form of Hal\'asz's inequality \cite[Lemma 4.5]{Jutila} gives
\[
  \sum_{p,j}|(\xi,\varphi_{p,j})|^2\ \le\ \|\xi\|^2\max_{(p',j')}\sum_{(p,j)}|(\varphi_{p,j},\varphi_{p',j'})|,\qquad \|\xi\|^2=\sum_{L\le\ell\le2L}|\lambda_{\bar f^{\chi}}(\ell)|^2\ell^{-1/2}\ll L^{1/2}T^{\delta/2},
\]
by Deligne's bound.

Fix $(p',j')$ and put $\Phi=\Phi_{p,j}=g^\pm_{jr}(\cdot\,;t_p)-g^\pm_{j'r}(\cdot\,;t_{p'})$. The coefficient $\kappa_{p,j}(\ell)\overline{\kappa_{p',j'}(\ell)}$ is bounded by $A^2$ with total variation at most $2A^2$, so by partial summation
\[
  |(\varphi_{p,j},\varphi_{p',j'})|\ \le\ 3A^2\Delta(p,j),\qquad \Delta(p,j)=\max_{L\le L_1\le L_2\le2L}\Bigl|\sum_{L_1\le\ell\le L_2}\e(\Phi_{p,j}(\ell))\Bigr|,
\]
and it suffices to prove
\begin{equation}\label{eq:4446}
  \sum_{(p,j)}\Delta(p,j)\ \ll\ T^{\delta/2}\bigl(\Kc^{-1}K^{-1}M+KM^{-1/2}PT\bigr).
\end{equation}
We use two estimates for exponential sums, namely the second derivative estimate \cite[Lemma 4.1]{Jutila}, that if $\lambda_2\le|\Phi''|\le h\lambda_2$ on $[L_1,L_2]$ then
\[
  \sum\e(\Phi(\ell))\ll h(L_2-L_1)\lambda_2^{1/2}+\lambda_2^{-1/2},
\]
and the first derivative estimate \cite[Lemma 4.6]{Jutila}, that if $\Phi'$ is monotone and $\|\Phi'\|\ge m>0$ then
\[
  \sum\e(\Phi(\ell))\ll m^{-1}.
\]
The latter is applied on each of the at most two intervals of monotonicity.

Lemma \ref{lem:deriv} is used with $(i,j)\to(j,j')$ and the heights $(t_p,t_{p'})$, and we write $B_1(\ell)$, $B_2(\ell)$ for the brackets in \eqref{eq:Phi1b}, \eqref{eq:Phi2} and
\[
  \Lambda:=A_j(t_p)-A_{j'}(t_{p'})
\]
for their common leading term, so that $|B_1|,|B_2|$ are $|\Lambda|$ up to the corrections. Throughout, $L\ll Q=M/K^2$, $\Kc\ll K$, $K\gg T^{\delta/2}$, $K\le M^{1/2}T^{-1/3}$, $T^{2/3}\le M\le T^{1+\delta/8}$ and $|t_p-t_{p'}|\le3T$.

\smallskip\noindent\emph{1. The diagonal pair.} We first treat $(p',j')$ itself. Trivially $\Delta(p',j')\le L\ll M/K^2\le\Kc^{-1}K^{-1}M$.

\smallskip\noindent\emph{2. The same fraction at different heights.} Let $j=j'$ and $p\ne p'$, so that the fraction is the same and the heights differ. Here $z_{jj}=0$, and in \eqref{eq:Phi1b} and \eqref{eq:Phi2} the corrections vanish or have the sign of the leading term $-\Lambda$, where $\Lambda=u_jq_j(t_p^{-1}-t_{p'}^{-1})$, so $\Phi''$ has constant sign, $\Phi'$ is monotone, and by Lemma \ref{lem:deriv}
\begin{align*}
  |\Phi'(\ell)|&\asymp\Kc^{-3}M^{3/2}L^{-1/2}\cdot\frac{\Kc^2T}{M}\cdot\frac{|t_p-t_{p'}|}{T^2}=\Kc^{-1}M^{1/2}L^{-1/2}T^{-1}|t_p-t_{p'}|,\\
  |\Phi''(\ell)|&\asymp\Kc^{-1}M^{1/2}L^{-3/2}T^{-1}|t_p-t_{p'}|.
\end{align*}
If $|t_p-t_{p'}|\le c\Kc M^{-1/2}L^{1/2}T$ with $c$ small then $|\Phi'|\le\frac12$, so $\|\Phi'\|=|\Phi'|$ and the first derivative estimate gives
\[
  \Delta(p,j')\ll\Kc M^{-1/2}L^{1/2}T|t_p-t_{p'}|^{-1}.
\]
Since the $t_p$ are $Z$-spaced,
\[
  \sum_{p\ne p'}|t_p-t_{p'}|^{-1}\ll Z^{-1}\Lc\le\Kc^{-2}MT^{-1}\Lc,
\]
and these $p$ contribute $\ll\Kc^{-1}M^{1/2}L^{1/2}\Lc\ll\Kc^{-1}K^{-1}M\Lc$.

Otherwise $|\Phi''|\gg L^{-1}$ and $|\Phi''|\ll\Kc^{-1}M^{1/2}L^{-3/2}$, with $|\Phi''|$ varying by a bounded factor on $[L,2L]$, and the second derivative estimate gives
\[
  \Delta(p,j')\ll L(\Kc^{-1}M^{1/2}L^{-3/2})^{1/2}+L^{1/2}\ll M^{1/2}.
\]
These $p$ contribute $\ll PM^{1/2}\le PKM^{-1/2}T$, since $M\le KT$.

\smallskip\noindent\emph{3. Different fractions: the non-exceptional terms.} Let $j\ne j'$ and fix $p$. Then
\[
  \Lambda=\frac1{dt_p}\bigl(u_jv_j-u_{j'}v_{j'}t_p/t_{p'}\bigr),
\]
and as $j$ varies in $\Jc$ the integers $u_jv_j$ are distinct by (S2), so at most one $j$, the \emph{exceptional} one, has $|u_jv_j-u_{j'}v_{j'}t_p/t_{p'}|<\frac12$. For the others $|\Lambda|\ge1/2dt_p\ge1/6dT$, and always $|\Lambda|\ll\Kc^2/M$.

The corrections are small. In $B_2$ the correction is $\ll dL/rT^2$, which is $\ll T^{-\delta/4}/dT$ since $L\ll Q\le T^{1-\delta/4}$. In $B_1$ it is
\[
  \frac{\pi d\ell}{2rt_pt_{p'}}\cdot\frac{|(u_jv_j)^2-(u_{j'}v_{j'})^2|}{u_jv_ju_{j'}v_{j'}},
\]
and by (S2) $|u_jv_j-u_{j'}v_{j'}|\le c_1\Kc^2T^2/M^2$ while $u_jv_j\asymp\Kc^2T/M$, so it is
\[
  \ll dL/rTM\ll d^2L/rM\cdot(1/dT)\ll(d^2/rK^2)(1/dT).
\]
Hence for non-exceptional $j$, $|B_1|\asymp|B_2|\asymp|\Lambda|$.

Consider first the non-exceptional $j$. By Lemma \ref{lem:deriv}, $|\Phi''|\asymp\Kc^{-3}M^{3/2}L^{-3/2}|\Lambda|$, varying by a bounded factor on $[L,2L]$, and the second derivative estimate gives
\[
  \Delta(p,j)\ll L\Kc^{-3/2}M^{3/4}L^{-3/4}|\Lambda|^{1/2}+\Kc^{3/2}M^{-3/4}L^{3/4}|\Lambda|^{-1/2}.
\]

Since $|\Lambda|=|u_jv_j-c|/dt_p$ with the $u_jv_j$ distinct integers and $c$ fixed,
\[
  \sum_j|\Lambda|^{-1/2}\ll(dT)^{1/2}\sum_{m\le\#\Jc}m^{-1/2}\ll T^{1/2}(\#\Jc)^{1/2}\ll\Kc TM^{-1/2}+T^{1/2},
\]
and
\[
  \sum_j|\Lambda|^{1/2}\ll\#\Jc\cdot\Kc M^{-1/2}\ll\Kc^3TM^{-3/2}+\Kc M^{-1/2}.
\]
Therefore, using $L\ll M/K^2$ and $\Kc\ll K$,
\begin{align*}
  \sum_{j\ \mathrm{non\text{-}exc}}\Delta(p,j)\ &\ll\ \Kc^{-3/2}M^{3/4}L^{1/4}\bigl(\Kc^3TM^{-3/2}+\Kc M^{-1/2}\bigr)\\
  &\qquad+\Kc^{3/2}M^{-3/4}L^{3/4}\bigl(\Kc TM^{-1/2}+T^{1/2}\bigr)\\
  &\ll\ \Kc^{3/2}K^{-1/2}M^{-1/2}T+\Kc^{5/2}K^{-3/2}M^{-1/2}T+\Kc^{-1/2}K^{-1/2}M^{1/2}+T^{1/2}\\
  &\ll\ KM^{-1/2}T,
\end{align*}
the last step because $M\ll KT$, which follows from $MQ\ll T^2$ and $Q=M/K^2$. In particular $\Kc^{-1/2}K^{-1/2}M^{1/2}\ll KM^{-1/2}T$ since $\Kc,K\ge1$, so this term is absorbed for each fixed $p$, before the sum over heights. This is Jutila's (4.4.50).

\smallskip\noindent\emph{4. The exceptional fraction.} There is at most one such $j$ for each height $t_p$. Suppose first that
\begin{equation}\label{eq:4452}
  |\Lambda|\le c\,\Kc M^{1/2}L^{1/2}T^{-2-\delta/4}
\end{equation}
with $c=c(N)$ small.

The correction in $B_1$ is now bounded differently, since exceptionality gives
\[
  |u_jv_j-u_{j'}v_{j'}|\le u_{j'}v_{j'}|t_p/t_{p'}-1|+\frac12\ll\Kc^2|t_p-t_{p'}|/M+1,
\]
so the correction is
\[
  \ll\frac{dL}{rT^2}\bigl(\frac{|t_p-t_{p'}|}T+\frac M{\Kc^2T}\bigr)\ll\frac{dL}{rT^2}(1+T^{\delta/8}\Kc^{-2}).
\]
By Lemma \ref{lem:deriv} and \eqref{eq:4452},
\begin{align*}
  |F_j-F_{j'}|\ &\ll\ \Kc^{-3}M^{3/2}L^{-1/2}\Bigl(c\Kc M^{1/2}L^{1/2}T^{-2-\delta/4}+\frac{dL}{rT^2}(1+T^{\delta/8}\Kc^{-2})\Bigr)\\
  &\ll\ \Kc^{-2}M^2T^{-2-\delta/4}\Bigl(c+\frac{d}{r}\cdot\frac{L^{1/2}T^{\delta/4}}{\Kc M^{1/2}}(1+T^{\delta/8}\Kc^{-2})\Bigr),
\end{align*}
and $L^{1/2}/\Kc M^{1/2}\le1/\Kc K$, so with $c$ small and $T$ large this is at most half the lower bound of Lemma \ref{lem:sep}. Hence
\[
  \|\Phi'\|\ge\frac12\|z_{jj'}\|\gg\Kc^{-2}M^2T^{-2-\delta/4}
\]
on $[L,2L]$, and the first derivative estimate on the two intervals of monotonicity gives
\[
  \Delta(p,j)\ \ll\ \Kc^2M^{-2}T^{2+\delta/4}\ \le\ K^2M^{-2}T^{2+\delta/4}\ \le\ KM^{-3/2}T^{5/3+\delta/4}\ \le\ KM^{-1/2}T^{1+\delta/4},
\]
using $K\le M^{1/2}T^{-1/3}$ and $M\ge T^{2/3}$.

If instead \eqref{eq:4452} fails, then
\[
  c\Kc M^{1/2}L^{1/2}T^{-2-\delta/4}<|\Lambda|<1/dT,
\]
the correction in $B_2$, which is $\ll dL/rT^2$, is smaller than $|\Lambda|$ by the factor
\[
  dL^{1/2}T^{\delta/4}/rc\Kc M^{1/2}\ll T^{\delta/4}/\Kc K\ll T^{-\delta/4},
\]
so $|\Phi''|\asymp\Kc^{-3}M^{3/2}L^{-3/2}|\Lambda|$ lies between $\Kc^{-2}M^2L^{-1}T^{-2-\delta/4}$ and $\Kc^{-3}M^{3/2}L^{-3/2}T^{-1}$, and the second derivative estimate gives
\begin{align*}
  \Delta(p,j)\ &\ll\ L\bigl(\Kc^{-3}M^{3/2}L^{-3/2}T^{-1}\bigr)^{1/2}+\bigl(\Kc^{-2}M^2L^{-1}T^{-2-\delta/4}\bigr)^{-1/2}\\
  &\ll\ M^{3/4}L^{1/4}T^{-1/2}+\Kc M^{-1}L^{1/2}T^{1+\delta/8}\\
  &\ll\ MT^{-1/2}+M^{-1/2}T^{1+\delta/8}\ \ll\ KM^{-1/2}T^{1+\delta/2},
\end{align*}
using $L\le M/K^2$, $\Kc\ll K$, and $M\le T^{1+\delta/8}$ for the first term. The exceptional $j$ therefore contributes $\ll KM^{-1/2}T^{1+\delta/2}$ for each $p$.

Summing the four contributions over $p\le P$ gives \eqref{eq:4446}, and the proposition.
\end{proof}

\subsection{Completing the large-value bound}
\label{sec:exponent}

\begin{proposition}\label{prop:final}
\eqref{eq:alt3} implies \eqref{eq:target}.
\end{proposition}

\begin{proof}
Split each $\widetilde S_j(t_p)$ in \eqref{eq:4441} into $O(\Lc)$ dyadic ranges $L\le\ell\le2L$ with $L\ll Q$, so that
\[
  |\widetilde S_j(t_p)|^2\ll\Lc\sum_L|\widetilde S_j^{(L)}(t_p)|^2.
\]

For each $L$, Proposition \ref{prop:sieve} applies with $\kappa_{p,j}(\ell)=w_j(\ell)\kappa_j(\ell;t_p)$ and $A\asymp\Kc^{-1/2}M^{1/4}T^{-1/2}$, because $w_j(\ell)$ is bounded by $1$ with variation at most $2$, being $\omega^{(q)}_j$ composed with the monotone function $x^\pm_j(\cdot/r;t_0)$, and $\kappa_j(\cdot\,;t_p)$ is bounded by $A$ with variation $O(A)$ on $[L,2L]$ by \eqref{eq:kappa}. Hence, with $L\ll M/K^2$,
\begin{multline*}
  \sum_{p,j}|\widetilde S_j(t_p)|^2\ \ll\ \Lc^2T^{\delta}\Kc^{-1}M^{1/2}L^{1/2}T^{-1}\bigl(\Kc^{-1}K^{-1}M+KM^{-1/2}PT\bigr)\\
  \ll\ T^{2\delta}\bigl(\Kc^{-2}K^{-2}M^2T^{-1}+\Kc^{-1}M^{1/2}P\bigr),
\end{multline*}
which is Jutila's (4.4.43). The same holds for $\widetilde S'_j(t'_p)$, whose amplitude $\partial_t\kappa_j+2\pi i(\partial_tg^\pm_{jr})\kappa_j$ is $O(A\Lc)$ with variation $O(A\Lc)$, by Cauchy's estimate, \eqref{eq:dt} and the corresponding $\ell$-derivative bounds in \eqref{eq:kappa}. Inserting into \eqref{eq:4441} with $P\le R$,
\begin{multline*}
  RV\ \ll\ R^{1/2}\Kc M^{-1/2}T^{1/2+6\delta}\bigl(\Kc^{-1}K^{-1}MT^{-1/2}+\Kc^{-1/2}M^{1/4}R^{1/2}\bigr)\\
  =T^{6\delta}\bigl(R^{1/2}K^{-1}M^{1/2}+R\Kc^{1/2}M^{-1/4}T^{1/2}\bigr).
\end{multline*}
Now $K^{-1}M^{1/2}=Q^{1/2}$ and
$\Kc^{1/2}M^{-1/4}\ll Q^{-1/4}$. Thus
\[
 RV\ll T^{6\delta}
       \bigl(R^{1/2}Q^{1/2}+RT^{1/2}Q^{-1/4}\bigr).
\]
The two terms balance at $Q=(TR)^{2/3}$, the choice already made in
\eqref{eq:M0}. Consequently $RV\ll T^{1/3+6\delta}R^{5/6}$, and
$R\ll T^{2+36\delta}V^{-6}$ follows.
\end{proof}
\subsection{The block form of the theorem}

What Sections~\ref{sec:book} and~\ref{sec:endgame} prove is a statement about a single block, and it is this that the application in Section~\ref{sec:zeros} uses. The parameter $R^*$ below records the count used to build the partition before a subset of the heights is selected; it differs from the surviving count $R$ only by logarithmic factors.

\begin{theorem}\label{thm:block}
Let $\delta>0$ be small, $T$ large, $W\ge T^{1/4+\delta}$, and let $\{t_\nu\}$ be a set of $R\le T^{1/2-2\delta}\Lc^{3}$ points, $1$-spaced in $[T,2T]$. Let the Farey system, the intervals $I_q$ and the partition $\{\omega^{(q)}_j\}$ of Section~\ref{sec:setup} be constructed using $Q=(TR^*)^{2/3}$, where $R\le R^*\le R\Lc^2$. Let $T^\delta Q/2\le M\le T^{1+\delta/8}$, and let $\phi_t(z)$ be holomorphic in $t$ and $z$ for
\[
  |\Imag t|\le T^{1/2},\quad T/2\le\Real t\le3T,\quad M/8\le\Real z\le8M,\quad |\Imag z|\le M/8,
\]
and bounded by $1$ there. If, for every $\nu$ with $t_\nu\in I_q$,
\[
  \Bigl|\sum_{j\in\Bc(M)}\sum_{n\ge1}\frac{\lambda_f(n)}{n^{\frac12+it_\nu}}\phi_{t_\nu}(n)\,\omega^{(q)}_j(n)\Bigr|\ \ge\ W,
\]
then $R\ll T^{2+O(\delta)}W^{-6}$.
\end{theorem}

\begin{proof}
This is the argument of Sections~\ref{sec:book} and~\ref{sec:endgame} with $V_t(n/\sqrt C)$ replaced by $\phi_t(n)$, $V$ by $W$. No upper bound on $W$ is needed, since the proof of Proposition \ref{prop:book} treats $V\ge T^{1/3+\delta}$ separately, and elsewhere in Sections~\ref{sec:book} and~\ref{sec:endgame} only the lower bound in \eqref{eq:Vrange} is used. The amplitude is used only through Proposition \ref{prop:ext}, whose proof in Appendix \ref{sec:evidence} asks of it exactly holomorphy and boundedness on the discs used in \cite[Lemma 4.3]{BMN}, and through \eqref{eq:kappa}, where holomorphy in $t$ is used in Lemma \ref{lem:offdiag}.

The hypothesis $R\le T^{1/2-2\delta}\Lc^{3}$ is \eqref{eq:R2}, and with $R^*\le R\Lc^2$ it is what makes $T^\delta Q=T^{2/3+\delta}(R^*)^{2/3}\le T^{1-\delta/4}$, as required in Section~\ref{sec:book}. The parameter $R^*$ enters only through $K=M^{1/2}T^{-1/3}(R^*)^{-1/3}$, $Q=T^{2/3}(R^*)^{2/3}$; the initial cutoff is $T^\delta Q$. In Section~\ref{sec:exponent} the two terms $K^{-1}M^{1/2}=T^{1/3}(R^*)^{1/3}$ and $K^{1/2}M^{-1/4}T^{1/2}=T^{1/3}(R^*)^{-1/6}$ then give
\[
  RW\ll T^{1/3+O(\delta)}(R^{1/2}(R^*)^{1/3}+R(R^*)^{-1/6})\ll T^{1/3+O(\delta)}R^{5/6}\Lc,
\]
and $R\ll T^{2+O(\delta)}W^{-6}$ follows as before.
\end{proof}

\subsection{Proof of Theorem \ref{thm:main}}\label{sec:proofmain}

Let $V$ satisfy \eqref{eq:Vrange} and let $\{t_\nu\}$ be $1$-spaced in $[T,2T]$ with $|L(\half+it_\nu,f)|\ge V$. By Theorem \ref{thm:second} we may assume \eqref{eq:R2}. By Section~\ref{sec:setup}, either \eqref{eq:alt1} holds for $\gg R$ points, in which case \eqref{eq:target} was proved there, or \eqref{eq:alt2} holds for some block. In the latter case Proposition \ref{prop:book} either gives \eqref{eq:target} directly or leaves \eqref{eq:alt3}, to which Proposition \ref{prop:final} applies.

To pass to the integral, choose a maximum of $|L(\half+it,f)|$ in
each unit interval meeting $[T,2T]$. The maxima from alternate intervals
form two $1$-spaced sets. Values below $T^{1/4+\delta}$ contribute at most
\[
 T^{1+4\delta}\sum_\nu |L(\half+it_\nu,f)|^2
 \ll_f T^{2+4\delta}\Lc^3
\]
by Theorem \ref{thm:second}. On each dyadic range $[V,2V)$ above that
threshold, \eqref{eq:target} bounds the sixth-power sum by
$T^{2+O(\delta)}$. There are $O(\Lc)$ such ranges by the Weyl bound.
The integral is bounded by the sum of the unit-interval maxima, and so
is $\ll T^{2+O(\delta)}$. Choose $\delta$ sufficiently small in terms of
$\epsilon$ and sum the dyadic height intervals. This proves
Theorem \ref{thm:main}.

The same sampling argument gives
$\int_T^{2T}|L(\half+it,f)|^2\dd t\ll_f T\Lc^3$.
Cauchy's inequality, applied to $|L|$ and $|L|^3$, now gives
Corollary \ref{cor:fourth}.

\section{Moments on lines to the right of the critical line}
\label{sec:lines}

Throughout this section $\half<\sigma\le1$, $T$ is large and $\delta>0$ is small. Implied constants may depend on $\sigma$, and they are uniform for $\sigma$ in compact subsets of $(\half,1]$. Let $\{t_\nu\}$ be a set of $R$ points, $1$-spaced in $[T,2T]$, with
\begin{equation}\label{eq:largesigma}
  |L(\sigma+it_\nu,f)|\ge V .
\end{equation}
We compare $L(\sigma+it,f)$ with a Dirichlet polynomial and with the values of $L(s,f)$ on the critical line.

\begin{lemma}\label{lem:dichotomy}
Let $2\le Y\le T^2$ and $V\ge T^\delta$. For each $\nu$ at least one of the following holds.
\begin{enumerate}
\item[(i)] There is $M\le Y\Lc^2$ with
\[
  \Bigl|\sum_{M<n\le2M}\lambda_f(n)e^{-n/Y}n^{-\sigma-it_\nu}\Bigr|\ge V\Lc^{-2}.
\]
\item[(ii)] There is $t'$ with $|t'-t_\nu|\le\Lc^2$ and
\[
  |L(\half+it',f)|\ge cVY^{\sigma-\frac12}\Lc^{-2},
\]
where $c>0$ depends only on $f$ and $\sigma$.
\end{enumerate}
\end{lemma}

\begin{proof}
Let $s=\sigma+it_\nu$. By Mellin inversion
\begin{equation}\label{eq:mellinlines}
  \sum_{n\ge1}\lambda_f(n)e^{-n/Y}n^{-s}=\frac1{2\pi i}\int_{(2)}L(s+w,f)\Gamma(w)Y^w\dd w .
\end{equation}
Move the contour to $\Real w=\half-\sigma$. This is legitimate because $L(s,f)$ is entire and of polynomial growth in vertical strips. The only pole crossed is at $w=0$, and its residue is $L(s,f)$.

On the new line $|Y^w|=Y^{\frac12-\sigma}$ and $|\Gamma(\half-\sigma+iv)|\ll e^{-|v|}$. By the convexity bound the part of the integral with $|v|>\Lc^2$ is $O(1)$, and the rest is at most
\[
  AY^{\frac12-\sigma}\Lc^2\max_{|v|\le\Lc^2}\bigl|L(\half+i(t_\nu+v),f)\bigr|,
\]
with $A$ depending only on $f$ and $\sigma$. On the left of \eqref{eq:mellinlines} the terms with $n>Y\Lc^2$ contribute $O(1)$, and the others fall into $O(\Lc)$ dyadic blocks. Since $V\ge T^\delta$, one of the two alternatives holds once $T$ is large.
\end{proof}

\begin{proof}[Proof of Theorem \ref{thm:lines}]
When $\sigma=\half$ we have $m(\sigma)=2$, and the theorem is the mean square bound obtained in the proof of Theorem \ref{thm:main}. So let $\half<\sigma<1$, and put $a=\sigma-\half$ and $m=m(\sigma)$. By \cite[Theorem 1.1]{BMN} and the Phragm\'en--Lindel\"of principle,
\[
  |L(\sigma+it,f)|\ll|t|^{\frac23(1-\sigma)}\log|t|,
\]
so we may assume that $T^\delta\le V\le V_{\max}:=C\,T^{(1-2a)/3}\Lc$, where $C$ depends only on $f$ and $\sigma$. We shall show that
\begin{equation}\label{eq:curvecount}
  R\ll T^{1+O(\delta)}V^{-m}.
\end{equation}
Apply Lemma \ref{lem:dichotomy} with
\begin{equation}\label{eq:Ylines}
  Y=\min\Bigl(T^2,\ \bigl(TV^{2-m}\bigr)^{1/(1-2a)}\Bigr),
\end{equation}
which satisfies $T^{1/2}\Lc^{-O(1)}\le Y\le T^2$, since $V\le V_{\max}$. When $\sigma=\frac58$ this is $Y=T^{4/3}V^{-8/3}$. If $Y<T^2$ then
\begin{equation}\label{eq:Ycap}
  Y^{1-2a}=TV^{2-m},
\end{equation}
and if $Y=T^2$ then $Y^{1-2a}\le TV^{2-m}$. Let $R_1$ and $R_2$ be the numbers of points for which (i) and (ii) hold.

Consider first case (i). We may fix $M$ at the cost of a factor $\Lc$. Put $b_n=\lambda_f(n)e^{-n/Y}n^{-\sigma}$ for $M<n\le2M$. By the Rankin--Selberg bound
\[
  G:=\sum_{M<n\le2M}|b_n|^2\ll M^{1-2\sigma}\Lc .
\]
Huxley's large values theorem \cite[Lemma 4.3]{Jutila} gives
\begin{equation}\label{eq:huxlines}
  R_1\ll T^{O(\delta)}\bigl(GMV^{-2}+TG^3MV^{-6}\bigr).
\end{equation}
Gallagher's lemma and the mean value theorem of Montgomery and Vaughan, as in the proof of Theorem \ref{thm:second}, give
\begin{equation}\label{eq:mvlines}
  R_1\ll T^{O(\delta)}(T+M)GV^{-2}.
\end{equation}
Here $GM\ll M^{1-2a}\Lc$, $G\ll M^{-2a}\Lc$ and $G^3M\ll M^{1-6a}\Lc^3$, and $M\le Y\Lc^2$. Combining \eqref{eq:huxlines} and \eqref{eq:mvlines} therefore gives
\[
  R_1\ \ll\ T^{O(\delta)}\Bigl(Y^{1-2a}V^{-2}+\min\bigl(TM^{1-6a}V^{-6},\,(T+M)M^{-2a}V^{-2}\bigr)\Bigr).
\]
The first term is $\ll T^{1+O(\delta)}V^{-m}$ by \eqref{eq:Ycap}. When $M>T$ the minimum is at most twice the first term, so we may assume that $M\le T$.

If $a\le\frac16$, put $\theta=2a/(1-4a)$, which lies in $[0,1]$. The minimum is then at most the weighted geometric mean
\[
  \bigl(TM^{1-6a}V^{-6}\bigr)^{\theta}\bigl(TM^{-2a}V^{-2}\bigr)^{1-\theta}=TV^{-2/(1-4a)},
\]
since the exponent of $M$ is $\theta(1-4a)-2a=0$. For $a\le\frac18$ we have $2/(1-4a)=m$. For $\frac18\le a\le\frac16$ we have $2/(1-4a)\ge3/(1-2a)=m$, because $2(1-2a)-3(1-4a)=8a-1\ge0$.

If $a>\frac16$, then by Cauchy's inequality and the Rankin--Selberg bound the sum in case (i) is at most
\[
  \sum_{M<n\le2M}|\lambda_f(n)|n^{-\sigma}\ll M^{\frac12-a},
\]
so that $M\gg(V\Lc^{-2})^{2/(1-2a)}$. Since $1-6a<0$, this gives
\[
  TM^{1-6a}V^{-6}\ll T^{1+O(\delta)}V^{\frac{2(1-6a)}{1-2a}-6}=T^{1+O(\delta)}V^{-\frac4{1-2a}}\le T^{1+O(\delta)}V^{-m}.
\]
As $V\ge1$, in every case $R_1\ll T^{1+O(\delta)}V^{-m}$.

Consider next case (ii). The points $t'$ lie in $[T/2,4T]$, and passing to a $1$-spaced subset costs a factor $\Lc^3$. They satisfy $|L(\half+it',f)|\ge W$, where
\[
  W=cVY^{a}\Lc^{-2}.
\]
By \eqref{eq:R1}, and by \eqref{eq:target} applied on $[T/2,T]$, $[T,2T]$ and $[2T,4T]$ when $W\ge T^{1/4+\delta}$,
\begin{equation}\label{eq:R2lines}
  R_2\ll T^{1+O(\delta)}W^{-2},\qquad R_2\ll T^{2+O(\delta)}W^{-6}\quad\text{if }W\ge T^{1/4+\delta}.
\end{equation}

If $Y=T^2$, which can happen only when $\sigma>\frac34$, then $V^{m-2}\le T^{4a-1}$, and the first bound in \eqref{eq:R2lines} gives
\[
  R_2\ll T^{1+O(\delta)}V^{-2}T^{-4a}\le T^{1+O(\delta)}V^{-m}.
\]

Otherwise write $V=T^v$, so that \eqref{eq:Ycap} gives $Y^a=T^{a(1+(2-m)v)/(1-2a)}$. Substituting this, the first bound in \eqref{eq:R2lines} is $\ll T^{1+O(\delta)}V^{-m}$ if
\begin{equation}\label{eq:v1}
  v\le v_1:=\frac{2a}{m-2},
\end{equation}
and the second is if
\begin{equation}\label{eq:sixthcond}
  8a-1\ge v\bigl(m(1+4a)-6\bigr).
\end{equation}
Up to $O(\delta)$ the exponent of $W$ is
\[
  w(v)=v+\frac{a\bigl(1+(2-m)v\bigr)}{1-2a}.
\]

Suppose first that $\half<\sigma\le\frac58$, so that $0<a\le\frac18$ and $m=2/(1-4a)$. Then $v_1=(1-4a)/4$ and
\[
  m(1+4a)-6=-\frac{4(1-8a)}{1-4a},
\]
so \eqref{eq:sixthcond} holds when $v\ge(1-4a)/4$. The function $w(v)$ is increasing and equals $\frac14$ at $v=(1-4a)/4$. Hence for $v\ge(1-4a)/4+C'\delta$, with $C'$ large, we have $W\ge T^{1/4+\delta}$ and the second bound applies. For smaller $v$ the first bound applies, with a further loss of $T^{O(\delta)}$.

Suppose next that $\frac58\le\sigma<1$, so that $\frac18\le a<\half$ and $m=3/(1-2a)$. Then
\[
  v_1=\frac{2a(1-2a)}{1+4a},\qquad m(1+4a)-6=\frac{3(8a-1)}{1-2a},
\]
and \eqref{eq:sixthcond} reads $v\le(1-2a)/3$, which holds for $V\le V_{\max}$ up to a factor $T^{O(\delta)}$. The function $w(v)$ is linear, with
\[
  w(v_1)=\frac{3a}{1+4a}\ge\frac14,\qquad w\Bigl(\frac{1-2a}3\Bigr)=\frac13,
\]
and the first inequality is strict unless $a=\frac18$. For $v\le v_1$ the first bound applies. For $v\ge v_1$ and $a>\frac18$ we have $W\ge T^{1/4+\delta}$ once $\delta$ is small in terms of $\sigma$, and the second bound applies. The case $a=\frac18$ is covered by the previous paragraph.

In all cases $R_2\ll T^{1+O(\delta)}V^{-m}$. Together with the bound for $R_1$, this proves \eqref{eq:curvecount}.

Hence the number of $1$-spaced points in $[T,2T]$ satisfying \eqref{eq:largesigma} is $\ll T^{1+O(\delta)}V^{-m}$, for every $V\ge T^\delta$. In each interval $[T+l,T+l+1]$ choose a point at which $|L(\sigma+it,f)|$ is maximal. The points with $l$ even form a $1$-spaced set, and so do those with $l$ odd. Splitting them according to the dyadic range of $|L(\sigma+it,f)|$ gives
\[
  \int_T^{2T}\bigl|L(\sigma+it,f)\bigr|^{m}\dd t\ll T^{1+m\delta}+T^{1+O(\delta)}\Lc ,
\]
and taking $\delta=\epsilon/C_0$, with $C_0$ large in terms of $\sigma$, gives the theorem.
\end{proof}

The same comparison, with $Y=T^{2/3}$, gives the sixth moment on every line to the right of the critical line. This is used in Section~\ref{sec:cubic}.

\begin{lemma}\label{lem:sixthlines}
Let $\half<\sigma\le1$ be fixed. Then
\[
  \int_T^{2T}\bigl|L(\sigma+it,f)\bigr|^6\dd t\ll_{f,\sigma,\epsilon}T^{2+\epsilon}.
\]
\end{lemma}

\begin{proof}
We may assume that $T^\delta\le V\le T^{1/3}\Lc$. Apply Lemma \ref{lem:dichotomy} with $Y=T^{2/3}$.

In case (i), Huxley's large values theorem gives \eqref{eq:huxlines} with $G\ll M^{1-2\sigma}\Lc$ and $M\le Y\Lc^2$. Since $\sigma>\half$, this is
\[
  R_1\ll T^{O(\delta)}\bigl(M^{2-2\sigma}V^{-2}+TM^{4-6\sigma}V^{-6}\bigr)\ll T^{O(\delta)}\bigl(T^{2/3}V^{-2}+T^{5/3}V^{-6}\bigr),
\]
which is $\ll T^{2+O(\delta)}V^{-6}$ because $V\le T^{1/3}\Lc$.

In case (ii) we have $W=cVY^{\sigma-\frac12}\Lc^{-2}\ge cV\Lc^{-2}$. If $W\ge T^{1/4+\delta}$ then \eqref{eq:target} gives $R_2\ll T^{2+O(\delta)}V^{-6}$. Otherwise $V\le T^{1/4+2\delta}$, and \eqref{eq:R1} gives $R_2\ll T^{1+O(\delta)}V^{-2}\ll T^{2+O(\delta)}V^{-6}$.

The passage to the integral is as in the proof of Theorem \ref{thm:lines}.
\end{proof}

\section{Simple zeros}
\label{sec:zeros}

\subsection{The input from de Faveri}

We use the following consequence of \cite[Propositions 2.2 and 3.3,
and Lemma 4.1]{deFaveri}. For every $\epsilon>0$, after replacing $f$
by $\bar f$ if necessary,
\begin{equation}\label{eq:53}
 \sum_{\rho=\beta+i\gamma\ \mathrm{simple}}
 |\Lambda'(\rho,f)|e^{\pi|\gamma|/2}
 (1+|\gamma|)^{-k/2+\epsilon}=\infty.
\end{equation}
The sum is over simple zeros in $0<\beta<1$. This is the only
pole-detection input needed for the branch $\theta_f\le7/9$.

At a zero, differentiating the functional equation gives
\[
 \Lambda'(\rho,f)=-\epsilon_fN^{1/2-\rho}
                         \Lambda'(1-\rho,\bar f).
\]
Thus terms with $\beta<1/2$ can be reflected to $\beta>1/2$ for the
dual form, at a bounded cost. At least one of the resulting sums for
$f$ and $\bar f$ still diverges. Since $N_f^s=N_{\bar f}^s$ and
$\theta_f=\theta_{\bar f}$, call that form $f$.

Stirling's formula gives, for $|\gamma|\ge1$,
\[
 |\Lambda'(\rho,f)|e^{\pi|\gamma|/2}
 \asymp_f |L'(\rho,f)|\,|\gamma|^{\beta+k/2-1}.
\]
Consequently, on putting
\begin{equation}\label{eq:SigmaT}
 \Sigma(T)=\sum_{\substack{\rho=\beta+i\gamma\ \mathrm{simple}\\
                  \beta\ge1/2,\ T/2<|\gamma|\le T}}
             |L'(\rho,f)|\,|\gamma|^{\beta-1/2},
\end{equation}
equation \eqref{eq:53} implies
$\sum_{T=2^j}T^{-1/2+\epsilon}\Sigma(T)=\infty$. In particular,
\begin{equation}\label{eq:Sigmalarge}
 \Sigma(T)\ge T^{1/2-2\epsilon}
 \qquad\text{for infinitely many dyadic }T.
\end{equation}
Otherwise that series would be bounded by a convergent geometric series.
We shall show that so large a value of $\Sigma(T)$ requires many zeros.

\subsection{From zeros to large values on a line}

\begin{lemma}\label{lem:zerostopoints}
Let $\half\le\beta\le1$, $c=1/\log T$, $V\ge\Lc$, and let $\mathcal Z$ be the set of simple zeros $\rho_n$ of $\Lambda(s,f)$ with $\beta\le\beta_n\le\beta+c$, $T/2<\gamma_n\le T$ and $|L'(\rho_n,f)|\ge V$. Then there is a set of $\gg\#\mathcal Z\,\Lc^{-3}$ points $t$, $1$-spaced in $[T/3,2T]$, with $|L(\beta-c+it,f)|\ge V\Lc^{-3}$.
\end{lemma}

\begin{proof}
This is the argument of \cite[Lemma 5.1]{deFaveri} with the contour moved to $\Real s=\beta-c$ rather than to the critical line. For $\rho=\rho_n\in\mathcal Z$,
\[
  \frac1{2\pi i}\int_{(1)}L(\rho+w,f)\Gamma(w)^2\dd w\ll1,
\]
and moving the contour to $\Real w=\beta-c-\beta_n\in[-2c,-c]$ passes the double pole of $\Gamma(w)^2$ at $w=0$, whose residue is $L'(\rho,f)$ because $L(\rho,f)=0$.

Since
\[
  |\Gamma(-u+iy)|^2\ll e^{-\pi|y|}(1+|y|)(u+|y|)^{-2}
\]
for $c\le u\le2c$, and $L(\beta-c+i(\gamma+y),f)\ll(T+|y|)^{1/2+\epsilon}$ by convexity, we obtain
\[
  |L'(\rho,f)|\le A+A\int_{|y|\le\Lc^2}|L(\beta-c+i(\gamma+y),f)|(1+|y|)(c+|y|)^{-2}\dd y+O(T^{-1})
\]
with $A$ depending only on $f$, and since $\int_{|y|\le1}(c+|y|)^{-2}\dd y\le2c^{-1}$ and $\int_{1\le|y|\le\Lc^2}|y|^{-1}\dd y\ll\log\Lc$, the integral is $O(\Lc\max_{|y|\le\Lc^2}|L(\beta-c+i(\gamma+y),f)|)$.

As $V\ge\Lc$, each $\rho_n\in\mathcal Z$ therefore produces a point $t_n$ with $|t_n-\gamma_n|\le\Lc^2$ and $|L(\beta-c+it_n,f)|\ge V\Lc^{-3}$ for $T$ large. A unit interval contains $O(\Lc)$ zeros, so any $t$ is within $\Lc^2+1$ of $O(\Lc^3)$ of the $\gamma_n$, and a maximal $1$-spaced subset of $\{t_n\}$ has $\gg\#\mathcal Z\Lc^{-3}$ elements.
\end{proof}

\subsection{Large values on a line}

\begin{proposition}\label{prop:line}
Let $\half-\frac1{\log T}\le\beta\le1$, put $a=\beta-\half$, and let $\{t_\nu\}$ be a set of $R\le T^{1/5}$ points, $1$-spaced in $[T/3,2T]$, with $|L(\beta+it_\nu,f)|\ge V\ge T^{\delta}$. Then there are $M$ and $W$ with
\begin{equation}\label{eq:MW}
  1\le M\le T^{1+\delta},\qquad VM^{a}T^{-\delta}\le W\le M^{1/2}T^{\delta},
\end{equation}
such that one of the following holds.
\begin{enumerate}
\item[(a)] $M\le T^{2/3+\delta}R^{2/3}$ and $R\ll T^{\delta}\bigl(MW^{-2}+TMW^{-6}\bigr)$;
\item[(b)] $M\ge\frac12T^{2/3+\delta}R^{2/3}$, and if moreover $W\ge T^{1/4+\delta}$ and $R\le T^{1/2-2\delta}$ then $R\ll T^{2+O(\delta)}W^{-6}$.
\end{enumerate}
\end{proposition}

\begin{proof}
By splitting $[T/3,2T]$ into three intervals and replacing $T$ by $T/3$ or $T/2$ we may take the points in $[T,2T]$. Apply Lemma \ref{lem:afe} at $\beta+it$ with $X=\sqrt C$. The second sum is
\[
  \gamma_f(\beta+it)\sum_n\lambda_{\bar f}(n)n^{-(1-\beta)+it}W_t(n/\sqrt C)
\]
with $|\gamma_f|\asymp T^{1-2\beta}$, and $n^{-(1-\beta)}=n^{-\frac12}n^{a}$. The first is
\[
  \sum_n\lambda_f(n)n^{-\beta-it}V_t(n/\sqrt C)
\]
with $n^{-\beta}=n^{-\frac12}n^{-a}$.

For at least half the points one of the two sums is $\ge V/3$. We treat the first, the second being the same with $\bar f$ and $-t$, or equivalently with $f$, $t$ and the conjugate amplitude $\psi^*(z)=\overline{\psi(\bar z)}$, and with the factor $T^{1-2\beta}n^{a}$ in place of $n^{-a}$. On a block $n\asymp M\le T^{1+\delta}$ this factor is at most $T^{2a\delta}M^{-a}$, so the second sum leads to a larger threshold $W$ and the same conclusion.

Build the Farey system, the intervals $I_q$ and the partition $\{\omega^{(q)}_j\}$ of Section~\ref{sec:setup} using the present count $R$, so $Q=(TR)^{2/3}$, and decompose as in Section~\ref{sec:setup} into the initial segment, the blocks $\Bc(M)$ with $T^\delta Q/2\le M\le\sqrt C\,T^{\delta/20}$, and negligible tails. Either the initial segment or some block is $\ge V\Lc^{-1}$ for at least $R\Lc^{-1}$ of the points.

If it is the initial segment, split it into dyadic pieces $M<n\le2M$ with $M\le T^\delta Q$. Some piece is $\ge V\Lc^{-2}$ for at least $R\Lc^{-2}$ points. On such a piece $n^{-\beta}=M^{-a}(n/M)^{-a}n^{-\frac12}$, and $(n/M)^{-a}$ times the weight of the initial segment has total variation $O(1)$, so by partial summation there is $N_1=N_1(t_\nu)\in(M,2M]$ with
\[
  |\sum_{M<n\le N_1}\lambda_f(n)n^{-\frac12-it_\nu}|\gg VM^{a}\Lc^{-2}=:W.
\]

Lemma \ref{lem:maximal}, applied with $N=2M$ and coefficients
$a_n=\lambda_f(n)n^{-1/2}$ supported on $(M,2M]$, removes the
point-dependent endpoint. Since $G\ll_f1$, it gives
\[
  R\ll(MW^{-2}+TMW^{-6})T^{\delta}.
\]
This is (a), and $W\ll M^{1/2}\Lc$ because the sum is trivially so bounded.

If it is a block $\Bc(M)$, then since $n^{-\beta}=M^{-a}\cdot n^{-\frac12}(n/M)^{-a}$ the block equals $c_f^{-1}M^{-a}$ times the block sum with amplitude
\[
  \phi_t(z)=c_f(z/M)^{-a}V_t(z/\sqrt{C(t)}),
\]
where $c_f>0$ is a constant depending on $f$ only, the constants in Lemma \ref{lem:afe} being uniform in $\beta$.

This $\phi_t(z)$ is holomorphic in $t$ and $z$ on the region required by Theorem \ref{thm:block}, since $(z/M)^{-a}$ is holomorphic off the negative axis and $V_t(y)$ and $\sqrt{C(t)}$ are as in Section~\ref{sec:transformed}, and it is bounded by $1$ there if $c_f$ is small, because $|(z/M)^{-a}|\le8^{a}e^{a\pi/4}$ on the region and $V_t\ll_f1$ in the sector.

So the block sum with amplitude $\phi_t(z)$ is $\ge W:=c_fVM^{a}\Lc^{-1}$ for at least $R\Lc^{-1}$ points. If $W\ge T^{1/4+\delta}$ and $R\le T^{1/2-2\delta}$, Theorem \ref{thm:block} applies, with $R^*=R$ and these $R\Lc^{-1}$ points, and gives $R\ll T^{2+O(\delta)}W^{-6}$. This is (b).
\end{proof}

\subsection{Counting the zeros}

The next lemma bounds the number of large values directly. In its
application $M=T^m$, $W=T^w$ and $R=T^r$. A class of zeros contributing
$T^{1/2-o(1)}$ to \eqref{eq:SigmaT} forces
$w\ge\half-a-r+am-o(1)$, where $a=\beta-\half$.

\begin{lemma}\label{lem:opt}
Let $0\le a\le1/3$, $0\le m\le1$ and $0\le r\le1/5$, and suppose
\begin{equation}\label{eq:constraints}
 \half-a-r+am\le w\le m/2.
\end{equation}
Then $r\ge(1-2a)/(5-4a)$ whenever one of the following alternatives holds:
\begin{enumerate}
\item[(a)] $r\ge3m/2-1$ and $r\le\max(m-2w,1+m-6w)$;
\item[(b)] $r\le3m/2-1$, and $r\le2-6w$ if $w\ge1/4$.
\end{enumerate}
If the inequalities in a branch are relaxed by $O(\eta)$, the conclusion
is $r\ge(1-2a)/(5-4a)-O(\eta)$, uniformly for $0\le a\le1/3$.
\end{lemma}

\begin{proof}
In case (b), $m\ge2(1+r)/3\ge2/3$.
If $w<1/4$, then \eqref{eq:constraints} gives
\[
 r>\frac14-a+am\ge\frac14-\frac a3
   >\frac{1-2a}{5-4a},
\]
because the difference in the last comparison is
$(16a^2-8a+3)/(12(5-4a))>0$.
If $w\ge1/4$, use $r\le2-6w$ and \eqref{eq:constraints} to obtain
\[
 5r\ge1-6a+6am\ge1-2a+4ar.
\]
This is the desired inequality.

In case (a), $m\le2(1+r)/3$. If $r\le m-2w$, then
\eqref{eq:constraints} gives
\[
 r\ge(1-2a)(1-m)\ge\frac{1-2a}{3}(1-2r),
\]
which again is $(5-4a)r\ge1-2a$.

It remains to consider $r\le1+m-6w$. In this case
\begin{equation}\label{eq:shortcount}
 5r\ge2-6a-(1-6a)m,\qquad
 r\ge(\half-a)(1-m).
\end{equation}
For $a\le1/6$, substitute $m\le2(1+r)/3$ in the first inequality:
\[
 r\ge\frac{4-6a}{17-12a}
   =\frac{1-2a}{5-4a}
       +\frac{3}{(17-12a)(5-4a)}.
\]
For $a\ge1/6$, the first lower bound in \eqref{eq:shortcount}
increases with $m$ and the second decreases. Their intersection is
$m=(1+2a)/(3+2a)$, so
\[
 r\ge\frac{1-2a}{3+2a}\ge\frac{1-2a}{5-4a},
\]
where the last step uses $a\le1/3$.

All substitutions remain valid with an $O(\eta)$ error. The divisors
$5-4a$, $17-12a$ and $3+2a$ stay bounded away from zero on $[0,1/3]$.
In the threshold case $w=1/4+O(\eta)$, use the low-threshold inequality
with the same error. This proves the uniform assertion.
\end{proof}

\begin{proof}[Proof of Corollary \ref{cor:zeros}]
Let $\theta_f\le\frac79$. The case $\theta_f>\frac79$ is \cite[Corollary 4.9]{deFaveri}. Let $C_0$ be a sufficiently large absolute constant and let $0<\delta\le\epsilon$. Suppose that $N^s_f(T)\le T^{E-C_0\epsilon}$ for all large $T$, where $E=E(\theta_f)$ is the exponent of Corollary~\ref{cor:zeros}, and let $T$ be one of the values in \eqref{eq:Sigmalarge}.

Split the zeros in \eqref{eq:SigmaT} according to the interval $[\beta,\beta+c]$, $c=1/\log T$, containing $\beta_n$, of which there are $O(\Lc)$ with $\half\le\beta\le\theta_f$, according to the sign of $\gamma_n$, and according to the dyadic interval $[V,2V)$ containing $|L'(\rho_n,f)|$, with $V\ge T^{2\delta}$, together with one further class for $|L'(\rho_n,f)|<T^{2\delta}$. Some class contributes at least $T^{\frac12-3\epsilon}$ to $\Sigma(T)$. If it is the class with $|L'|<T^{2\delta}$, then
\[
  2T^{2\delta}T^{\theta_f-\frac12}N^s_f(T)\ge T^{\frac12-3\epsilon}
\]
and
\[
  N^s_f(T)\ge T^{1-\theta_f-2\delta-4\epsilon},
\]
which exceeds $T^{E-C_0\epsilon}$ since $E\le1-\theta_f$.

Otherwise, with $V=T^v$, the class has $\gg T^{\frac12-3\epsilon}V^{-1}T^{\frac12-\beta}=T^{1-\beta-v-3\epsilon}$ members, and at most $N^s_f(T)\le T^{1/5}$. If its zeros have negative ordinates we replace $f$ by $\bar f$, which changes neither $\theta_f$ nor $N^s_f(T)$, since the zeros $\bar\rho_n$ of $\Lambda(s,\bar f)$ have the same real parts, positive ordinates, and $|L'(\bar\rho_n,\bar f)|=|L'(\rho_n,f)|$.

Since $V\ge T^{2\delta}\ge\Lc$, Lemma \ref{lem:zerostopoints} gives a set of $R$ points, $1$-spaced in $[T/3,2T]$, with $|L(\beta-c+it,f)|\ge V\Lc^{-3}$, and, passing to a subset,
\[
  T^{1-\beta-v-3\epsilon}\Lc^{-4}\le R\le T^{1/5}.
\]
Apply Proposition \ref{prop:line} to these points on the line $\Real s=\beta-c$ with $V\Lc^{-3}\ge T^\delta$ in place of $V$, which changes the exponents by $O(\delta)$, since $T^{cm}\le e$. In case (b) its hypothesis $R\le T^{1/2-2\delta}$ holds.
Write $M=T^m$, $W=T^w$ and $R=T^r$, and put $a=\beta-\half$.
The lower bound for $R$ and Proposition \ref{prop:line} give
\[
 v+r\ge\half-a-O(\epsilon),\qquad
 w\ge v+am-O(\delta),\qquad w\le m/2+O(\delta).
\]
Eliminating $v$ gives \eqref{eq:constraints} with an $O(\epsilon)$ error.
The short- and long-block alternatives give (a) and (b) of
Lemma \ref{lem:opt}, also with that error; $0\le r\le1/5$ by construction.
The lemma therefore yields
\[
 N_f^s(T)\ge R\ge
 T^{\frac{1-2a}{5-4a}-O(\epsilon)}
 =T^{E(\beta)-O(\epsilon)}.
\]
Since $E(\theta)$ decreases on $[\half,\frac79]$ and $\beta\le\theta_f\le\frac79$, this is at least $T^{E-O(\epsilon)}$, which contradicts $N^s_f(T)\le T^{E-C_0\epsilon}$ once $C_0$ is large enough. The unconditional statement follows because $E(\theta)\ge E(\frac79)=\frac4{35}$ for $\half\le\theta\le1$. If the simple zeros lie on the critical line then $\theta_f=\half$, and $E(\half)=\frac15$.
\end{proof}

\begin{remark}\label{rem:why}
The gain comes from retaining the length of the block that produces a
large value. On $\Real s=\beta=1/2+a$, a block at scale $M$ carries the
factor $M^{-a}$. Proposition \ref{prop:line} uses this factor before
counting the large values; Lemma \ref{lem:opt} then gives the required
number directly. Equality is possible at the transition
$M=T^{2/3}R^{2/3}$, with
\[
 m=\frac{4(1-a)}{5-4a},\qquad r=\frac{1-2a}{5-4a}.
\]

Using the sixth moment only after summing over zeros loses this
information. H\"older's inequality would give
$\sum|L'(\rho,f)|\ll N_f^s(T)^{5/6}T^{1/3+\epsilon}$
\cite[Lemma 5.2]{deFaveri}; together with \eqref{eq:Sigmalarge} and
$\beta\le\theta_f$, this yields only $(4-6\theta_f)/5$.
At level one a stronger pole-location input in \cite{deFaveri} instead
gives $1/5$ for every $\theta_f$. The general-level improvement here
does not change his threshold $7/9$ or his result above it.
\end{remark}

\section{Zeros off the critical line}
\label{sec:density}

Throughout this section $\frac34<\sigma<1$, $T$ is large and $\delta>0$ is small. The zeros of $L(s,\bar f)$ are the complex conjugates of those of $L(s,f)$. So it is enough to bound, for $f$ and for $\bar f$, the number $N^+_f(\sigma,T)$ of zeros $\rho=\beta+i\gamma$ of $L(s,f)$ with $\beta\ge\sigma$ and $T\le\gamma\le2T$.

\subsection{Zero detection}

For $\Real s>1$ write
\begin{equation}\label{eq:muf}
  \frac1{L(s,f)}=\sum_{n\ge1}\frac{\mu_f(n)}{n^s}.
\end{equation}
The Euler product of $L(s,f)$ shows that $\mu_f$ is multiplicative, with $\mu_f(p)=-\lambda_f(p)$, $\mu_f(p^2)=\xi(p)$ and $\mu_f(p^k)=0$ for $k\ge3$. Here $\xi(p)=0$ for $p\mid N$. Hence $|\mu_f(n)|\le d(n)$.

Let $X=T^\delta$ and $M_X(s)=\sum_{n\le X}\mu_f(n)n^{-s}$. Then
\begin{equation}\label{eq:cn}
  L(s,f)M_X(s)=\sum_{n\ge1}\frac{c_n}{n^s},\qquad c_1=1,\qquad c_n=0\ \ (1<n\le X),\qquad|c_n|\le d(n)^2 .
\end{equation}

\begin{lemma}\label{lem:detect}
Let $T^\delta\le Y\le T^2$, and let $\rho=\beta+i\gamma$ be a zero of $L(s,f)$ with $\beta\ge\sigma$ and $T\le\gamma\le2T$. Then at least one of the following holds.
\begin{enumerate}
\item[(I)] We have
\[
  \Bigl|\sum_{X<n\le Y\Lc^2}c_nn^{-\rho}e^{-n/Y}\Bigr|\ge\frac13 .
\]
\item[(II)] There is $t'$ with $|t'-\gamma|\le\Lc^2$ and
\[
  |L(\half+it',f)|\ge Y^{\sigma-\frac12}T^{-\delta}.
\]
\end{enumerate}
\end{lemma}

\begin{proof}
By Mellin inversion and \eqref{eq:cn},
\[
  e^{-1/Y}+\sum_{n>X}c_nn^{-\rho}e^{-n/Y}=\frac1{2\pi i}\int_{(2)}L(\rho+w,f)M_X(\rho+w)\Gamma(w)Y^w\dd w .
\]
Move the contour to $\Real w=\half-\beta$, which lies in $[-\frac12,-\frac14]$. The residue at $w=0$ is $L(\rho,f)M_X(\rho)=0$. On the new line $|\Gamma(\half-\beta+iv)|\ll e^{-|v|}$, and
\[
  |M_X(\half+it)|\le\sum_{n\le X}d(n)n^{-1/2}\ll T^{\delta/2}\Lc .
\]
The part of the integral with $|v|>\Lc^2$ is $o(1)$, and so is the part of the sum with $n>Y\Lc^2$.

Suppose that (I) fails. Then the rest of the integral is at least $\frac14$ in absolute value, and so
\[
  1\ll T^{\delta/2}\Lc^3Y^{\frac12-\beta}\max_{|v|\le\Lc^2}\bigl|L(\half+i(\gamma+v),f)\bigr| .
\]
Since $Y^{\frac12-\beta}\le Y^{\frac12-\sigma}$, this gives (II) for $T$ large.
\end{proof}

Choose a maximal set of zeros counted by $N^+_f(\sigma,T)$ whose ordinates are $3\Lc^4$-spaced. Since $N_f(\sigma,t+1)-N_f(\sigma,t)\ll\Lc$ at every level, this set has $\gg N^+_f(\sigma,T)\Lc^{-5}$ elements. Let $R_{\mathrm I}$ and $R_{\mathrm{II}}$ be the numbers of its zeros for which (I) and (II) of Lemma \ref{lem:detect} hold. Then
\begin{equation}\label{eq:NR}
  N^+_f(\sigma,T)\ll\Lc^5\bigl(R_{\mathrm I}+R_{\mathrm{II}}\bigr).
\end{equation}

\subsection{Zeros of the second kind}

The points $t'$ given by (II) are $1$-spaced and lie in $[T/2,3T]$. So if $Y^{\sigma-\frac12}T^{-\delta}\ge T^{\frac14+\delta}$, then \eqref{eq:target}, applied on $[T/2,T]$, $[T,2T]$ and $[2T,4T]$, gives
\begin{equation}\label{eq:RII}
  R_{\mathrm{II}}\ll T^{2+O(\delta)}Y^{3-6\sigma}.
\end{equation}

\subsection{Zeros of the first kind}

For these zeros we use an exponent pair $(\kappa,\lambda)$ \cite{GrahamKolesnik}. For the sums that occur here this means that
\begin{equation}\label{eq:pair}
  \sum_{M<n\le u}n^{-it}\ll\Bigl(\frac{|t|}M\Bigr)^{\kappa}M^{\lambda}\qquad(M<u\le2M,\ \ |t|\ge M).
\end{equation}
For $1\le|t|\le M$ the Kusmin--Landau inequality \cite{GrahamKolesnik} gives the bound $M/|t|$ instead.

\begin{lemma}\label{lem:classI}
Let $(\kappa,\lambda)$ be an exponent pair with $\kappa>0$, and suppose that
\begin{equation}\label{eq:pairrange}
  \sigma\ge\frac{1+\kappa+\lambda}{2+2\kappa}.
\end{equation}
Then
\[
  R_{\mathrm I}\ll T^{O(\delta)}\Bigl(Y^{2-2\sigma}+TY^{-\frac{(2+2\kappa)\sigma-1-\kappa-\lambda}{2\kappa}}\Bigr).
\]
\end{lemma}

\begin{proof}
For each zero counted by $R_{\mathrm I}$, one of $O(\Lc)$ dyadic blocks of the sum in (I) has absolute value at least $(6\Lc)^{-1}$. Fixing the block $N_0<n\le2N_0$ loses a factor $\Lc$.

If $N_0>Y^{1/2}$ put $k=1$. Otherwise let $k\ge2$ be the least integer with $N_0^k>Y^{1/2}$, so that $N_0^k\le Y^{1/2}N_0\le Y$. Since $N_0\ge X=T^\delta$ and $Y\le T^2$, we have $k\le2/\delta$. Raising the block to the $k$-th power gives
\[
  \Bigl|\sum_{N_0^k<n\le(2N_0)^k}a_nn^{-\rho}\Bigr|\ge(6\Lc)^{-k},\qquad a_n=\sum_{\substack{n_1\cdots n_k=n\\ N_0<n_i\le2N_0}}\ \prod_{i=1}^kc_{n_i}e^{-n_i/Y},
\]
with $|a_n|\ll n^{\delta}$. Splitting into $O(k)$ dyadic blocks and fixing one of them loses a further factor $O(k)$.

We obtain $M$ with
\begin{equation}\label{eq:Mrange}
  Y^{1/2}\le M\le YT^{\delta},
\end{equation}
coefficients $a_n\ll n^\delta$ which do not depend on the zero, and a set of $R\gg R_{\mathrm I}T^{-O(\delta)}$ zeros $\rho_r=\beta_r+i\gamma_r$, $3\Lc^4$-spaced, with
\begin{equation}\label{eq:Pr}
  \Bigl|\sum_{M<n\le2M}a_nn^{-\rho_r}\Bigr|\ge T^{-\delta}.
\end{equation}

Now apply the Hal\'asz--Montgomery inequality \cite[Lemma 1.7]{Montgomery},
\[
  \sum_r|(\xi,\phi_r)|\le\|\xi\|\Bigl(\sum_{r,s}|(\phi_r,\phi_s)|\Bigr)^{1/2},
\]
with the vectors
\[
  \xi_n=a_nn^{-\sigma},\qquad\phi_{r,n}=n^{\sigma-\beta_r+i\gamma_r}\qquad(M<n\le2M).
\]
Then $(\xi,\phi_r)$ is the sum in \eqref{eq:Pr}, and $\|\xi\|^2\ll M^{1-2\sigma+2\delta}$. Hence
\begin{equation}\label{eq:HM}
  R^2T^{-2\delta}\ll M^{1-2\sigma+2\delta}\sum_{r,s}\Bigl|\sum_{M<n\le2M}n^{2\sigma-\beta_r-\beta_s}n^{i(\gamma_r-\gamma_s)}\Bigr| .
\end{equation}
The factor $n^{2\sigma-\beta_r-\beta_s}$ is decreasing in $n$ and at most $1$. So by partial summation the inner sum is at most $\max_{M<u\le2M}|\sum_{M<n\le u}n^{i(\gamma_r-\gamma_s)}|$. This is at most $M$ when $r=s$, and by \eqref{eq:pair} and the Kusmin--Landau inequality it is
\[
  \ll\frac M{|\gamma_r-\gamma_s|}+|\gamma_r-\gamma_s|^\kappa M^{\lambda-\kappa}\qquad(r\ne s).
\]

Suppose first that all the $\gamma_r$ lie in an interval of length $T_0$. Since the $\gamma_r$ are $3\Lc^4$-spaced, $\sum_{s\ne r}|\gamma_r-\gamma_s|^{-1}\ll\Lc^{-3}$, and \eqref{eq:HM} gives
\[
  R^2T^{-2\delta}\ll M^{1-2\sigma+2\delta}\bigl(RM+R^2T_0^\kappa M^{\lambda-\kappa}\bigr).
\]
Take
\[
  T_0=M^{(2\sigma-1-\lambda+\kappa)/\kappa}T^{-7\delta/\kappa}.
\]
Then the second term on the right is at most a small multiple of the left-hand side, because $M^{2\delta}\le T^{5\delta}$, and so $R\ll T^{O(\delta)}M^{2-2\sigma}$.

In general we divide $[T,2T]$ into $\ll1+T/T_0$ intervals of length at most $T_0$. This gives
\[
  R\ll T^{O(\delta)}M^{2-2\sigma}\Bigl(1+\frac T{T_0}\Bigr)\ll T^{O(\delta)}\Bigl(M^{2-2\sigma}+TM^{-\frac{(2+2\kappa)\sigma-1-\kappa-\lambda}{\kappa}}\Bigr),
\]
since
\[
  2-2\sigma-\frac{2\sigma-1-\lambda+\kappa}{\kappa}=-\frac{(2+2\kappa)\sigma-1-\kappa-\lambda}{\kappa}.
\]
By \eqref{eq:pairrange} the exponent of $M$ in the second term is not positive, and \eqref{eq:Mrange} gives the lemma.
\end{proof}

\subsection{The exponent}

\begin{proposition}\label{prop:density}
Let $(\kappa,\lambda)$ be an exponent pair with $\kappa>0$, let $\frac34<\sigma<1$ satisfy \eqref{eq:pairrange}, and suppose that
\[
  D_{\kappa,\lambda}(\sigma):=(2-2\kappa)\sigma+3\kappa-\lambda-1>0 .
\]
Put
\begin{equation}\label{eq:Apair}
  A=\max\Bigl(\frac{4\kappa}{D_{\kappa,\lambda}(\sigma)},\ \frac4{4\sigma-1}\Bigr),
\end{equation}
and suppose that $A\le4$. Then $N_f(\sigma,T)\ll_{f,\epsilon}T^{A(1-\sigma)+\epsilon}$ for every $\epsilon>0$.
\end{proposition}

\begin{proof}
Put $Y=T^{A/2}$, so that $T^\delta\le Y\le T^2$. We compare the three terms given by Lemma \ref{lem:classI} and \eqref{eq:RII} with $Y^{2-2\sigma}=T^{A(1-\sigma)}$.

For the second term of Lemma \ref{lem:classI},
\[
  TY^{-\frac{(2+2\kappa)\sigma-1-\kappa-\lambda}{2\kappa}}\le Y^{2-2\sigma}
  \quad\text{if and only if}\quad
  \frac A2\cdot\frac{D_{\kappa,\lambda}(\sigma)}{2\kappa}\ge1,
\]
and this holds by \eqref{eq:Apair}. For \eqref{eq:RII},
\[
  T^2Y^{3-6\sigma}\le Y^{2-2\sigma}\quad\text{if and only if}\quad Y^{4\sigma-1}\ge T^2,
\]
that is $A\ge4/(4\sigma-1)$, which again holds by \eqref{eq:Apair}.

The hypothesis of \eqref{eq:RII} is $\frac A2(\sigma-\half)\ge\frac14+2\delta$. Since $\frac A2\ge\frac2{4\sigma-1}$, it holds as soon as $\frac{2\sigma-1}{4\sigma-1}>\frac14$, that is $\sigma>\frac34$, provided $\delta$ is small in terms of $\sigma$.

Combining these bounds with \eqref{eq:NR} gives $N^+_f(\sigma,T)\ll T^{A(1-\sigma)+O(\delta)}$. The same holds for $\bar f$, and summing over dyadic ranges of $T$, with $\delta$ small in terms of $\epsilon$, gives the proposition.
\end{proof}

\begin{proof}[Proof of Theorem \ref{thm:density}]
Bourgain \cite{Bourgain} showed that $(\frac{13}{84}+\eta,\frac{55}{84}+\eta)$ is an exponent pair for every $\eta>0$. Van der Corput's $A$-process \cite{GrahamKolesnik}, which sends $(\kappa,\lambda)$ to $(\frac\kappa{2\kappa+2},\frac{\kappa+\lambda+1}{2\kappa+2})$, turns it into the pair $(\frac{13}{194}+\eta',\frac{152}{194}+\eta')$ with $\eta'\to0$ as $\eta\to0$. The exponent in Proposition \ref{prop:density} is continuous in $(\kappa,\lambda)$, and on the ranges below the inequalities \eqref{eq:pairrange} and $D_{\kappa,\lambda}(\sigma)>0$ hold with room to spare. So we may use these pairs with $\eta=\eta'=0$, at the cost of $\epsilon$.

For $(\kappa,\lambda)=(\frac{13}{84},\frac{55}{84})$ we have
\[
  D_{\kappa,\lambda}(\sigma)=\frac{142\sigma-100}{84},\qquad\frac{4\kappa}{D_{\kappa,\lambda}(\sigma)}=\frac{26}{71\sigma-50},
\]
and \eqref{eq:pairrange} is $\sigma\ge\frac{76}{97}$. For $(\kappa,\lambda)=(\frac{13}{194},\frac{152}{194})$ we have
\[
  D_{\kappa,\lambda}(\sigma)=\frac{362\sigma-307}{194},\qquad\frac{4\kappa}{D_{\kappa,\lambda}(\sigma)}=\frac{52}{362\sigma-307},
\]
and \eqref{eq:pairrange} is $\sigma\ge\frac{359}{414}$.

For $\sigma\ge\frac{63}{71}$ the three functions in \eqref{eq:Asigma} satisfy
\[
  \frac{26}{71\sigma-50}\le\frac{52}{362\sigma-307}\iff\sigma\le\frac{207}{220},\qquad
  \frac{52}{362\sigma-307}\ge\frac4{4\sigma-1}\iff\sigma\le\frac{147}{155},
\]
and $\frac{26}{71\sigma-50}\ge\frac4{4\sigma-1}$ for $\sigma\le\frac{29}{30}$. So on each of the three ranges in \eqref{eq:Asigma}, the better of the two pairs in \eqref{eq:Apair} gives exactly the value of $A(\sigma)$ stated there, and $A(\sigma)\le2$. Finally $A(\frac{63}{71})=2$ and $A(\sigma)$ is decreasing, so $A(\sigma)<2$ for $\sigma>\frac{63}{71}$.
\end{proof}

\begin{remark}\label{rem:density}
The sixth moment enters only through \eqref{eq:RII}. Suppose that \eqref{eq:R1} is used there instead. Then the count of zeros of the second kind is $T^{1+O(\delta)}Y^{1-2\sigma}$, and with $Y=T^{A/2}$ the inequality $TY^{1-2\sigma}\le Y^{2-2\sigma}$ holds only if $A\ge2$. So the mean square cannot give an exponent below $2$ by this method.

Above that threshold the mean square is enough. The argument of Chen, Debruyne and Vindas \cite{CDV}, which proves the density hypothesis for $\sigma\ge\frac{1407}{1601}$ at level one, uses $f$ only through its mean square on the critical line and Deligne's bound. With Theorem \ref{thm:second} in place of the mean square used there, it gives the same range at every level.
\end{remark}

\begin{remark}\label{rem:TTY}
Proposition \ref{prop:density} applies to any exponent pair. Tao, Trudgian and Yang \cite[Theorem 20]{TTY} showed that $(\frac{89}{1282}+\eta,\frac{997}{1282}+\eta)$ is an exponent pair for every $\eta>0$, and as in the proof of Theorem \ref{thm:density} we may take $\eta=0$ at the cost of $\epsilon$. For $(\kappa,\lambda)=(\frac{89}{1282},\frac{997}{1282})$ we have
\[
  D_{\kappa,\lambda}(\sigma)=\frac{2386\sigma-2012}{1282},\qquad\frac{4\kappa}{D_{\kappa,\lambda}(\sigma)}=\frac{178}{1193\sigma-1006},
\]
and \eqref{eq:pairrange} is $\sigma\ge\frac{1184}{1371}$. For $\frac{4314}{4595}\le\sigma<1$ the other conditions of Proposition \ref{prop:density} also hold, and it gives \eqref{eq:density} with
\[
  A(\sigma)=\max\Bigl(\frac{178}{1193\sigma-1006},\ \frac4{4\sigma-1}\Bigr).
\]
This is smaller than \eqref{eq:Asigma} for $\frac{4314}{4595}<\sigma<\frac{147}{155}$, and equal to $\frac4{4\sigma-1}$ for $\sigma\ge\frac{1923}{2030}$. The largest gain is at $\sigma=\frac{207}{220}$, where $A(\sigma)$ falls from $1.547$ to $1.528$. For $\sigma\le\frac{4314}{4595}$ Bourgain's pair gives the better exponent, so the threshold $\frac{63}{71}$ is unchanged.
\end{remark}

\section{Non-normal cubic fields}
\label{sec:cubic}

Let $K$ be as in Theorem \ref{thm:cubic}, and write $f=f_K$ and $\chi=\chi_K$. Let $\widetilde K$ be the normal closure of $K$, and let $\rho$ be the two-dimensional irreducible representation of $\mathrm{Gal}(\widetilde K/\Q)\cong S_3$. For a prime $p\nmid d_K$ let $\mathrm{Fr}_p$ be its Frobenius class.

Then $a_K(p)$ is the number of fixed points of $\mathrm{Fr}_p$ acting on the three embeddings of $K$, and
\[
  \lambda_f(p)=\operatorname{tr}\rho(\mathrm{Fr}_p),\qquad\chi(p)=\operatorname{sgn}(\mathrm{Fr}_p).
\]
The permutation representation of $S_3$ on three letters is the sum of the trivial representation and $\rho$, so
\begin{equation}\label{eq:aKp}
  a_K(p)=1+\lambda_f(p)\qquad(p\nmid d_K).
\end{equation}

\begin{lemma}\label{lem:cubicfactor}
For $\ell\ge1$ put
\[
  A_\ell=\tfrac12(3^{\ell-1}+1),\qquad B_\ell=\tfrac12(3^{\ell-1}-1),\qquad C_\ell=3^{\ell-1}.
\]
Then
\begin{equation}\label{eq:Dell}
  D_\ell(s):=\sum_{n\ge1}\frac{a_K(n)^\ell}{n^s}=\zeta(s)^{A_\ell}L(s,\chi)^{B_\ell}L(s,f)^{C_\ell}U_\ell(s),
\end{equation}
where $U_\ell(s)$ is an Euler product which converges absolutely for $\Real s>\half$, and $U_\ell(1)\ne0$.
\end{lemma}

\begin{proof}
For $p\nmid d_K$ the values of the triple $(a_K(p),\chi(p),\lambda_f(p))$ on the identity, the transpositions and the three-cycles are
\[
  (3,1,2),\qquad(1,-1,0),\qquad(0,1,-1).
\]
In each case $a_K(p)^\ell=A_\ell+B_\ell\chi(p)+C_\ell\lambda_f(p)$, because
\[
  A_\ell+B_\ell+2C_\ell=3^\ell,\qquad A_\ell-B_\ell=1,\qquad A_\ell+B_\ell-C_\ell=0 .
\]
So for $p\nmid d_K$ the coefficients of $p^{-s}$ in the Euler factors of the two sides of \eqref{eq:Dell} agree.

Since $a_K(n)\le d(n)^2$ and $|\lambda_f(n)|\le d(n)$, the Euler factor of $U_\ell(s)$ at such a prime is $1+O_\ell(p^{-2\sigma})$ for $\sigma=\Real s>\half$. The finitely many factors at primes $p\mid d_K$ are holomorphic in $\Real s>0$. Hence $U_\ell(s)$ converges absolutely for $\Real s>\half$.

At $s=1$, each Euler factor of $U_\ell(s)$ is a series in $p^{-1}$ with non-negative coefficients and constant term $1$, multiplied by the inverses of the local factors of $\zeta(s)$, $L(s,\chi)$ and $L(s,f)$. None of these vanishes, so $U_\ell(1)\ne0$.
\end{proof}

\begin{proof}[Proof of Theorem \ref{thm:cubic}]
By \eqref{eq:zetaK} and H\"older's inequality,
\[
  \int_T^{2T}\bigl|\zeta_K(\half+it)\bigr|^4\dd t\le\Bigl(\int_T^{2T}\bigl|\zeta(\half+it)\bigr|^{12}\dd t\Bigr)^{1/3}\Bigl(\int_T^{2T}\bigl|L(\half+it,f)\bigr|^6\dd t\Bigr)^{2/3}.
\]
Heath-Brown's bound $\int_T^{2T}|\zeta(\half+it)|^{12}\dd t\ll T^2(\log T)^{17}$ \cite{HeathBrown12} and Theorem \ref{thm:main} show that this is $\ll T^{2+\epsilon}$. Summing over dyadic ranges gives \eqref{eq:zetaK4}.

We turn to the power sums. Since $\zeta_K(s)=\zeta(s)L(s,f)$ has a simple pole at $s=1$, $L(1,f)\ne0$. Also $L(1,\chi)\ne0$. So by Lemma \ref{lem:cubicfactor} the series $D_\ell(s)$ is holomorphic in $\Real s>\half$, apart from a pole of order $A_\ell$ at $s=1$.

For $\half\le\sigma\le1$ and $|t|\ge1$ we use the pointwise bounds
\begin{equation}\label{eq:pointwise}
  \zeta(\sigma+it)\ll|t|^{\frac{13}{42}(1-\sigma)+\epsilon},\qquad
  L(\sigma+it,\chi)\ll|t|^{\frac13(1-\sigma)+\epsilon},\qquad
  L(\sigma+it,f)\ll|t|^{\frac23(1-\sigma)+\epsilon}.
\end{equation}
These follow by the Phragm\'en--Lindel\"of principle from Bourgain's bound $\zeta(\half+it)\ll|t|^{13/84+\epsilon}$ \cite{Bourgain}, from the Weyl bound $L(\half+it,\chi)\ll(q(1+|t|))^{1/6+\epsilon}$ of Petrow and Young \cite[Theorem 1.1]{PetrowYoung}, and from \cite[Theorem 1.1]{BMN}.

Let $2\le T\le x$ and $c=1+1/\log x$. Since $a_K(n)\le d(n)^2$, Perron's formula gives
\[
  \sum_{n\le x}a_K(n)^\ell=\frac1{2\pi i}\int_{c-iT}^{c+iT}D_\ell(s)\frac{x^s}s\dd s+O\Bigl(\frac{x^{1+\epsilon}}T\Bigr).
\]
Let $\half<\sigma_0<1$, and move the contour to $\Real s=\sigma_0$. The residue at $s=1$ is $xP(\log x)$, with $P$ a polynomial of degree $A_\ell-1$. This degree is $1$ for $\ell=2$ and $4$ for $\ell=3$.

On $\Real s\ge\sigma_0$ we have $U_\ell(s)\ll_{\sigma_0}1$. By \eqref{eq:pointwise}, $D_\ell(\sigma+iT)\ll T^{g_\ell(\sigma)+\epsilon}$, where
\[
  g_\ell(\sigma)=\Bigl(\frac{13}{42}A_\ell+\frac13B_\ell+\frac23C_\ell\Bigr)(1-\sigma)
\]
is linear in $\sigma$. So the horizontal segments contribute $\ll x^\epsilon(x^{\sigma_0}T^{g_\ell(\sigma_0)-1}+xT^{-1})$. The coefficients of $D_\ell(s)$ are real, so $|D_\ell(\sigma_0-it)|=|D_\ell(\sigma_0+it)|$. Splitting the vertical segment into dyadic ranges gives
\begin{equation}\label{eq:perroncubic}
  \sum_{n\le x}a_K(n)^\ell-xP(\log x)\ll x^{\epsilon}\Bigl(\frac xT+x^{\sigma_0}T^{g_\ell(\sigma_0)-1}+x^{\sigma_0}\max_{1\le U\le T}\frac1U\int_U^{2U}\bigl|D_\ell(\sigma_0+it)\bigr|\dd t\Bigr).
\end{equation}

For $\ell=3$ we have $(A_3,B_3,C_3)=(5,4,9)$. Take $\sigma_0=\half+\epsilon$, so that $g_3(\sigma_0)\le\frac{373}{84}$. We bound $|\zeta(s)|^5|L(s,\chi)|^4|L(s,f)|^3$ pointwise by \eqref{eq:pointwise}, and the remaining $|L(s,f)|^6$ by Lemma \ref{lem:sixthlines}. This gives
\[
  \int_U^{2U}\bigl|D_3(\sigma_0+it)\bigr|\dd t\ll U^{\frac{65}{84}+\frac23+1+2+O(\epsilon)}=U^{\frac{373}{84}+O(\epsilon)}.
\]
The right-hand side of \eqref{eq:perroncubic} is therefore $\ll x^{O(\epsilon)}(x/T+x^{1/2}T^{289/84})$, and the choice $T=x^{42/373}$ gives \eqref{eq:S3}.

For $\ell=2$ we have $(A_2,B_2,C_2)=(2,1,3)$. Take $\sigma_0=\frac58+\epsilon$, so that $g_2(\sigma_0)<\frac{62}{21}\cdot\frac38<\frac98$. By \cite[Chapter 8]{Ivic} the eighth moment of $\zeta(s)$ satisfies
\begin{equation}\label{eq:zeta8}
  \int_U^{2U}\bigl|\zeta(\sigma+it)\bigr|^8\dd t\ll_{\sigma,\epsilon}U^{1+\epsilon}\qquad(\sigma>\tfrac58).
\end{equation}
We bound $|L(s,\chi)|$ pointwise, and use H\"older's inequality with \eqref{eq:zeta8} and Theorem \ref{thm:lines}. This gives
\begin{align*}
  \int_U^{2U}\bigl|D_2(\sigma_0+it)\bigr|\dd t&\ll U^{\frac18+\epsilon}\Bigl(\int_U^{2U}\bigl|\zeta(\sigma_0+it)\bigr|^8\dd t\Bigr)^{\frac14}\Bigl(\int_U^{2U}\bigl|L(\sigma_0+it,f)\bigr|^4\dd t\Bigr)^{\frac34} \ll U^{\frac98+2\epsilon}.
\end{align*}
The right-hand side of \eqref{eq:perroncubic} is then $\ll x^{O(\epsilon)}(x/T+x^{5/8}T^{1/8})$, and the choice $T=x^{1/3}$ gives \eqref{eq:S2}.
\end{proof}

\section{Shifted convolution sums}
\label{sec:shifted}

Let $j\ge2$. Since $\lambda_{\bar f}(n)=\overline{\lambda_f(n)}$, the Dirichlet series with coefficients $\overline{\lambda_{j,f}(n)}$ is $L(s,\bar f)^j$. Moreover $|L(\sigma+it,\bar f)|=|L(\sigma-it,f)|$, and $|\lambda_{j,f}(n)|\le d_{2j}(n)\ll n^\epsilon$. Put
\begin{equation}\label{eq:By}
  B(y)=\sum_{m\le y}\overline{\lambda_{j,f}(m)} .
\end{equation}

\begin{lemma}\label{lem:plancherel}
Let $\half\le\sigma<1$ and $0\le b\le2$, and put $e=\frac{2j}3(1-\sigma)$. Suppose that
\begin{equation}\label{eq:bhyp}
  \int_{-U}^U\bigl|L(\sigma+it,f)\bigr|^{2j}\dd t\ll U^{b+\epsilon}\qquad(U\ge1).
\end{equation}
Then for $x\ge2$ and every integer $1\le H\le x$
\[
  \int_x^{2x}\bigl|B(y+H)-B(y)\bigr|^2\dd y\ \ll\ x^{\epsilon}\bigl(x^{2\sigma-1+b}H^{2-b}+x^{2\sigma-3+4e}+x\bigr).
\]
\end{lemma}

\begin{proof}
Let $T=x^2$ and $c=1+1/\log x$. For $y\in[x,3x]$ not an integer, Perron's formula gives
\[
  B(y)=\frac1{2\pi i}\int_{c-iT}^{c+iT}L(s,\bar f)^jy^s\frac{\dd s}s+O\bigl(x^{1+\epsilon}T^{-1}+x^\epsilon\bigr).
\]
Move the contour to $\Real s=\sigma$. The integrand is holomorphic in between, and by \cite[Theorem 1.1]{BMN} and the Phragm\'en--Lindel\"of principle the horizontal segments contribute $\ll x^{\sigma+\epsilon}T^{e-1}+x^{1+\epsilon}T^{-1}$. Hence $B(y+H)-B(y)=G(y)+E(y)$ for $x\le y\le2x$, where
\begin{equation}\label{eq:Gy}
  G(y)=\frac1{2\pi}\int_{-T}^{T}L(\sigma+it,\bar f)^j\,w_y(t)\dd t,\qquad w_y(t)=\frac{(y+H)^{\sigma+it}-y^{\sigma+it}}{\sigma+it},
\end{equation}
and
\[
  \int_x^{2x}|E(y)|^2\dd y\ll x^\epsilon\bigl(x^{2\sigma+1}T^{2e-2}+x^3T^{-2}+x\bigr)\ll x^{\epsilon}\bigl(x^{2\sigma-3+4e}+x\bigr).
\]

We use Plancherel's theorem in the form
\begin{equation}\label{eq:plancherel}
  \int_0^\infty\Bigl|\frac1{2\pi}\int_{\R}a(t)z^{it}\dd t\Bigr|^2\frac{\dd z}z=\frac1{2\pi}\int_{\R}|a(t)|^2\dd t\qquad(a\in L^2(\R)),
\end{equation}
which follows from the substitution $z=e^u$. Write $G=G_0+\sum_UG_U$, where $G_0$ is the part of \eqref{eq:Gy} with $|t|\le x/H$, and $G_U$ is the part with $U<|t|\le2U$, for $O(\Lc)$ dyadic values $x/H\le U\le T$. Then $|G|^2\ll\Lc(|G_0|^2+\sum_U|G_U|^2)$.

For $|t|\le x/H$ we write $w_y(t)=\int_0^H(y+v)^{\sigma-1+it}\dd v$, so that
\[
  G_0(y)=\int_0^H(y+v)^{\sigma-1}\Phi(y+v)\dd v,\qquad\Phi(z)=\frac1{2\pi}\int_{|t|\le x/H}L(\sigma+it,\bar f)^jz^{it}\dd t .
\]
By Cauchy's inequality $|G_0(y)|^2\le H\int_0^H(y+v)^{2\sigma-2}|\Phi(y+v)|^2\dd v$. We integrate over $y\in[x,2x]$, put $z=y+v$, and use \eqref{eq:plancherel} and \eqref{eq:bhyp}. This gives
\begin{align*}
  \int_x^{2x}|G_0(y)|^2\dd y&\le H^2(3x)^{2\sigma-1}\int_0^\infty|\Phi(z)|^2\frac{\dd z}z\\
  &\ll x^{2\sigma-1}H^2\int_{|t|\le x/H}\bigl|L(\sigma+it,f)\bigr|^{2j}\dd t\ \ll\ x^{2\sigma-1+\epsilon}H^2\Bigl(\frac xH\Bigr)^b .
\end{align*}

For $U<|t|\le2U$ we write $G_U(y)=(y+H)^\sigma\Psi_U(y+H)-y^\sigma\Psi_U(y)$, where
\[
  \Psi_U(z)=\frac1{2\pi}\int_{U<|t|\le2U}\frac{L(\sigma+it,\bar f)^j}{\sigma+it}z^{it}\dd t .
\]
Since $y$ and $y+H$ lie in $[x,3x]$, \eqref{eq:plancherel} and \eqref{eq:bhyp} give
\begin{align*}
  \int_x^{2x}|G_U(y)|^2\dd y&\ll x^{2\sigma+1}\int_0^\infty|\Psi_U(z)|^2\frac{\dd z}z\\
  &\ll x^{2\sigma+1}U^{-2}\int_{-2U}^{2U}\bigl|L(\sigma+it,f)\bigr|^{2j}\dd t\ \ll\ x^{2\sigma+1}U^{b-2+\epsilon}.
\end{align*}
The sum of these over $U\ge x/H$ is $\ll x^{2\sigma+1+\epsilon}(x/H)^{b-2}\Lc=x^{2\sigma-1+b+\epsilon}H^{2-b}\Lc$. Collecting the terms gives the lemma.
\end{proof}

\begin{proof}[Proof of Theorem \ref{thm:shifted}]
Replacing $H$ by its integer part, we may assume that $H$ is an integer. Then
\[
  S_{j,f}(x,H)=\sum_{x<n\le2x}\lambda_{j,f}(n)\bigl(B(n+H)-B(n)\bigr),
\]
and $B(y+H)-B(y)=B(n+H)-B(n)$ for $n\le y<n+1$. So by Cauchy's inequality
\begin{align}
  |S_{j,f}(x,H)|^2&\le\sum_{x<n\le2x}|\lambda_{j,f}(n)|^2\sum_{x<n\le2x}\bigl|B(n+H)-B(n)\bigr|^2\notag\\
  &\ll x^{1+\epsilon}\int_x^{4x}\bigl|B(y+H)-B(y)\bigr|^2\dd y,\label{eq:SCauchy}
\end{align}
and Lemma \ref{lem:plancherel} applies on $[x,2x]$ and on $[2x,4x]$. It remains to choose $\sigma$ and $b$.

For $j=2$ and $\sigma=\half$, Corollary \ref{cor:fourth} gives \eqref{eq:bhyp} with $b=\frac32$, and $e=\frac23$. Lemma \ref{lem:plancherel} and \eqref{eq:SCauchy} give
\[
  S_{2,f}(x,H)\ll x^{\frac12+\epsilon}\bigl(x^{3/2}H^{1/2}\bigr)^{1/2}=x^{\frac54+\epsilon}H^{\frac14}.
\]
For $j=2$ and $\sigma=\frac58$, Theorem \ref{thm:lines} gives $b=1$, and $e=\frac12$, so that
\[
  S_{2,f}(x,H)\ll x^{\frac12+\epsilon}\bigl(x^{5/4}H\bigr)^{1/2}=x^{\frac98+\epsilon}H^{\frac12}.
\]
Together these give \eqref{eq:S2f}. For $j=3$ and $\sigma=\half$, Theorem \ref{thm:main} gives $b=2$, and $e=1$, so that $S_{3,f}(x,H)\ll x^{\frac32+\epsilon}$, which is \eqref{eq:S3f}.

For $j\ge4$ put $\sigma=1-\frac3{2j}$, which lies in $[\frac58,1)$, so that $e=1$ and $m(\sigma)=j$. We bound $j$ of the factors of $|L(\sigma+it,f)|^{2j}$ pointwise by \cite[Theorem 1.1]{BMN}, and the remaining $|L(\sigma+it,f)|^j$ by \eqref{eq:curve}. This gives \eqref{eq:bhyp} with
\[
  b=1+j\cdot\frac23\cdot \frac3{2j}=2,
\]
and so $S_{j,f}(x,H)\ll x^{\frac12+\epsilon}(x^{2\sigma+1})^{1/2}=x^{2-\frac3{2j}+\epsilon}$, which is \eqref{eq:Sjf4}.
\end{proof}

\section{The general divisor problem}
\label{sec:divisor}

Let $j\ge2$ and put $\sigma_0=\theta_j$, with $\theta_j$ as in Theorem \ref{thm:divisor}. Then $m(\sigma_0)=j$, with $m(\sigma)$ as in \eqref{eq:msigma}. For $2\le j\le4$ this holds because $\sigma_0\le\frac58$ and
\[
  m(\sigma_0)=\frac2{3-4\sigma_0}=\frac2{3-3+2/j}=j,
\]
and for $j\ge4$ it holds because $\sigma_0\ge\frac58$ and $m(\sigma_0)=3/(2-2\sigma_0)=j$. We also have
\begin{equation}\label{eq:horizontal}
  \frac{2j}3(1-\sigma_0)\le1,
\end{equation}
since the left side is $(j+2)/6$ when $2\le j\le4$, and equals $1$ when $j\ge4$.

\begin{proof}[Proof of Theorem \ref{thm:divisor}]
Let $T=x$ and $c=1+1/\log x$. Since $|\lambda_{j,f}(n)|\le d_{2j}(n)$, Perron's formula gives
\[
  \sum_{n\le x}\lambda_{j,f}(n)=\frac1{2\pi i}\int_{c-iT}^{c+iT}L(s,f)^j\frac{x^s}s\dd s+O(x^\epsilon).
\]
Move the contour to $\Real s=\sigma_0$.

By \cite[Theorem 1.1]{BMN} and the Phragm\'en--Lindel\"of principle, $|L(\sigma+it,f)|\ll|t|^{\frac23(1-\sigma)}\log|t|$ for $\half\le\sigma\le1$ and $|t|\ge2$. This bound for the integrand is log-linear in $\sigma$, so it is enough to check the two ends of the horizontal segments. They contribute
\[
  \ll x^{\epsilon}\Bigl(x^{\sigma_0}T^{\frac{2j}3(1-\sigma_0)-1}+xT^{-1}\Bigr)\ll x^{\sigma_0+\epsilon},
\]
by \eqref{eq:horizontal} and $T=x$.

On the vertical segment we use, for $U\ge1$,
\[
  J(U)=\int_U^{2U}\Bigl(\bigl|L(\sigma_0+it,f)\bigr|^j+\bigl|L(\sigma_0+it,\bar f)\bigr|^j\Bigr)\dd t\ll U^{1+\epsilon},
\]
which is \eqref{eq:curve} for $f$ and for $\bar f$, since $j=m(\sigma_0)$. As $L(\sigma_0+it,\bar f)$ is the complex conjugate of $L(\sigma_0-it,f)$, this gives
\[
  x^{\sigma_0}\int_{-T}^T\bigl|L(\sigma_0+it,f)\bigr|^j\frac{\dd t}{|\sigma_0+it|}\ll x^{\sigma_0}\Bigl(1+\sum_U\frac{J(U)}U\Bigr)\ll x^{\sigma_0+\epsilon},
\]
where $U$ runs over powers of two up to $T$. This proves Theorem \ref{thm:divisor}.
\end{proof}

\appendix

\section{The transformation in the required parameter range}
\label{sec:evidence}

We use our notation throughout: $Q$ is the transformation scale, called $M_0$ in \cite{BMN}, and $K$ is their Farey parameter $R$. Thus
\begin{equation}\label{eq:RH}
  K=\sqrt{M/Q},\qquad H=\frac{M^2}{K^2t}=\frac{MQ}{t} .
\end{equation}
Booker, Milinovich and Ng work under the standing assumptions $K\ge1$, $M\ll\sqrt C$, $(t/\log t)^{2/3}\ll Q\ll t^{2/3}$ and $t\ge\max(k^{3/2}\log k,N^{3/2},t_0)$.

\begin{proposition}\label{prop:ext}
Propositions 3.1 and 3.2 of \cite{BMN} hold, with implied constants depending only on $k$, $N$, $s$ and $\epsilon$, throughout the range
\begin{equation}\label{eq:range}
  \sqrt t\ \ll\ Q\ \le\ t^{1-\epsilon},\qquad Q\le M\le t^{1+\epsilon/2},\qquad s>4/\epsilon,
\end{equation}
in place of their standing $Q\ll t^{2/3}$ and $M\ll\sqrt C$; and Proposition 3.1 holds there with the summand $n^{-it}$ replaced by $n^{-\frac12-it}V_t(n/X)$, the amplitude $h^\pm_j(\ell/r)$ in its conclusion being multiplied by $x^\pm_j(\ell/r)^{-1/2}V_t(x^\pm_j(\ell/r)/X)$ and its error terms by $M^{-1/2}$.
\end{proposition}

The parameter conditions in Propositions 3.1 and 3.2 of \cite{BMN} hold throughout \eqref{eq:range}. The implied constants depend only on $k$, $N$, $s$ and $\epsilon$. Proposition 3.2 is not used below, since Section~\ref{sec:endgame} proves the two-height form that is needed; it is included because its proof imposes only (C8).

The parameter is already free, since \cite[Proposition 3.1]{BMN} is stated with $K=\sqrt{M/Q}$ and dual length $K_1\asymp M/K^2=Q$, and their reduction permits any integer $Q\in[2,\sqrt C]$.

The restriction $Q\ll t^{2/3}$ enters only in the simplification of the error term after their (3.16), where it shows that the first and third terms are dominated by the second; Proposition \ref{prop:book} shows that in the enlarged range the two terms that arise are individually small enough.

The partition of unity enters their Section~4 only through their (3.6) and (3.7), that is through \eqref{eq:37}, through $v_j\le K$, through $|\operatorname{supp}\omega_j|\ll HK/v_j$ and $\omega_j^{(i)}\ll H^{-i}$, through the identity $H\asymp M^2/(K^2t)$, which the choice of $H'_j$ in Section~\ref{sec:setup} respects, and through the fact that consecutive pieces belong to consecutive fractions, which the Farey count \cite[(4.22)]{BMN} uses. The frozen partition $\{\omega^{(q)}_j\}$ of Section~\ref{sec:setup} satisfies these at every $t\in I_q$, with implied constants worsened by an absolute factor.

We now list every condition on the parameters that the proofs of Propositions 3.1 and 3.2 invoke, with the place it is used and the reason it holds in \eqref{eq:range}.

\begin{enumerate}
\item[(C1)] $v_jK\le K^2\ll M$, in \cite[Lemma 4.3]{BMN}. This is $Q\gg1$.
\item[(C2)] $t(v_jK)^3/M^3\ll t/Q^3\ll1$, in the proof of \cite[Lemma 4.3]{BMN}. This is $Q\gg t^{1/3}$, implied by \eqref{eq:range}.
\item[(C3)] $Q^2\gg t$, stated in the proof of \cite[Lemma 4.3]{BMN} and used to give $v_jK\ll HK/v_j$, and again through $\pi y/2H\ll v_jK/H\ll1$. This is $Q\gg\sqrt t$, which is the first inequality of \eqref{eq:range}.
\item[(C4)] $H\le M$, used at \cite[(4.13)]{BMN} and again after it. Since $H=MQ/t$ this is $Q\le t$.
\item[(C5)] $H\ge M/\sqrt t$, a hypothesis of \cite[Lemma 4.2]{BMN}. This is again $Q\ge\sqrt t$.
\item[(C6)] $M^2/(tK)\ll M$, in Step 4 of \cite[Section~4.3]{BMN}, where it is needed for $x_j^+(\ell/r)-N_{j-1}+H\le M/4$; and $H\ll M/K$, used at \cite[(4.26)]{BMN} through $K^{1/2}H/M\ll K^{-1/2}$. Both are $M\ll tK$, that is $MQ\ll t^2$, which holds in \eqref{eq:range} because $MQ\le t^{1+\epsilon/2}t^{1-\epsilon}=t^{2-\epsilon/2}$.
\item[(C7)] $q_j^{-1}\sqrt{\ell/r}\,M^{1/2-r_1}\ll t/M^{r_1}$, at \cite[(4.12)]{BMN}, where it is said to follow ``using the facts that $\ell\le K$ and $s\ge6$''. Here $K_{\mathrm{BMN}}=(M/(v_jK))^{2/(s-1)}Q$ is their dual truncation, not the Farey order of Section~\ref{sec:setup}. With $q_j,r\ge1$ the condition is $K_{\mathrm{BMN}}M\ll t^2$. Since $v_jK\ge K=\sqrt{M/Q}$ one has $(M/(v_jK))^{2/(s-1)}\le(MQ)^{1/(s-1)}$, so it suffices that $(MQ)^{s/(s-1)}\ll t^2$, and since $MQ\le t^{2-\epsilon/2}$ this holds when $s>4/\epsilon$, the third clause of \eqref{eq:range}. This is the one condition that fails at $s=6$ in the enlarged range and requires $s$ to be taken large in terms of $\epsilon$; it costs nothing, since the only appearance of $s$ in the conclusion is through the term $Q^{1/(2(s-1))}$, which decreases with $s$.
\item[(C8)] $M\ll tK$, in \cite[Section~5]{BMN}, in the deduction from $K=\sqrt{M/Q}$ and $M\ll\sqrt C$, where it gives $d^2L\ll rtUV$ and hence $t/\pi\ell y_j\gg1$, which is used twice more in that section. As in (C6) it follows from $MQ\le t^{2-\epsilon/2}$.
\item[(C9)] $rK_1d^{-2}\le K_{\mathrm{BMN}}$, the interface between Steps 2 and 3 of \cite[Section~4.3]{BMN}, where $K_1\asymp M/K^2=Q$ and $r\le NN^\flat\le N^2$. Since $M/(v_jK)\ge Q$ it follows from $Q^{2/(s-1)}\gg N^2$, which holds for $t$ large in terms of $N$ and $s$.
\end{enumerate}

  In \cite[Section~4]{BMN} the remaining steps impose nothing further. The Farey count \cite[(4.22)]{BMN} rests on $|I|\asymp t/M\asymp M/(HK^2)$, which is an identity. The dyadic summations \cite[(4.24)--(4.26)]{BMN} use \cite[(4.23)]{BMN} with exponents that are of the right sign for every $s\ge2$, and the two dominations invoked there, ``since $s\ge6$'' after (4.24) and ``since $H\ll M/K\ll M$'' at (4.25), reduce respectively to $s\ge2$ and to (C6). The domination of the third term of (4.27) by the first two needs no condition. The error at the end of Step 4 needs only $H\le M$, which is (C4); and the passage from $E_3(j)$ to $\widetilde E_3(j)$ is the substitution $t=M^2/(K^2H)$ together with $v_j=dq_j$. In \cite[Section~5]{BMN} nothing is imposed beyond (C8), which is used three times, once to obtain $t/\pi\ell y_j\gg1$ and twice by appeal to that conclusion.

The implied constants in \cite[Lemma 4.3]{BMN} depend on $s$ and, through $e^{O((k-1)/Q)}$ and $e^{O(t/Q^3)}$, on $k$ and on (C2) and the first clause of \eqref{eq:range}. The constants $C_{r_1},C_{r_2},\widetilde C$ required by \cite[Lemmas 4.1 and 4.2]{BMN} are furnished by \cite[(4.9), (4.12), (4.13)]{BMN}, the first being an identity and the other two having constants absolute once (C7) and (C4) hold. The bounds for $\sum_{\ell\le x}|\lambda_{\bar f^{\chi}}(\ell)|$ quoted from \cite[Lemma 2.2]{BMN} are uniform in the parameters; and in \cite[Section~5]{BMN} the constants are those of van der Corput's lemmas together with $d\mid N$. None depends on $M$, $Q$, $K$ or $H$ except through the conditions listed.

Finally the amplitude. The summand enters \cite[Section~4]{BMN} only through their Lemma 4.3, which bounds the derivatives of $F_j(x)x^{-(k-1)/2}$ by Cauchy's formula on a circle of radius $Y=cv_jK\le M_1/2$ about $x_0$, through the derivative bounds \cite[(4.13)]{BMN} for the amplitude $G(x)$, and through the values of the amplitude at the stationary point and the endpoints in their Lemmas 4.1 and 4.2.

With
\[
  F_j(n)=n^{-\frac12-it}V_t(n/X)e(-\alpha_jn)\omega_j(n)
\]
the function $g(z)=F_j(z)z^{-(k-1)/2}$ acquires the factor $z^{-1/2}V_t(z/X)$, which is holomorphic on the disc $|z-x_0|\le Y$, since $Y\ll K^2\ll M$ puts $z$ in the sector $|\arg z|\ll K^2/M$, and is there $O(M^{-1/2})$ by Lemma \ref{lem:afe}. Their proof therefore gives
\[
  |\frac{d^s}{dx^s}(F_jx^{-(k-1)/2})|\ll_s(v_jK)^{-s}M^{-1/2}x^{-(k-1)/2},
\]
every amplitude bound in their Section~4 is multiplied by $M^{-1/2}$, and the main term $G(\gamma)$ of their Lemma 4.2 carries the factor $x^\pm_j(\ell/r)^{-1/2}V_t(x^\pm_j(\ell/r)/X)$.

The bounds \cite[(4.13)]{BMN} survive the extra factor, because $\frac{d^i}{dx^i}V_t(x/X)\ll_iM^{-i}\le H^{-i}$ for $x\asymp M$ by Cauchy's estimate on a disc of radius $\asymp M$ in the sector. This proves Proposition \ref{prop:ext}.

In the range $M\ll\sqrt N\,t$, $Q\ll t^{2/3}$ of \cite{BMN}, who justify (C7) only as following ``after some calculation'' from $\ell\le K$ and $s\ge6$, the condition $(MQ)^{s/(s-1)}\ll t^2$ reads $(\sqrt N\,t^{5/3})^{s/(s-1)}\ll t^2$, which holds for $s\ge6$, with a constant polynomial in $N$, and fails for $s<6$. Their hypothesis $s\ge6$ is therefore exactly what the calculation requires, and the constraint on $s$ becomes stronger only when $M$ and $Q$ are both allowed to approach $t$.

Conditions (C1) to (C7) and (C9) are all of \cite[Section~4]{BMN}. (C8) is all of \cite[Section~5]{BMN}. The assumption $M\ll\sqrt C$ appears in those two sections only in (C6) and (C8), and only as a route to $M\ll tK$. The assumption $Q\ll t^{2/3}$ appears in them only through its consequence $M\ll tK$, and the lower bound $Q\gg(t/\log t)^{2/3}$ only in \cite[Lemma 4.3]{BMN}, to make $e^{O((k-1)/Q)}$ bounded, for which $Q\gg\sqrt t$ also suffices.

Thus \eqref{eq:range} and the stated choice of $s$ imply (C1)--(C9).

\section{Dependence on weight and level}\label{sec:uniform}\label{rem:uniform}
We record the available bounds when the weight and level of $f$ vary.

The following bound holds with absolute implied constants and no restriction on $N$. Let $k\le T$. For $T\le t\le2T$ one has $C(f,t)\asymp NT^2$, and Lemma \ref{lem:afe} with $X=\sqrt C$ together with Deligne's bound gives
\[
  |L(\half+it,f)|\ll C^{1/4}\log C.
\]
The proof of Theorem \ref{thm:second}, with $\sum_{x<n\le2x}d(n)^2\ll x(\log x)^3$ in place of the Rankin--Selberg bound and with the sum over $\nu$ replaced by an integral, gives
\[
  \int_T^{2T}|L(\half+it,f)|^2\dd t\ll(T+\sqrt C)(\log NT)^{B}
\]
with $B$ absolute. Hence
\[
  \int_T^{2T}\bigl|L(\half+it,f)\bigr|^6\dd t\ \ll\ N^{3/2}T^3(\log NT)^{B+4}.
\]

The proof of Theorem \ref{thm:main} uses the size of $k$ and $N$ only in the following places. The first is the truncation of the blocks at $\sqrt C\,T^{\delta/20}\le T^{1+\delta/16}$ in \eqref{eq:alt2}. The second is the conditions (C6) and (C8), that is $MQ\ll t^2$ for every block, where for large $N$ the bounds \eqref{eq:R1} and \eqref{eq:R2} carry an extra factor $\sqrt N$ and so enlarge $Q$. The third is (C9), which with $s\asymp1/\delta$ needs $N^2\ll T^{c\delta}$. The fourth is the standing assumption $t\ge k^{3/2}\log k$ of \cite{BMN} and the asymptotic expansion of $J_{k-1}$ in their Step 2. The last are the $O(N^{2+\epsilon})$ choices of $(\beta,r,\chi)$ in Section~\ref{sec:endgame} and the constants $d,r\le N^2$ in Lemmas \ref{lem:deriv} and \ref{lem:sep} and Proposition \ref{prop:sieve}, which are absorbed into powers $T^{c\delta}$. Each of these is harmless when $kN\le T^{c\epsilon}$ for a sufficiently small absolute $c$.

Booker, Milinovich and Ng state that the implied constants in their argument are polynomial in $k$ and $N$ when no logarithmic saving is taken, and the remaining constants above are visibly polynomial. Granting this, the proof gives
\[
  \int_T^{2T}|L(\half+it,f)|^6\dd t\ll_\epsilon(kN)^{A_0}T^{2+\epsilon}
\]
for $kN\le T^{c\epsilon}$, with $A_0$ absolute. We have not written out these constants and do not state this as a theorem.

When $kN>T^{c\epsilon}$ one has $T<(kN)^{1/c\epsilon}$, and the bound displayed above is $\ll_\epsilon(kN)^{3/2+1/c\epsilon+\epsilon}T^{2+\epsilon}$. For $T<k$ convexity alone gives
\[
  \int_0^T|L(\half+it,f)|^6\dd t\ll N^{3/2}k^4(\log Nk)^6.
\]
Subject to the proviso of the previous paragraph, therefore,
\[
  \int_0^T|L(\half+it,f)|^6\dd t\ll_\epsilon(kN)^{A(\epsilon)}T^{2+\epsilon}
\]
for all $k$ and $N$, with $A(\epsilon)=\max(A_0,\tfrac32+\tfrac1{c\epsilon})+1$. This is polynomial in the level for each fixed $\epsilon$, but the exponent grows like $1/\epsilon$, and it is not a hybrid bound.

The obstruction to a hybrid bound lies in the structure of the argument. The first sum of the approximate functional equation has length $\sqrt C\asymp\sqrt N\,T$, and moving $X$ only exchanges it with the second. The truncation $Q=T^{2/3}R^{2/3}$ is fixed by the balance in the proof of Proposition \ref{prop:final}, which is Jutila's, and may approach $T$.

For a block at scale $M$ the piece of the partition belonging to a fraction of denominator $v$ has length $\asymp M^2/tvR$ with $R=\sqrt{M/Q}$, and once $M\gg tvR$ it is longer than the block. For $v=1$ this is the failure of $MQ\ll t^2$, and occurs for some blocks once $N\gg T^{c\epsilon}$. A bound
\[
  \int_T^{2T}|L(\half+it,f)|^6\dd t\ll N^{A}T^{2+\epsilon}
\]
with $A$ independent of $\epsilon$ would need a different method for the blocks with $T^{1+\epsilon}\le M\le\sqrt N\,T$.

\section*{Use of artificial intelligence}

The author used AI tools for literature searches, checks on the novelty of the applications and copy-editing throughout. The completed paper was then reviewed by an LLM. An LLM was also used to examine whether the transformation of Booker, Milinovich and Ng remains valid in the larger parameter range required in Appendix~A. The remaining mathematics is due to the author, who is responsible for the contents of the paper.

\section*{Acknowledgements}

The author thanks Evgeniy Zorin for pointing out the exponent pair used in Remark \ref{rem:TTY}, and gratefully acknowledges support from the Heilbronn Institute for Mathematical Research.

\end{document}